\documentclass[11pt]{article}

\usepackage{amssymb,latexsym,times}
\usepackage{amsthm}

\usepackage{hyperref} 
\usepackage{mathptmx}
\usepackage{amscd}
\usepackage{color}
\usepackage[all]{xy}

\usepackage{graphicx} 
\usepackage{tikz-cd}

\newtheorem{Thm}{Theorem}[section]
\newtheorem{Lem}[Thm]{Lemma}
\newtheorem{Cor}[Thm]{Corollary}
\newtheorem{Prop}[Thm]{Proposition}

\newtheorem{example}[Thm]{Example}
\newtheorem{remark}[Thm]{Remark}
\newtheorem{Def}[Thm]{Definition}

\newcommand{\A}{\mathbb{A}}

\newcommand{\Z}{\mathbb{Z}}

\newcommand{\Pro}{\mathbb{P}}
\newcommand{\N}{\mathbb{N}}
\newcommand{\C}{\mathbb{C}}
\newcommand{\R}{\mathcal{R}}

\newcommand{\Hom}{\operatorname{Hom}}

\usepackage{xcolor}
\usepackage{version}

\excludeversion{NB}

\newcommand{\Lam}{\Lambda}

\newcommand{\bsm}{\begin{smallmatrix}}
\newcommand{\esm}{\end{smallmatrix}}

\def\cqfd{\hfill $\Box$ \medskip}
\def\CC{{\mathcal C}}
\def\resp{{\em resp.\ }}

\def\<{\langle\,}
\def\>{\,\rangle}

\def\P{{\mathbb P}}

\def\B{{\mathcal B}}

\def\g{\mathfrak g}

\def\<{\langle}
\def\>{\rangle}

\def\lra{\longrightarrow}
\def\ra{\rightarrow}

\def\1{\mathbf 1}

\def\ii{{\mathbf i}}

\def\P{{\mathcal P}}

\def\a{\alpha}

\def\mod{{\rm mod}\,}

\def\Hom{{\rm Hom}}

\def\De{\Delta}

\def\la{\lambda}
\def\AA{\mathcal{A}}

\def\O{\mathcal{O}}
\def\1{{\mathbf 1}}

\def\hg{\widehat{\mathfrak{g}}}

\def\G{\Gamma}

\def\bb{\mathbf b}
\def\bc{\mathbf c}

\def\bi{\mathbf{i}}

\def\ba{\mathbf{a}}

\def\Si{\Sigma}

\def\mod{\mathrm{mod}}

\def\ldeg{\mathrm{ldeg}}

\def\bQ{\underline{Q}}

\def\Spec{\mathrm{Spec}}
\def\Schk{\mathrm{Sch/_K}}
\def\Kalg{\mathrm{K \hspace{-0.1cm} - \hspace{-0.1cm}Alg}}
\def\Set{\mathrm{Set}}

\newcommand{\NC}{\mathcal{N}}
\newcommand{\oc}{\overline{ c}}

\newcommand{\trdeg}{\mathrm{trdeg}}

\newcommand{\red}{\color{red}}
\newcommand{\blue}{\color{blue}}
\newcommand{\green}{\color{green}}

\newcommand{\BB}{\mathcal{B}}

\newcommand{\VV}{\mathcal{V}}
\newcommand{\WW}{\mathcal{W}}
\newcommand{\wt}{\widetilde}
\newcommand{\Ucc}{U(c^{-1})\overline{c}}

\def\GG{\mathcal{G}}

\begin{document}

\title{\bf Cluster structures on schemes of bands}
\author{L. Francone and B. Leclerc}

\date{}

\maketitle

\begin{abstract}
We introduce new geometric objects, called \emph{$(G,c)$-bands}, associated with a simple simply-con\-nec\-ted algebraic group $G$, and a Coxeter element $c$ in its Weyl group. We show that bands of a given type are the rational points of an infinite dimensional affine scheme, whose ring of regular functions has a cluster algebra structure. We also show that two important invariant subalgebras of this ring are cluster subalgebras. These three cluster structures have already appeared as Grothendieck rings of certain categories of representations of quantum affine algebras, their Borel subalgebras, and shifted quantum affine algebras. Schemes of bands provide a common geometric setting in which these Grothendieck rings can be studied and related to each other. 
\end{abstract}

\setcounter{tocdepth}{3}
{\small \tableofcontents}

\section{Introduction}

Cluster algebras have become ubiquitous combinatorial structures. They have in particular deep connections with algebraic groups and representation theory. A prototypical example is the cluster algebra structure on the ring 
$\C[U]$ of regular functions on a maximal unipotent subgroup $U$ of a simple complex algebraic group $G$ of type 
$A$, $D$, $E$. It was conjectured by Fomin and Zelevinsky and proved by Kang, Kashiwara, Kim and Oh \cite{KKKO} and Qin \cite{Q} that the set of cluster monomials of $\C[U]$ is contained in Lusztig's dual canonical basis. 
On the other hand it was shown in \cite{HL0} that $\C[U]$ is isomorphic to the complexified Grothendieck ring 
$\C\otimes K_0(\CC_Q)$ 
of a monoidal subcategory $\CC_Q$ of the category of finite-dimensional representations of a quantum affine algebra of the same Dynkin type as $G$, and that this isomorphism maps the dual canonical basis of $\C[U]$ to the basis of classes of simple modules of $\CC_Q$.

\smallskip
In \cite{GHL} cluster structures on the Grothendieck rings of certain categories $\O^{\rm shift}_\Z$ of 
representations of \emph{shifted} quantum affine algebras of type $A$, $D$, $E$ were discovered, and it was conjectured that cluster monomials are classes of simple objects of $\O^{\rm shift}_\Z$. These cluster algebras $\AA$ are of infinite rank, that is, each cluster consists of infinitely many cluster variables. It was observed in \cite{GHL} that the algebras $\C\otimes \AA$ contain infinitely many copies of the ring $\C[G]$ of regular functions on the simple and simply-connected algebraic group $G$ of the same type $A$, $D$, $E$.
The following question then arises : is there a natural algebro-geometric object whose ring of regular functions is isomorphic to $\C\otimes \AA$ ? In this paper we construct a family of such objects, parametrized by Coxeter elements $c$ of the Weyl group $W$ of $G$, which we call spaces of $(G,c)$-bands. Since $\AA$ has infinite rank these spaces have to be infinite-dimensional.

\smallskip
The definition of a $(G,c)$-band was motivated by \cite[\S10]{GHL}, which was in turn inspired by the  theory of $q$-opers and the notion of $q$-Wronskian developed in a series of papers by Frenkel, Koroteev, Sage and Zeitlin
\cite{KSZ,FKSZ,KZ}. 

Let $B$ be a Borel subgroup of $G$ containing the maximal torus $T$, and let $U$ be its unipotent radical.
Let $B^-$ be the opposite Borel subgroup of $B$ with respect to $T$, and $U^-$ its unipotent radical.
Let $U(c^{-1}) := U \cap (c\,U^-c^{-1})$, and let $\bar{c}$ denote a representative of $c$ in the normalizer of~$T$. 

\begin{Def}\label{def-band}
A $(G,c)$-band over $\C$ is a sequence $b=(g(s))_{s\in \Z}$ of elements $g(s)\in G$ such that 
\[
 g(s)g(s+1)^{-1} \in U(c^{-1})\,\bar{c},\qquad (s\in\Z).
\]
\end{Def}

To give a simple and concrete example, which will be developed in Section~\ref{sec:SL(n) cst}, if $G = SL(n)$ and $c  = c_{st} = s_1s_2\cdots s_{n-1}$ is the standard Coxeter element, then an $(SL(n),c_{st})$-band can be represented by
an infinite matrix $\mathrm{B} \in \mathrm{Mat}_{\infty,n}(\C)$ such that every submatrix of $\mathrm{B}$ consisting of $n$ consecutive rows has 
determinant 1.

The next theorem, which will be proved in Section~\ref{sec bands}, collects geometric properties of the space
of $(G,c)$-bands.

\begin{Thm}
The $(G,c)$-bands over $\C$ are the $\C$-rational points of an infinite-dimensional affine integral scheme $B(G,c)$. 
The ring $R(G,c)$ of regular functions on $B(G,c)$ is a unique factorization domain.
\end{Thm}

By construction there are natural projection maps:
\[
\pi_s : B(G,c) \longrightarrow G,\qquad (s\in \Z).
\]
At the level of $\C$-points, the map $\pi_s$ sends a band $b=(g(t))_{t\in \Z}$ to $g(s)$. We show that
the dual maps $\pi_s^* : \C[G] \to R(G,c)$ are injective, thus we get infinitely many subalgebras 
$\pi_s^*(\C[G])$ of $R(G,c)$ isomorphic to $\C[G]$.

Recall that Fomin and Zelevinsky have defined certain regular functions $\De_{u(\varpi_i),v(\varpi_i)}$ on $G$, called generalized minors, depending on the choice of $u,v\in W$ and of a fundamental weight $\varpi_i$
\cite[\S 1.4]{FZ}. These give rise to distinguished elements $\De^{(s)}_{u(\varpi_i),v(\varpi_i)} := \pi^*_s(\De_{u(\varpi_i),v(\varpi_i)})$ of $R(G,c)$.
Let $m_i$ be the smallest integer $k$ such that $c^k(\varpi_i) = w_0(\varpi_i)$. It follows from Definition~\ref{def-band} that these regular functions satisfy relations of the form
\begin{equation}\label{eq-glueing}
 \De^{(s)}_{c^k(\varpi_i),v(\varpi_i)} = \De^{(s+1)}_{c^{k-1}(\varpi_i),v(\varpi_i)},
 \qquad (1\le k \le m_i,\ v \in W),
\end{equation}
see Proposition~\ref{lem: glueing formulas} below.

As already mentioned, in \cite{GHL} a cluster algebra $\AA$ associated with $G$ was introduced. More precisely, $\AA$ is the algebra $\AA_{w_0}$ of \cite[\S3]{GHL} attached to the longest element $w_0$ of $W$. 
Moreover, a family of initial seeds of $\AA$ whose quivers $\Gamma_c$ correspond to Coxeter elements $c$ of $W$ was described. Given a Coxeter element $c$, let $\wt{c} := w_0c^{-1}w_0$. 
Our first main result is the following.

\begin{Thm}\label{Thm2}
There exists an algebra isomorphism $\C\otimes \AA \longrightarrow R(G,c)$ such that all cluster variables of the initial seed of $\AA$ with quiver $\Gamma_{\wt{c}}$ are mapped to regular functions of the form
$
 \De^{(s)}_{c^k(\varpi_i),\,\wt{c}^l(\varpi_i)}
$
for suitable integers $k,l,s$. 
\end{Thm}

We refer to Section~\ref{sec: statement of main theorem} for a more precise description of this isomorphism.
For now, let us just note that all subalgebras $\pi^*_s(\C[G])$ are cluster subalgebras, whose cluster structure
coincides with the classical cluster algebra structure on $\C[G]$. 
More precisely, Berenstein, Fomin and Zelevinsky \cite{BFZ} have shown that the coordinate ring of the open double
Bruhat cell $G^{w_0,w_0}$ is an upper cluster algebra with invertible coefficients. In a recent paper of Oya \cite{O} it is proved that the same cluster algebra, but with invertibility of coefficients removed, is isomorphic to $\C[G]$ and is equal to its upper counterpart. (See also \cite[Theorem B1]{QY} for a similar result, but only for the upper cluster algebra.)
This result plays an important role in our proof
of Theorem~\ref{Thm2}. Note also that the cluster algebra $\C\otimes\AA$ 
is equal to its upper cluster algebra (Corollary \ref{cor:AU R(G,c)}).

\smallskip
Let us now point out the analogies between $(G,c)$-bands and the $(G,q)$-opers and $(G,q)$-Wronskians of \cite{KSZ, FKSZ, KZ}. Let $q\in\C^*$ be of infinite multiplicative order. Following \cite{FRS, SS},
these papers introduce the notion of a $q$-connection on a principal $G$-bundle over $\Pro^1(\C)$. A $q$-connection can be regarded, via a trivialization, as an equivalence class of rational functions $z\mapsto g(z)$ in $G(\C(z))$ for the $q$-gauge transformation relation:
\[
 g_1(z) \sim g_2(z) \mbox{ if and only if there exists $h(z) \in G(\C(z))$ such that } g_1(z) = h(qz)g_2(z)h(z)^{-1}.  
\]
Then a $(G,q)$-oper (with respect to a fixed Coxeter element $c$) is a $q$-connection having a representative $A(z)$ satisfying a condition of the form
\[
  A(z) \in U(c^{-1})(\C(z))\,\bar{c},
\]
see \cite[Theorem 8.2]{FKSZ} or \cite[Theorem 3.3]{KZ}.
Here, to make the analogy more transparent, we deliberatly omit additional diagonal parameters $\Lambda_i(z)\in\C[z]$ or $\phi_i(z)\in \C(z)$ encoding possible singularities of the $(G,q)$-oper.
Hence, assuming that $A(z)$ is well-defined on the discrete subset $q^\Z\subset\C$, its restriction to this subset is a sequence of elements of $G$ satisfying 
\[
A(q^s) \in U(c^{-1})\,\bar{c},\qquad (s\in\Z). 
\]
Therefore, if $b = (g(s))_{s\in\Z}$ is a $(G,c)$-band, then the sequence $a = (g(s)g^{-1}(s+1))_{s\in\Z}$ can be regarded as a discrete analogue of a $q$-oper. Similarly, comparing Eq.~(\ref{eq-glueing}) with \cite[Proposition 4.17]{KZ}, we see that the band $b$ itself can be viewed as a discrete analog of a $(G,q)$-Wronskian. 

\smallskip
There is another class of geometric objects related to $(G,c)$-bands. 
For $M<N\in\Z$ let us denote by $B(G,c,M,N)$ the space of \emph{finite} $(G,c)$-bands 
\begin{equation}
    \label{eq:fin bands intro}
b = (g(s))_{M\le s\le N} \mbox{\quad with \quad} g(s)g(s+1)^{-1} \in U(c^{-1})\,\bar{c} \mbox{\quad for\quad}M\le s < N.
\end{equation}
These spaces $B(G,c,M,N)$ are smooth affine varieties which form an inverse system, and we show that $B(G,c)$ can be obtained as limit of this system. 
In other words, the spaces of finite bands endow $B(G,c)$ with the structure of a \textit{pro-variety}.
Let $R(G,c,M,N)$ be the ring of regular functions on $B(G,c,M,N)$.
We show that $R(G,c,M,N)$ can be naturally identified with the subring of $R(G,c)$ generated by the subalgebras $\pi_s^*(\C[G])$ with $M \leq s \leq N$.
In Section~\ref{sec-reduction-finite-bands},  Theorem~\ref{Thm2} is reduced to establishing that $R(G,c,M,N)$ is isomorphic to a suitable cluster subalgebra $\C\otimes \AA_{M,N}$ of $\C\otimes \AA$. 
The cluster algebra $\AA_{M,N}$ has $2n$ frozen variables, where $n$ is the rank of $G$. Let $\BB_{M,N}$ denote the
algebra obtained from $\AA_{M,N}$ by localizing at the product of all these frozen variables. While writing this paper we learned from Qin \cite{Q2} that 
$\C\otimes \BB_{M,N}$ is isomorphic to the cluster algebra structure on the coordinate ring of a certain decorated double Bott-Samelson cell, a new class of varieties generalizing double Bruhat cells introduced by Shen and Weng \cite{SW}, and shown by them to have a cluster structure. Thus, the open subspace of $B(G,c,M,N)$ defined by the non-vanishing of the frozen variables is isomorphic to a decorated double Bott-Samelson cell, that is, to a moduli space of configurations of points in $G/U$ with prescribed relative positions controlled by $c$. 

\smallskip
However, for our purposes, it is simpler and more natural to work with $(G,c)$-bands. For example, it is easy to deduce from Definition~\ref{def-band} that the action of $G$ on itself by right translations lifts to a free action of $G$ on $B(G,c)$ (see below, \S\ref{subsec-action}). This induces a linear action of $G$ on the algebra of regular functions $R(G,c)$. 
In the second part of the paper we study how this $G$-action allows to identify two interesting cluster subalgebras of
$R(G,c)$, which have already appeared in \cite{HL1, HL2, H} in connection with the representation theory of quantum affine algebras.

Our second main result is the following.
\begin{Thm}\label{Thm3}
\begin{itemize}
 \item[(i)] The subalgebra $R(G,c)^U$ of $U$-invariant functions in $R(G,c)$ has the structure of an upper cluster algebra with a distinguished initial cluster equal to 
 \[
 \{\De^{(s)}_{\varpi_i,\varpi_i} \mid (s\in \Z,\ i\in I)\}.
 \]
 \item[(ii)] The subalgebra $R(G,c)^G$ of $G$-invariant functions in $R(G,c)$ has the structure of a cluster algebra with various special initial 
 clusters, all consisting of cluster variables of the form 
 \[
\theta_{i,k}^{(s)}:\quad b=(g(s))_{s \in \Z} \mapsto \De_{\varpi_i,\varpi_i}\left(g(s)g(s+k)^{-1}\right), 
 \]
 for suitable choices of $i\in I$, $s\in \Z$ and $k\in \Z_{>0}$.
\end{itemize}
\end{Thm}
For a more precise statement of Theorem~\ref{Thm3}~(i), see Theorem~\ref{Thm.6.11}.
The cluster algebra of (i) is isomorphic to the complexified Grothendieck ring of the category $O^+_\Z$ of representations of a Borel subalgebra of the quantum affine algebra $U_q(\widehat{\g})$ studied in \cite{HL2}. It is also isomorphic to the complexified Grothendieck ring of the subcategory $\CC^{\rm shift}_\Z$ of finite-dimensional objects of $\O^{\rm shift}_\Z$ studied in \cite{H}. Under these isomorphisms, the distinguished initial cluster coincides with the cluster of 
classes of positive prefundamental modules in $O^+_\Z$ or $\CC^{\rm shift}_\Z$.

For a more precise statement of Theorem~\ref{Thm3}~(ii), see Theorem~\ref{Thm.7.3}.
The cluster algebra of (ii) is isomorphic to the Grothendieck ring of a category $\CC_\Z$ of finite-dimensional modules over the (unshifted) quantum affine algebra $U_q(\widehat{\g})$. This category $\CC_\Z$ was introduced in \cite{HL}, and the cluster structure on its Grothendieck ring was further studied in many papers, in particular \cite{HL1}, \cite{Q}, \cite{KKOP1}, \cite{KKOP2}. Under this isomorphism, the $G$-invariant cluster variables $\theta_{i,k}^{(s)}$ of (ii) identify with classes of Kirillov-Reshetikhin modules. As in Theorem \ref{Thm2}, one can show that this cluster algebra is equal to its upper cluster algebra (see Remark \ref{rem: AU R(G,c)G}).

\smallskip
The paper is structured as follows. After setting up  our notation and terminology for algebraic groups and schemes in Section~\ref{sec_notation}, we present in detail in Section~\ref{sec:SL(n) cst} the simplest example of $G=SL(n)$ and $c=c_{st} = s_1s_2\cdots s_{n-1}$.
In Section~\ref{sec bands}, we define the schemes of bands $B(G,c)$ and finite bands $B(G,c,M,N)$, and establish their main properties. Section~\ref{sec_cluster_bands} is devoted to the statement and proof of the cluster algebra structure on $R(G,c)$. Similarly, Section~\ref{sec_U_inv} and Section~\ref{sec_G_inv} are devoted to the cluster algebra structures on $R(G,c)^U$ and $R(G,c)^G$. To conclude the paper, we provide in Section~\ref{sec_band_qaf} a dictionary between bands and representations of quantum affine algebras.
We emphasize that, although the main motivation for introducing $(G,c)$-bands comes from the representation theory of quantum affine algebras, all the material in Sections~\ref{sec:SL(n) cst}
to Section~\ref{sec_G_inv} is developed without any reference to this theory and can be read without prior knowledge of the vast literature on the subject. On the other hand, readers who are already familiar with representation theory of quantum affine algebras might start looking at Section~\ref{sec_band_qaf} to understand the intuitions behind our geometric construction.

\bigskip
{\bf Acknowledgements.} We are grateful to an anonymous referee for a careful reading of the manuscript and for suggesting numerous improvements in the text. L.F. thanks LMNO (Université de Caen Normandie) for an invitation in March 2024 during which this joint project was started. We are very grateful to C. Geiss, D. Gori, D.~Hernandez and H. Oya for many fruitful and inspiring discussions during the preparation of this work.
L.F. acknowledges support of the MUR Excellence Department Project 2023--2027 awarded to the Department of Mathematics, University of Rome Tor Vergata CUP E83C18000100006, and of the 
PRIN2022 CUP E53D23005550006.

\section{Setup}
\label{sec_notation}

\subsection{Algebraic groups}
\label{sec: alg gps}

From now on, we replace the base field $\C$ by an algebraically closed field $K$ of characteristic $0$.
Let $G$ be a simple, simply-laced, simply-connected algebraic group over $K$. Let $\g$ be the Lie algebra of $G$. We denote by $I$ the set of vertices of the Dynkin diagram of $\g$. We denote by $C=[c_{ij}]_{i,j\in I}$
the associated Cartan matrix.
That is, we have $c_{ii} = 2$, and for $i\not = j$ we put $c_{ij} = -1$ if $i$ and $j$
are connected by an edge of the Dynkin diagram, and $c_{ij} = 0$ otherwise.
Let $T$ denote a maximal torus of $G$, $B$ a Borel subgroup containing $T$, $U$ the unipotent radical of $B$.
Let $B^-$ be the opposite Borel subgroup of $B$ with respect to $T$, and $U^-$ its unipotent radical.
Let $W$ denote the Weyl group of $G$, with its Coxeter generators $s_i\ (i\in I)$, length function $l$, and longest element $w_0$. For $w\in W$, let $U(w) := U \cap (w^{-1}\,U^-w)$. We recall that $U(w)$ is a unipotent group of dimension $l(w)$. As an algebraic variety, it is hence isomorphic to an affine space $\A^{l(w)}$ of dimension $l(w)$ \cite[\S 2.4]{FZ}.

Let $P$ denote the weight lattice of $\g$, with its $\Z$-basis of fundamental weights $\{\varpi_i \mid i\in I\}$. Let $Q\subset P$ denote the root lattice, with its $\Z$-basis of simple roots $\{\a_i \mid i\in I\}$.
We have $\alpha_i = \sum_{j\in I} c_{ij}\varpi_j$.
For $u,v\in W$ and $i\in I$ we denote by $\De_{u(\varpi_i),v(\varpi_i)}$ the corresponding Fomin-Zelevinsky generalized minor, a regular function on $G$ which depends only on the weights $u(\varpi_i), v(\varpi_i) \in P$.

Let $c$ be a fixed Coxeter element of $W$. We denote by $m_i$ the minimal $k\in\Z_{\ge 0}$ such that $c^k(\varpi_i) = w_0(\varpi_i)$. We denote by $\wt{c}$ the Coxeter element $w_0c^{-1}w_0$.
Clearly $\wt{\wt{c}} = c$. 

Following \cite{FZ}, we will denote by $\overline{s_i} := \exp(-e_i)\exp(f_i)\exp(-e_i)$ a specific lift of the Coxeter generator $s_i$ of $W$ to the normalizer $N(T)$ of $T$ in $G$. The elements $\overline{s_i}$ satisfy the braid relations, so we can define more generally, using reduced expressions, a lift $\overline{w}$ of every $w\in W$, such that $\overline{w}_1\overline{w}_2 = \overline{w_1w_2}$ if 
$l(w_1w_2) = l(w_1) + l(w_2)$. Note however that $\overline{s_i}$ is not an involution. We only have $\overline{s_i}^4 = 1$. 
The elements $\overline{s_i}\ (i\in I)$ generate a finite covering $\overline{W}$ of $W$ in $N(T)$.

\subsection{Schemes}
We fix in this section our notation for schemes.
We refer to the book \cite{EH} for a comprehensive treatment of the subject.

Let $\Schk$ be the category of $K$-schemes, and let $X$ be a $K$-scheme. 
The functor $\Hom( - , X) : (\Schk)^{op} \lra \Set$ assigning to a $K$-scheme $Z$ the set of  $K$-schemes morphisms from $Z$ to $X$ and to a morphism of $K$-schemes $f: Z' \lra Z$ the function $\Hom(Z,X) \lra \Hom(Z', X)$ given by composition with $f$ is the \textit{functor of points} of $X$. 
An element of the set $X(Z):= \Hom(Z,X)$ is called a $Z$-point of $X$.
If $R$ is an associative, unitary and commutative $K$-algebra, we will use the shorthand notation 
$X(R)$ to denote the set $X(\Spec(R))$, and we will refer to an element of $X(R)$ as an $R$-point of $X$.
Recall that there is a natural bijection between the set $X(K)$ and the set of $K$-rational points of $X$, which assigns to an element $f$ of $X(K)$ the image of the unique point of $\Spec(K)$ under the morphism $f$. 
Through the paper, we consider these two sets as canonically identified through this bijection.

Assume that $Y$ is another $K$-scheme and let $f:X \lra Y$  be a morphism of $K$-schemes.
Then, the composition with $f$ determines a natural transformation between the functors of points of $X$ and $Y$.
By the Yoneda lemma, any natural transformation $\Hom(-,X) \lra \Hom(-,Y)$ arises in this way from a unique morphism of $K$-schemes $X \lra Y.$

The scheme $X$ is a \textit{variety} if it is reduced, separated, and of finite type over $\Spec(K)$. 
If $X$ is a variety, we use the notation $x \in X$ to indicate that $x$ is a $K$-point of $X.$ 
Recall that a morphism between two varieties is uniquely determined by the values it takes on $K$-points.

Recall that if $G$ is an algebraic group and $Z$ is a $K$-scheme, the set $G(Z)$ is actually a group.

\section{The case of $SL(n)$ and $c_{st}$}
\label{sec:SL(n) cst}

Before dealing with the general case, we explain our construction and main results in the simplest example of $G = SL(n)$ and
the standard Coxeter element $c_{st} = s_1\cdots s_n$. 
Indeed, in this case bands have a concrete and elementary 
description in terms of infinite matrices.
The results of this section are stated without proof since they are special cases of the general results
proved in the following sections of the paper (see in particular \S\ref{Sect-4.5}).

\subsection{$(SL(n), c_{st})$-bands and their cluster structure}
\label{sec:first decription}

Let $n \geq 2$ be an integer. We identify $SL(n,K) :=SL(n)(K)$ with the group of $n\times n$ matrices $g = [g_{ij}]$ over $K$ with determinant 1, and we choose for $B$ the Borel subgroup of upper triangular matrices.
Let $c=c_{st} := s_1s_2\cdots s_{n-1}$ be the standard Coxeter element. In this case, we have $\wt{c_{st}} = c_{st}$ and $m_i = n-i\ (1\le i \le n-1)$.
It is easy to check that the subgroup $U(c_{st}^{-1})$ consists of all matrices
of the form
\[
\pmatrix
{
 1 & a_1  &\cdots & a_{n-1}\cr
 0 & 1  & \cdots & 0\cr
 \vdots&\vdots&\ddots& \vdots\cr
 0 & 0 & \cdots & 1
},
\qquad (a_1,\ldots, a_{n-1} \in K).
\]
We also have 
\[
\overline{c_{st}} =
\pmatrix
{
 0 & 0  &\cdots & 0 &(-1)^{n-1}\cr
 1 & 0  & \cdots & 0 &0\cr
 0 & 1  & \cdots & 0& 0\cr 
 \vdots&\vdots&\ddots&\vdots & \vdots\cr
 0 & 0 & \cdots & 1& 0
}.
\]
It then follows that
\begin{equation}
    \label{eq: Steinbers SLn cst}
U(c_{st}^{-1}) \oc_{st}=\left\{
\pmatrix
{
 a_1 & a_2  &\cdots & a_{n-1} & (-1)^{n-1}\cr
  1 & 0 & \cdots & 0 & 0\cr
  0 & 1 & \cdots & 0 & 0 \cr
 \vdots&\vdots&\ddots& \vdots & \vdots\cr
 0 & 0 & \cdots & 1 & 0
},
\qquad (a_1,\ldots, a_{n-1} \in K)\right\}.
\end{equation}
 Moreover, for any associative, unitary and commutative $K$-algebra $R$, the set $(U(c_{st}^{-1})\oc_{st})(R)$ of $R$-points of $U(c_{st}^{-1})\oc_{st}$ is the set obtained by replacing $K$ with $R$ in Eq.~(\ref{eq: Steinbers SLn cst}).
Thus, a simple computation of linear algebra shows that two matrices $g(s) := [g^{(s)}_{ij}]$ and $g(s+1) := [g^{(s+1)}_{ij}]$ of $SL(n,R):=SL(n)(R)$ satisfy the condition $g(s)g(s+1)^{-1} \in (U(c_{st}^{-1})\,\overline{c_{st}})(R)$ if and only if
\begin{equation}
\label{eq:band cond SL}
    g^{(s)}_{ij} = g^{(s+1)}_{i-1,j},\qquad (1<i\le n,\ 1\le j\le n). 
\end{equation}
In other words, the last $n-1$ rows of $g(s)$ are equal to the first $n-1$ rows of $g(s+1).$
Thus Definition~\ref{def-band} translates in this case as:

\begin{Def}
\label{def: SL(n) cst bands}
    A $(SL(n), c_{st})$-band is a sequence of elements 
$$
g(s) \in SL(n,K) \quad (s\in\Z),
$$
satisfying the relations
\begin{equation}\label{glue-eq2}
g_{ij}^{(s)} = g_{i-1,j}^{(s+1)}, \quad (1< i \le n,\ 1\le j\le n,\ s\in \Z). 
\end{equation}
\end{Def} 
Explicitly, Eq.~(\ref{glue-eq2}) means that the first $n-1$ rows of $g(s+1)$ coincide with
the last $n-1$ rows of $g(s)$, for every $s\in \Z$.
Therefore, we can think of
an $(SL(n),c_{st})$-band as an $(\infty\times n)$-array 
\begin{equation}
    \label{eq:inf arrays1}
    \mathrm{B} = \left[b_{ij}\right], \quad (i\in \Z,\ 1\le j\le n,\ b_{ij}\in K)
\end{equation} 
such that every $n\times n$ submatrix
\begin{equation}
    \label{eq:inf arrays2}
     \mathrm{B}(s) = \left[b_{ij}\right], \quad (s\le i\le s+n-1,\ 1\le j\le n)
\end{equation}
  consisting of $n$ consecutive rows of $\mathrm{B}$ belongs to $SL(n,K)$. That is, such that 
$
 \det  \mathrm{B}(s) = 1
$
for every $s\in\Z$. Such an array $\mathrm{B}$ corresponds to the $(SL(n,K),c_{st})$-band $(\mathrm{B}(s))_{s\in \Z}$. This is the reason for the name \emph{band}.

\begin{example}
{\rm 
The sequence 
\[
 g(-s) := \left[
\begin{array}{cc}
s+2&s+3\\
s+1&s+2
\end{array}
\right],\ (s\ge 0),\quad
g(s) := \left[
\begin{array}{cc}
s&3s-1\\
s+1&3s+2
\end{array}
\right],\ (s > 0),
\]
is an $(SL(2),c_{st})$-band.
It can be represented by the infinite array
\[
  \mathrm{B}=
\pmatrix{
\vdots&\vdots\cr
4&5\cr
3&4\cr
{\red 2}&{\red 3}\cr
{\red 1}&{\red 2}\cr
2&5\cr
3&8\cr
4&11\cr
\vdots&\vdots
}, 
\]
where $\mathrm{B}(0)$ is marked in red colour.
}
\end{example}

\medskip

Let 
$K[X_{ij} \mid  i \in \Z, \ 1\le j \le n]$
be the polynomial ring in the variables $X_{ij}.$ For $s \in \Z$, we set
\begin{equation}
\label{eq:Y dets}
Y_{s,n} := \det [X_{ij} \mid  s\le i \le s+n-1, \ 1\le j \le n]. 
\end{equation}
We can consider the quotient ring
\begin{equation}
\label{eq:Rn SL(s)}
\R_n:=  K[X_{ij} \mid  i \in \Z, \ 1\le j \le n] / (Y_{s,n}-1 \mid s \in \Z), 
\end{equation}
and the associated infinite dimensional affine $K$-scheme
\begin{equation}\label{eq: inf bands SL(n)}
\B_n := \Spec(\R_n). 
\end{equation}
The previous discussion shows that the set of $(SL(n), c_{st})$-bands can be naturally identified with the set of $K$-rational points of $\B_n$. 

\medskip
In this particular case, Theorem~\ref{Thm2} can be stated as
Theorem \ref{thm: cluster structure SL cst bands} below.
We first explicitly describe an initial seed of the cluster structure in this case.
The quiver of this seed is the quiver $\Gamma_{c_{st}}$ of type $A_{n-1}$ of \cite[Section 3.4]{GHL}.
This is an infinite quiver with $n-1$ columns and vertex set  
\[
\VV := \{(i,r)\mid 1\le i \le n-1,\ r\in \Z,\ r\equiv i-1\ (2)\},
\]
where the first coordinate $i$ indicates the column on which the vertex $(i,r)$ sits.
This quiver $\G_{c_{st}}$ can be divided into three segments : a finite connected middle part containing green and red vertices, and two infinite upper and lower parts containing the remaining black vertices. 
The red vertices have labels of the form
\[
 (i, 1-i - 4k),\quad (i\in I,\ 0\le k< n-i),
\]
the green vertices have labels of the form
\[
 (i, 1-i - 4k-2),\quad (i\in I,\ 0\le k< n-i),
\]
the black vertices of the upper part have labels of the form
\[
 (i,r),\quad (i\in I,\ r\equiv i-1\ (2),\ r>1-i),  
\]
and the black vertices of the lower part have labels of the form
\[
 (i,r),\quad (i\in I,\ r\equiv i-1\ (2),\ r\le 1+3i-4n).
\]

For $(i,r)\in \VV$, set $r_i := (r-1+i)/2$, and define $\De_{(i,r)}\in \R_n$ as follows.
Denote by $\De^{(s)}_{u(\varpi_i),v(\varpi_i)}$
the regular function on $\B_n$ given by 
\[
\mathrm{B}\mapsto \De_{u(\varpi_i),v(\varpi_i)}(\mathrm{B}(s)),
\qquad (s\in\Z,\ 1\le i\le n-1,\ u,v \in W).
\]
Then
\begin{itemize}
 \item for a black vertex $(i,r)$ in the upper part of $\G_{c_{st}}$ we set 
$
\De_{(i,r)} := \Delta^{(r_i-1)}_{w_0(\varpi_i),\varpi_i}
$;
 \item for a black vertex $(i,r)$ in the lower part of $\G_{c_{st}}$ we set 
$
 \De_{(i,r)} :=\Delta^{(r_i+2(n-i)-1)}_{\varpi_i,\,w_0(\varpi_i)}
$;
 \item for a red vertex $(i,r)= (i, 1-i - 4k)$ we set $\De_{(i,r)} := \Delta^{(0)}_{c_{st}^{n-i-1-k}(\varpi_i),\, c_{st}^{\,k}(\varpi_i)}$;
 \item for a green vertex $(i,r) = (i, 1-i - 4k-2)$ we set $\De_{(i,r)} := \Delta^{(0)}_{c_{st}^{n-i-1-k}(\varpi_i),\, c_{st}^{\,\,k+1}(\varpi_i)}$.
\end{itemize}

\begin{Thm}\label{thm: cluster structure SL cst bands}
    The ring $\R_n$ has a cluster algebra structure with initial seed
    \[
    \left(\G_{c_{st}},\ \{\De_{(i,r)} \mid (i,r) \in \VV \}\right).
    \]
    Therefore, we have a ring isomorphism $K\otimes \AA \to \R_n$, where $\AA = \AA_{w_0}$ is 
    the cluster algebra defined in \cite{GHL} for $G = SL(n)$.
\end{Thm}

Figure~\ref{Fig2n} illustrates Theorem~\ref{thm: cluster structure SL cst bands} in the case of $G=SL(3)$.
More precisely, this figure displays a finite section of the doubly infinite quiver $\Gamma_{c_{st}}$ containing the middle part with red and green vertices. Each vertex $(i,r)$ of the left quiver is replaced in the right quiver by the corresponding initial cluster variable $\De_{(i,r)}$.
Here we have used the traditional notation $\De_{P,Q}$ to indicate the determinant of the band $B$ whose row indices belong to $P\subset \Z$ and column indices belong to $Q\subset\{1,2,3\}$.

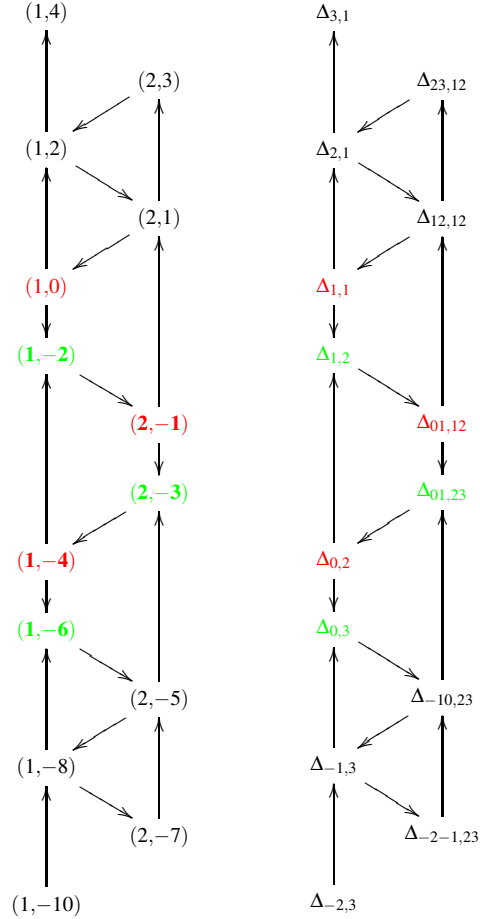
\begin{figure}[t!]
\[
\def\objectstyle{\scriptstyle}
\def\lablestyle{\scriptstyle}
\xymatrix@-1.0pc{
{(1,4)}
\\
&\ar[ld] {(2,3)} 
\\
{(1,2)}\ar[rd]\ar[uu]&& 
\\
&\ar[ld] (2,1) \ar[uu]
\\
\ar[uu]{\red\mathbf(1,0)}\ar[d]&&
\\
\mathbf{\green (1,-2)}\ar[rd]
\\
&\ar[uuu] \mathbf{\red (2,-1)}\ar[d] 
\\
& \ar[ld]\mathbf{\green (2,-3)}
\\
\ar[uuu]\mathbf{\red (1,-4)} \ar[d] 
\\
\mathbf{\green (1,-6)}\ar[rd]&&
\\
& \ar[uuu] {(2,-5)}\ar[ld] 
\\
\ar[uu]{(1,-8)} \ar[rd] && 
\\
& \ar[uu]{(2,-7)} 
\\
\ar[uu]{(1,-10)}
}
\qquad
\xymatrix@-1.0pc{
{\De_{3,1}}
\\
&\ar[ld] {\De_{23,12}} 
\\
{\De_{2,1}}\ar[rd]\ar[uu]&
\\
&\ar[ld] \De_{12,12} \ar[uu]
\\
\ar[uu]{\red\Delta_{1,1}}\ar[d]&
\\
{{\green \Delta_{1,2}}}\ar[rd]
\\
&\ar[uuu] {\red \De_{01,12}}\ar[d] 
\\
& \ar[ld]{\green \De_{01,23}}
\\
\ar[uuu]{\red \De_{0,2}} \ar[d] 
\\
{{\green \De_{0,3}}}\ar[rd]&
\\
& \ar[uuu] {\De_{-10,23}}\ar[ld] 
\\
\ar[uu]{\De_{-1,3}} \ar[rd] &
\\
& \ar[uu]{\De_{-2-1,23}} 
\\
\ar[uu]{ \De_{-2,3}}
}
\]
\caption{\label{Fig2n} {\it A finite segment of an initial seed of the cluster algebra structure on $\R_3$}}
\end{figure}

\subsection{Finite bands}
\label{subsect:reformulations}

\begin{figure}[t!]
\[
\def\objectstyle{\scriptstyle}
\def\lablestyle{\scriptstyle}
    \xymatrix@-1.0pc{
{\blue (1,4)}
\\
&\ar[ld] {\blue (2,3)} 
\\
{(1,2)}\ar[rd]\ar[uu]&& 
\\
&\ar[ld] (2,1) \ar[uu]
\\
\ar[uu]{\red\mathbf(1,0)}\ar[d]&&
\\
\mathbf{\green (1,-2)}\ar[rd]
\\
&\ar[uuu] \mathbf{\red (2,-1)}\ar[d] 
\\
& \ar[ld]\mathbf{\green (2,-3)}
\\
\ar[uuu]\mathbf{\red (1,-4)} \ar[d] 
\\
\mathbf{\green (1,-6)}\ar[rd]&&
\\
& \ar[uuu] {(2,-5)}\ar[ld] 
\\
\ar[uu]{(1,-8)} \ar[rd] && 
\\
& \ar[uu]{\blue (2,-7)} 
\\
\ar[uu]{\blue (1,-10)}
}
\quad
\xymatrix@-1.0pc{
{\blue \De^{(1)}_{w_0(\varpi_1), \varpi_1}}
\\
&\ar[ld] { \blue \De^{(1)}_{w_0(\varpi_2), \varpi_2}} 
\\
{\De^{(0)}_{w_0(\varpi_1), \varpi_1}}\ar[rd]\ar[uu]&
\\
&\ar[ld] { \De^{(0)}_{w_0(\varpi_2), \varpi_2}} \ar[uu]
\\
\ar[uu]{\red \De^{(0)}_{c_{st}(\varpi_1), \varpi_1 }}\ar[d]&
\\
{{\green  \De^{(0)}_{c_{st}(\varpi_1), c_{st} (\varpi_1)}}}\ar[rd]
\\
&\ar[uuu] {\red \De^{(0)}_{\varpi_2, \varpi_2}}\ar[d] 
\\
& \ar[ld]{{\green \De^{(0)}_{\varpi_2, c_{st}(\varpi_2)} }}
\\
\ar[uuu]{\red \De^{(0)}_{\varpi_1, c_{st} (\varpi_1)}} \ar[d] 
\\
{{\green \De^{(0)}_{\varpi_1, c_{st}^2 (\varpi_1)}}}\ar[rd]&
\\
& \ar[uuu] {\De^{(-1)}_{\varpi_2, w_0(\varpi_2)}}\ar[ld] 
\\
\ar[uu]{ \De^{(-1)}_{\varpi_1, w_0(\varpi_1)}} \ar[rd] &
\\
& \ar[uu]{\blue \De^{(-2)}_{\varpi_2, w_0(\varpi_2)}} 
\\
\ar[uu]{ \blue  \De^{(-2)}_{\varpi_1, w_0(\varpi_1)}}
}
\qquad
\xymatrix@-1.0pc{
{\blue (1,4)}
\\
&\ar[ld] {\blue (2,3)} 
\\
{(1,2)}\ar[rd]\ar[uu]&& 
\\
&\ar[ld] (2,1) \ar[uu]
\\
\ar[uu]{\red\mathbf(1,0)}\ar[d]&&
\\
\mathbf{\green (1,-2)}\ar[rd]
\\
&\ar[uuu] \mathbf{\red (2,-1)}\ar[d] 
\\
& \ar[ld]\mathbf{\green (2,-3)}
\\
\ar[uuu]\mathbf{\red (1,-4)} \ar[d] 
\\
\mathbf{\green (1,-6)}\ar[rd]&&
\\
& \ar[uuu] {\blue (2,-5)}  
\\
\ar[uu]{ \blue (1,-8)} && 
}
\quad
\xymatrix@-1.0pc{
{\blue \De^{(1)}_{w_0(\varpi_1), \varpi_1}}
\\
&\ar[ld] { \blue \De^{(1)}_{w_0(\varpi_2), \varpi_2}} 
\\
{\De^{(0)}_{w_0(\varpi_1), \varpi_1}}\ar[rd]\ar[uu]&
\\
&\ar[ld] { \De^{(0)}_{w_0(\varpi_2), \varpi_2}} \ar[uu]
\\
\ar[uu]{\red \De^{(0)}_{c_{st}(\varpi_1), \varpi_1 }}\ar[d]&
\\
{{\green  \De^{(0)}_{c_{st}(\varpi_1), c_{st} (\varpi_1)}}}\ar[rd]
\\
&\ar[uuu] {\red \De^{(0)}_{\varpi_2, \varpi_2}}\ar[d] 
\\
& \ar[ld]{{\green \De^{(0)}_{\varpi_2, c_{st}(\varpi_2)} }}
\\
\ar[uuu]{\red \De^{(0)}_{\varpi_1, c_{st} (\varpi_1)}} \ar[d] 
\\
{{\green \De^{(0)}_{\varpi_1, c_{st}^2 (\varpi_1)}}}\ar[rd]&
\\
& \ar[uuu] { \blue \De^{(-1)}_{\varpi_2, w_0(\varpi_2)}}
\\
\ar[uu]{ \blue \De^{(-1)}_{\varpi_1, w_0(\varpi_1)}}  &
}
\]
\caption{\label{Fig2} {\it The iced-quivers $\Gamma_{c_{st}, -2,1}$ and $\Gamma_{c_{st}, -1,1}$ for $n=3$}}
\end{figure}
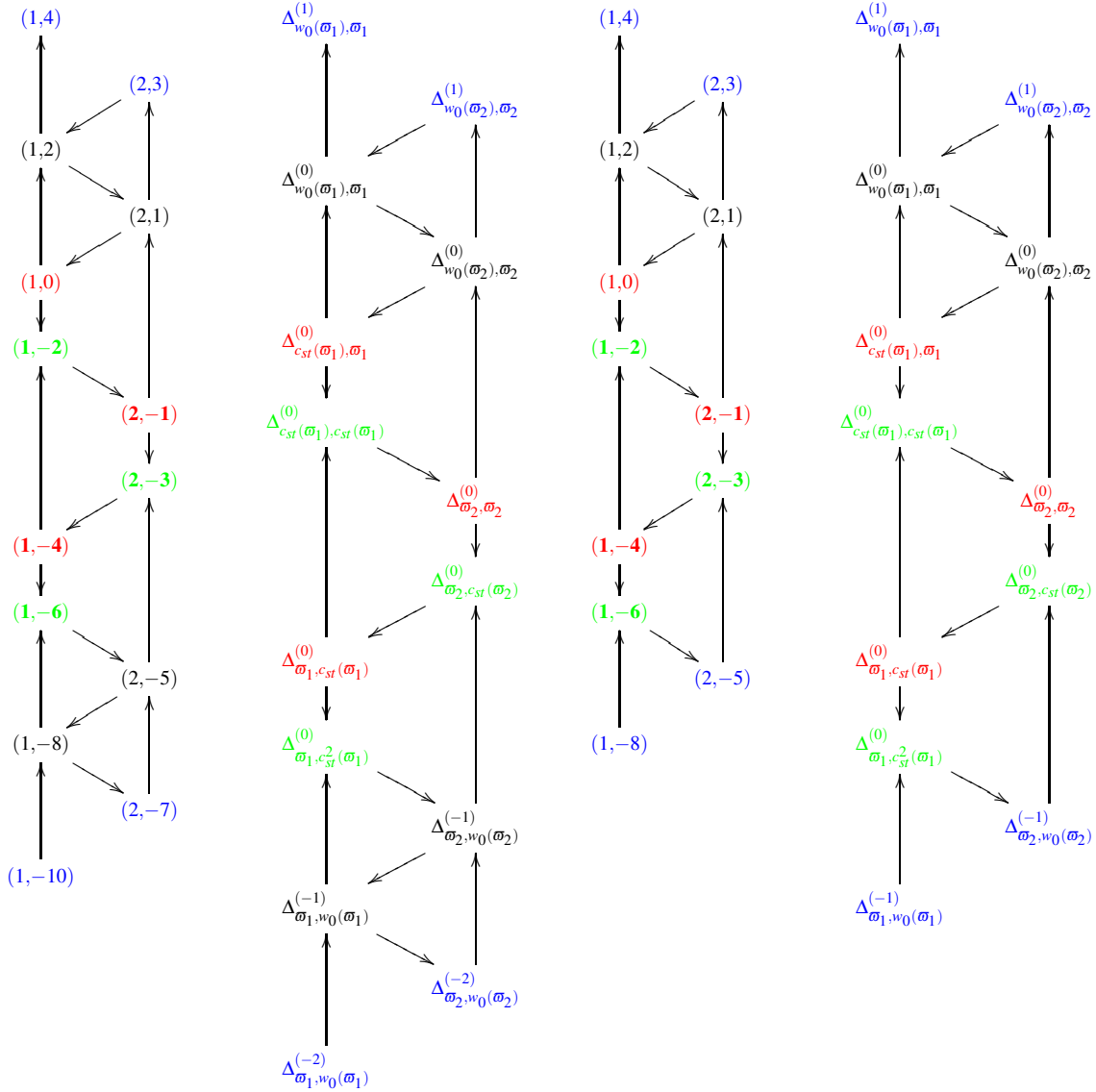

The schemes of bands admit some finite dimensional analogues whose points are described by Eq. (\ref{eq:fin bands intro}). 
In the case of $SL(n)$ and the Coxeter element $c_{st}$, the previous elementary description of bands can be extended to their finite analogues. 
In fact, a $(SL(n), c_{st}, M,N)$-band $(M\leq N)$ can be represented by  a $(N-M +n)\times n$-array 
\[
 \mathrm{B} = \left[b_{ij}\right], \quad (M \leq i \leq N+n-1,\ 1\le j\le n,\ b_{ij}\in K)
\] 
such that every $n\times n$ submatrix
\[
   \mathrm{B}(s) = \left[b_{ij}\right], \quad (M \leq s \leq N, \ s\le i\le s+n-1,\ 1\le j\le n)
\]
consisting of $n$ consecutive rows of $B$ belongs to $SL(n,K)$.  
Such an array $\mathrm{B}$ corresponds to the $(SL(n,K),c_{st}, M,N)$-band $(\mathrm{B}(s))_{M \le s \le N}$.

Recall the notation of Eq. (\ref{eq:Y dets}) and Eq. (\ref{eq:Rn SL(s)}).  
Let 
$$
\R_{n,M,N}:=  K[X_{ij} \mid  M \leq i \leq N+n-1, \ 1\le j \le n] / (Y_{s,n}-1 \mid M \leq s \leq N), \qquad (M \leq N)
$$
and 
$$
    \B_{n,M,M}:= \Spec \bigl( \R_{n,M,N} \bigr).
$$
The previous discussion implies that  $(SL(n), c_{st}, M,N)$-bands are the $K$-rational points of the scheme $\B_{n,M,M}$. 
The schemes of finite bands $\B_{n,M,M} \ (M \leq N)$ are smooth affine varieties.

 Notice that the natural map $\R_{n,M,N} \lra \R_n$ is injective.
 We denote its induced morphism of schemes  by $\pi_{M,N} : \B_n \lra \B_{n,M,N}$.
 At the level of $K$-points, the morphism $\pi_{M,N}$ sends a band $\mathrm{B}$ to the array obtained by forgetting the rows of $\mathrm{B}$ that do not belong to the interval $[M,N].$ 
The fact that the $K$-algebra $\R_n$ is the union of its subalgebras $\R_{n,M,N} \, (M\leq N)$ can be rephrased by saying that the scheme $\B_n$  is the limit of an inverse system of schemes involving the $\B_{n,M,N}$ with respect to which the $\pi_{M,N} \ (M \leq N)$ are  the canonical morphisms of the limit structure.  
The cluster structure of the ring $\R_n$ described by Theorem \ref{thm: cluster structure SL cst bands} is compatible with this limit structure. 
In order to explain this fact, we need to introduce some notation.

\smallskip
In \S\ref{sec:first decription}, the element  $\Delta_{(i,r)}$ of $\R_n$ is defined as $\Delta_{u(\varpi_i),v(\varpi_i)}^{(s)}$ for some $u,v \in W$  and $s \in \Z.$
Let $s_{(i,r)}$ denote this integer $s$.
For $M \leq 0 \leq N$, we set
$$
\VV_{M,N}:= \{ (i,r) \in \VV \mid M \leq s_{(i,r)} \leq N\},
$$
and we denote by $\Gamma_{c_{st},M,N}$ the full subquiver of $\Gamma_{c_{st}}$ supported on  $\VV_{M,N}$.
The quiver $\Gamma_{c_{st},M,N}$ is considered as an iced-quiver by declaring that its highest and lowest vertices in any column are frozen.
For instance, the leftmost side of Figure \ref{Fig2} displays the iced-quiver $\Gamma_{c_{st},-2,1}$ together with a second copy in which any vertex is replaced by its initial cluster variable.
Similarly, the rightmost side of  Figure \ref{Fig2} displays the iced-quiver $\Gamma_{c_{st},-1,1}$.
The frozen vertices are painted in blue.

If we denote by $\AA_{M,N}$ the cluster algebra of the quiver $\Gamma_{c_{st},M,N}$, with non-invertible frozen variables, then we have a natural inclusion  $\AA_{M,N} \subseteq \AA$.
Moreover, any element of $\AA$ belongs to $\AA_{M,N}$ for $-M,N$ sufficiently large.
Indeed, the cluster algebra $\AA$ is the colimit of the system of subalgebras $\AA_{M,N}$.

\medskip

Moreover, the isomorphism $K\otimes \AA \lra \R_n$ of Theorem~\ref{thm: cluster structure SL cst bands} induces isomorphisms 
\begin{equation}
\label{eq: clust partial bands SL}
K\otimes \AA_{M,N} \lra \R_{n,M,N} \qquad (M \leq 0 \leq N).
\end{equation}

The fact that the maps (\ref{eq: clust partial bands SL}) are isomorphisms will be the main step in the proof of Theorem~\ref{thm: cluster structure SL cst bands}, which then follows as an easy consequence of the properties of limits.
Proving that the maps $K\otimes \AA_{M,N} \lra \R_{n,M,N}$ are isomorphisms is actually simpler than proving directly that the global morphism $K\otimes \AA \lra \R_{n}$ is an isomorphism, because the algebras $\R_{n,M,N}$ are finitely generated.

\subsection{$U$-invariants}

\begin{figure}[t!]
\[
\xymatrix@-1.0pc{
{\De_{3}}\ar[rd]
\\
&\ar[ld] {\De_{23}} 
\\
{\De_{2}}\ar[rd]\ar[uu]&
\\
&\ar[ld] \De_{12} \ar[uu]
\\
\ar[uu]{\Delta_{1}}\ar[rd]&
\\
& \ar[uu] {\De_{01}}\ar[ld] 
\\
\ar[uu]{\De_{0}} \ar[rd] &
\\
& \ar[uu]{\De_{-10}} \ar[ld]
\\
\ar[uu]{ \De_{-1}}\ar[rd]
\\
& \ar[uu]{\De_{-2-1}} \ar[ld]
\\
\ar[uu]{ \De_{-2}}
}
\]
\caption{\label{Fig3a} {\it A finite segment of an initial seed of the cluster algebra structure on $\R_3^U$}}
\end{figure}
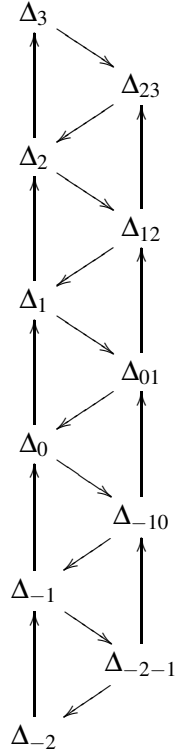

The group $G = SL(n)$ has a natural right action on the space of bands via matrix multiplication
\[
 \mathrm{B}\cdot g := \mathrm{B}g,\qquad (\mathrm{B} \in \B_n, \ g\in SL(n)).
\]
We can restrict this action to the subgroup $U$ and consider the subalgebra $\R_n^U$ of $U$-invariant
regular functions on $\B_n$.

For instance, for every subset $P$ of $\Z$ with cardinality $1\le k\le n$, the flag minor $\De_{P,[1,\,k]}$
is an element of $\R_n^U$. In particular, for every $i\in I$ and $s\in \Z$ the function 
$
\De_{[s+n-i,\,s+n-1],\,[1,\,i]} = \De^{(s)}_{w_0(\varpi_i), \varpi_i}
$
belongs to $\R_n^U$. 

The statement of Theorem~\ref{Thm3}~(i) takes the following form in the case of $(SL(n),c_{\rm st})$-bands.

\begin{Thm}\label{Thm.3.4}
The algebra $\R_n^U$ is a cluster algebra with initial seed
\[
 \left(\Gamma,\ \{\De_{[s+n-i,\,s+n-1],\,[1,\,i]} \mid 1\le i\le n-1,\ s\in\Z\}\right),
\]
where the quiver $\Gamma$ is obtained by taking the semi-infinite upper black part of the quiver $\Gamma_{c_{st}}$
and continuing it in the same way down to $-\infty$.
\end{Thm}

Figure~\ref{Fig3a} displays a finite segment of the initial seed of Theorem~\ref{Thm.3.4} in the case of $G=SL(3)$.
In this picture, we use the shorthand notation $\De_{[s+n-i,\,s+n-1]}$ for the flag minor $\De_{[s+n-i,\,s+n-1],\,[1,\,i]}$. 

Note that the cluster sub-algebra of $\R_n^{U}$ supported on the finite sub-cluster with vertices
\[
\{\De_{[s+n-i,\,s+n-1],\,[1,\,i]} \mid 1\le i\le n-1,\ 1+i-n\le s\le 1\in\Z\}
\]
where $\De_{[1,\,i],\,[1,\,i]}$ and $\De_{[n-i+1,\, n],\,[1,\,i]}$ are frozen,
recovers the classical cluster structure on the basic affine space $SL(n)/U$ (see \cite[\S6.5]{FWZ}).

\subsection{$G$-invariants}

We now consider the subalgebra $\R^G$ of $SL(n)$-invariant functions on $\B_n$.

For instance, for every subset $P$ of $\Z$ with cardinality $n$, the maximal minor $\De_{P,[1,\,n]}$
is an element of $\R_n^G$. We will use the shorthand notation $\De_P := \De_{P,[1,\,n]}$. 
In particular the following functions
\[
 \psi^{(s)}_i := \Delta_{[s,s+n]\setminus \{s+i\}},\qquad (1\le i \le n-1,\ s\in \Z)
\]
belong to $\R^G$. More generally, the functions
\[
 \psi^{(s)}_{i,k} := \Delta_{[s,\,s+i-1]\sqcup [s+i+k,\,s+n+k-1]},\qquad (1\le i \le n-1,\ s\in \Z,\ k\ge 1)
\]
belong to $\R^G$.

On the other hand, the regular functions $\theta^{(s)}_{i,k}$ defined at the level of points by
\[
\theta^{(s)}_{i,k}(\mathrm{B}) := \De_{\varpi_i,\varpi_i}(\mathrm{B}(s)\mathrm{B}(s+k)^{-1}),\qquad \qquad (1\le i \le n-1,\ s\in \Z,\ k\ge 1) 
\]
are clearly $G$-invariant. We abbreviate $\theta^{(s)}_i := \theta^{(s)}_{i,1}$.

\begin{Lem}\label{Lem3.5}
For every $1\le i \le n-1$, $s\in \Z$, $k\ge 1$, we have 
\[
 \psi^{(s)}_{i,k} = \theta^{(s)}_{i,k}. 
\]
\end{Lem}

\begin{proof}
Recall that $\mathrm{B}(s)\mathrm{B}(s+1)^{-1}\in U(c_{st}^{-1})\overline{c}_{st}$. Hence there 
exists $a_1,\ldots,a_{n-1}\in K$ such that 
\begin{equation}\label{eq.10}
\mathrm{B}(s)\mathrm{B}(s+1)^{-1} = 
\pmatrix
{
 a_1 & a_2  &\cdots & a_{n-1} &(-1)^{n-1}\cr
 1 & 0  & \cdots & 0 &0\cr
 0 & 1  & \cdots & 0& 0\cr 
 \vdots&\vdots&\ddots&\vdots & \vdots\cr
 0 & 0 & \cdots & 1& 0
}.
\end{equation}
Clearly we have $\theta^{(s)}_{i}(\mathrm{B}) = (-1)^{i-1}a_i\ (1\le i\le n-1)$.
On the other hand, Eq.(\ref{eq.10}) shows that the $a_i$'s are solutions of a linear system
whose coefficients are entries of the $(n+1)\times n$ submatrix of $\mathrm{B}$ taken on rows $\{s,s+1,\ldots,s+n\}$.
Taking into account the band condition $\det \mathrm{B}(s+1) = 1$, elementary linear algebra shows that 
$(-1)^{i-1}a_i= \Delta_{[s,s+n]\setminus \{s+i\}}(\mathrm{B}) = \psi^{(s)}_{i}(\mathrm{B})$. This proves the 
Lemma for $k=1$.

The proof for $k>1$ goes as follows. One shows that the functions $\theta^{(s)}_{i,k}$ and $\psi^{(s)}_{i,k}$
satisfy the same system of functional relations:
\[
\theta^{(s)}_{i,k}\theta^{(s+1)}_{i,k} = 
\theta^{(s)}_{i,k+1}\theta^{(s+1)}_{i,k-1} + 
\theta^{(s)}_{i+1,k}\theta^{(s+1)}_{i-1,k},\quad (k\ge 1,\ 1\le i \le n-1,\ s\in \Z),
\]
\[
\psi^{(s)}_{i,k}\psi^{(s+1)}_{i,k} = 
\psi^{(s)}_{i,k+1}\psi^{(s+1)}_{i,k-1} + 
\psi^{(s)}_{i+1,k}\psi^{(s+1)}_{i-1,k},\quad (k\ge 1,\ 1\le i \le n-1,\ s\in \Z),
\]
where we put $\theta^{(s)}_{0,k} = \theta^{(s)}_{n,k} = \theta^{(s)}_{i,0} = 
\psi^{(s)}_{0,k} = \psi^{(s)}_{n,k} = \psi^{(s)}_{i,0} = 1$.

The functional relations for the $\psi$'s can be checked by using the short Pl\"ucker relations satisfied by the maximal minors of $\mathrm{B}$. 
For instance, if $n=3$, $s=1$, and $k=1$, the relation
satisfied by the $\psi$'s can be written as
\[
\De_{134}\De_{245} = \De_{145}\De_{234} + \De_{124}\De_{345}, 
\]
a valid Pl\"ucker relation.
On the other hand 
\[
 \mathrm{B}(0)\mathrm{B}(2)^{-1} =
\pmatrix
{
 \theta_{1,1}^{(0)}(\mathrm{B}) & -\theta_{2,1}^{(0)}(\mathrm{B})  &1\cr
 1 & 0  &0\cr
 0 & 1&   0
 } 
\pmatrix
{
 \theta_{1,1}^{(1)}(\mathrm{B}) & -\theta_{2,1}^{(1)}(\mathrm{B})  &1\cr
 1 & 0  &0\cr
 0 & 1&   0
 }  
\]
and taking the entry $(1,1)$ of this matrix product on both sides we get
\[
\theta^{(0)}_{1,2}(\mathrm{B}) = \left(\theta^{(0)}_{1,1}\theta^{(1)}_{1,1} - \theta^{(0)}_{2,1}\right)(\mathrm{B}) 
\]
as claimed. The general proof of the functional relations for the $\theta$'s will be given
in Proposition~\ref{Prop.7.1}.

Clearly these relations allow to calculate all the $\theta^{(s)}_{i,k}$ as 
polynomials in the $\theta^{(s)}_{i,1}$, and similarly for the $\psi$'s. 
Hence the relations $\theta^{(s)}_{i,1}=\psi^{(s)}_{i,1}$ imply that $\theta^{(s)}_{i,k}=\psi^{(s)}_{i,k}$ 
for every $k>1$.
\cqfd
\end{proof}

\begin{figure}[t!]
\[
\xymatrix@-1.0pc{
{\De_{023}}\ar[d]& \ar[l] \De_{013}
\\
{\De_{034}}\ar[r] &  \De_{-103}\ar[u]\ar[d]
\\
{\De_{-134}}\ar[d]\ar[u]&\De_{-104}\ar[l]
\\
\De_{-145} \ar[r]&\De_{-2-14}\ar[u]\ar[d]
\\
{\De_{-245}}\ar[d]\ar[u]&\De_{-2-15}\ar[l]
\\
{}\save[]+<0cm,0ex>*{\vdots}\restore&{}\save[]+<0cm,0ex>*{\vdots}\ar[u]\restore  
}
\]
\caption{\label{Fig3b} {\it A finite segment of an initial seed of the cluster algebra structure on $\R_3^G$}}
\end{figure}
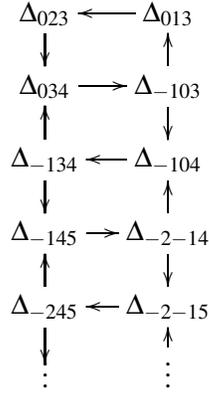

The statement of Theorem~\ref{Thm3}~(ii) takes the following form in the case of $(SL(n),c_{\rm st})$-bands.

\begin{Thm}\label{Thm.3.5}
The algebra $\R_n^G$ is a polynomial ring, namely
\[
 \R_n^G = K[\theta^{(s)}_i\mid 1\le i \le n-1,\ s\in \Z].
\]
It has a cluster algebra structure with a class of initial seeds whose cluster variables are all of the form 
$\theta^{(s)}_{i,k}$ for suitable integers $s,i,k$. One remarkable initial seed is
\[
 \left(\Lambda,\ \{\theta^{(s(i,k))}_{i,k} | 1\le i \le n-1,\ k>0 \}    \right)
\]
where $s(i,k) = 1 -\lfloor\frac{k+1}{2}\rfloor$ (\resp $s(i,k) = -\lfloor\frac{k}{2}\rfloor$) if $i$ is odd (\resp even),
and $\Lambda$ is the quiver with arrows
\begin{eqnarray*}
&&\theta^{(s(i,k))}_{i,k} \to \theta^{(s(i\pm 1,k))}_{i\pm 1,k} \mbox{ if $i+k$ is odd and $i\pm1 \in [1,n-1]$},\\ 
&&\theta^{(s(i,k))}_{i,k} \to \theta^{(s(i,k + 1))}_{i,k + 1} \mbox{ if $i+k$ is even},\\
&&\theta^{(s(i,k))}_{i,k} \to \theta^{(s(i,k - 1))}_{i,k - 1} \mbox{ if $i+k$ is even and $k>1$}.
\end{eqnarray*}
\end{Thm}

Figure~\ref{Fig3b} displays the top part of the initial seed of Theorem~\ref{Thm.3.5} in the case of $G=SL(3)$.
Using Lemma~\ref{Lem3.5}, we have written the initial cluster variables as maximal minors $\Delta_P$.

Note that, in contrast with the quivers $\Gamma_{\widetilde{c}}$ and $\Gamma$, the quiver $\Lambda$ is only half-infinite.

\begin{remark}
{\rm
In the case $n=2$, the cluster algebra $\R_2^G$ is a coefficient-free version of the cluster structure
on the homogeneous coordinate ring  of the infinite Grassmannian $\mathrm{Gr}(2,\pm\infty)$ 
studied by Grabowski and Gratz \cite{GG}, 
associated with a locally finite triangulation of the infinity-gon (see \cite[Theorem 3.16]{GG}). 
In our case, the coefficients $\De^{i,i+1}$ of \cite{GG} are equal to 1 because of the definition of
$\B_2$.

Let $X_2$ be the affine cone over the infinite Grassmannian $\mathrm{Gr}(2,\pm\infty)$ arising from the Pl{\"u}cker embedding, as considered in \cite[Appendix A]{GG}.
Geometrically, the previous remark reflects the fact that the scheme $\B_2//G= \Spec(\R_2^{G})$ can be realised as the closed subscheme of $X_2$ defined by the vanishing of the functions $\De^{i,i+1}-1 \ ( i \in \Z)$.

For general $n$, the cluster algebra $\R_n^G$ is a coefficient-free version of the cluster algebra
on the homogeneous coordinate ring of the infinite Grassmannian $\mathrm{Gr}(n,\pm\infty)$ recently studied in \cite{ACFGS} using additive categorification techniques.
The above geometric description of the scheme $\B_2 //G$ extends naturally to the case of a general $n$.
}
\end{remark}

\begin{remark}
{\rm
Clearly, $\R_n^G\subset \R_n^U$. So any element of $\R_n^G$ can be written as a Laurent polynomial 
in the initial cluster variables $\De_{[s,s+i-1],[1,i]} = \De^{(s)}_{\varpi_i,\varpi_i}$ of $\R_n^U$. 
For example, for $n=3$, an easy calculation shows that we can write $\theta^{(0)}_1 = \Delta_{023}$ as
\[
\Delta_{023} = \frac{\Delta_0}{\Delta_1} + \frac{\Delta_2}{\Delta_1}\frac{\Delta_{01}}{\Delta_{12}}
+ \frac{\Delta_{23}}{\Delta_{12}},
\]
that is,
\[
\theta^{(0)}_1 = \frac{\De_{\varpi_1,\varpi_1}^{(0)}}{\De_{\varpi_1,\varpi_1}^{(1)}} 
+ \frac{\De_{\varpi_1,\varpi_1}^{(2)}}{\De_{\varpi_1,\varpi_1}^{(1)}}
\frac{\De_{\varpi_2,\varpi_2}^{(0)}}{\De_{\varpi_2,\varpi_2}^{(1)}}
+ \frac{\De_{\varpi_2,\varpi_2}^{(2)}}{\De_{\varpi_2,\varpi_2}^{(1)}}.
\]
For an interpretation of such formulas in terms of representations of quantum affine algebras see 
Remark~\ref{Rem.8.4} below.
} 
\end{remark}

\section{Bands}
\label{sec bands}

In this section we define bands for an arbitrary group $G$ and an arbitrary Coxeter element $c$. 
We first prepare some general results on limits of schemes over inverse systems.

\subsection{Limits and pro-varieties}
\label{sec:some limits}

We fix here our notation for limits of schemes and pro-objects. A reference on limits is the book \cite{Mac}. 

Let $A$ be a pre-ordered set and $X_\bullet$ be an inverse system of $K$-schemes over $A$. 
In other words, $X_\bullet$ is a collection of $K$-schemes $X_a \ (a \in A)$, and of morphisms of $K$-schemes $\pi_{b;\ a} : X_b \lra X_a \ (a \leq b)$. 
A \textit{limit} of $X_\bullet$ is a pair $(X, \pi)$ where $X$ is a $K$-scheme and $\pi=(\pi_a)_{a \in A}$ is a collection of morphisms $\pi_a : X \lra X_a \ (a \in A)$ such that
$$
\pi_{b; \ a} \circ \pi_b  =\pi_a \quad (a \leq b),
$$
and satisfying the following condition.  

\medskip

\noindent \textbf{Universal property of limits.}
\textit{For every $K$-scheme $Z$, and for every collection of morphisms $x_a: Z \lra X_a \ (a \in A)$ such that $\pi_{b; \ a} \circ x_b= x_a \ (a \leq b)$, there exists a unique morphism of $K$-schemes $x: Z \lra X$ satisfying the conditions $\pi_a \circ x= x_a \ (a \in A).$
    }

\medskip

The previous universal property can be rephrased by saying that the composition with the components of $\pi$ defines a natural bijection
\begin{equation}
\label{eq: functor of limit}
X(Z) \simeq \{ x= (x_a)_{a \in A} \mid  x_a \in X_a(Z),  \ {\rm{ and }} \   \pi_{b;\ a} \circ x_b = x_a \  {\rm{whenever}} \ a \leq b \}, \qquad (Z \in \Schk).
\end{equation}
In other words, the collection of morphisms $\pi=(\pi_a)_{a \in A}$ is a universal object for the functor $(\Schk)^{op} \lra \Set$ described by the right hand side of Eq. (\ref{eq: functor of limit}). 
If a limit of the inverse system $X_\bullet$ exists, then it is unique up to isomorphism. 
 Assume that the limit of the system $X_\bullet$ exists and denote it by $(X, \pi)$.
 We refer to the morphisms $\pi_a \ (a \in A)$ as the canonical morphisms of the limit structure of $X$. 
 Most of the times we will let the data of these morphisms implicit and we will just say that $X$ is the limit of the system $X_\bullet.$

\medskip

In this paper, a \textit{pro-variety} is a $K$-scheme which is the limit of an inverse system of algebraic varieties over a \textit{directed set},  that is, a pre-ordered set in which every finite subset has an upper bound. 

\begin{remark}
{\rm
    \label{rem: cofinal subset} A subset $A'$ of $A$ is \textit{cofinal} in $A$ if for every element $a \in A$ there exists $a' \in A'$ such that $a \leq a'$. 
    If $A'$ is a cofinal subset of $A$ and $A$ is a directed set, then $A'$ is also a directed set. 
    Moreover, we have that the inverse system $X_\bullet$ admits a limit over $A$ if and only if it admits a limit over $A'$, and in this case the two limits are canonically isomorphic.
    Indeed, the functor described by the right hand side of Eq. (\ref{eq: functor of limit}) is canonically isomorphic to the same functor obtained by replacing $A$ with $A'$, by means of the natural transformation that forgets the components (of points) not in $A'$.
    }
\end{remark}

Form now on, assume that the limit $X$ of the system $X_\bullet$ exists, and let $Y_\bullet$ be a second  inverse system of $K$-schemes over $A$ with limit $Y$.
 Let $\psi_\bullet: X_\bullet \lra Y_\bullet$ be a morphism of inverse systems. 
 In other words, $\psi_\bullet$ is a collection of morphisms of $K$-schemes $\psi_a :X_a \lra Y_a \ (a \in A)$ such that, for every $a \leq b$, the diagram 
\[\begin{tikzcd}
	{X_b} & {X_a} \\
	{Y_b} & {Y_a}
	\arrow[from=1-1, to=1-2]
	\arrow["{\psi_b}"', from=1-1, to=2-1]
	\arrow["{\psi_a}", from=1-2, to=2-2]
	\arrow[from=2-1, to=2-2]
\end{tikzcd}\]
is commutative. 
The horizontal maps of the square above are the ones defining the two inverse systems. 
The morphism of inverse systems $\psi_\bullet$ induces a morphism of schemes $\psi: X \lra Y$. 
At the level of the functors of points described through Eq. (\ref{eq: functor of limit}), $\psi$ is given as follows:
$$
\begin{array}{rcl}
   X(Z)  & \lra & Y(Z)  \\
    x=(x_a)_{a \in A} & \longmapsto & \psi(x)=(\psi_a \circ x_a)_{a \in A} 
\end{array} \qquad (Z \in \Schk).
$$

\subsection{The affine case and infinite products}
From now on, we assume that $A$ is a directed set and we take a closer look at the case where the schemes $X_a$ are affine. 
In particular, we assume that $X_a= \Spec(R_a) \ (a \in A)$ for some $K$-algebra $R_a$. 
The collection of homomorphisms 
$$
\pi_{b; \ a}^*: R_a \lra R_b \qquad (a \leq b)
$$
is a direct system  of $K$-algebras over $A$. 
Since the poset $A$ is a directed set, the colimit of this system exists in the category of $K$-algebras.
We denote this colimit by $R$, and we let $\pi_a^*: R_a \lra R \ (a \in A)$ be the canonical morphisms of the colimit structure of $R$.
 We refer to \cite[Chapter II, Ex. 14-23]{AM} for more details on the construction and the properties of colimits of rings over directed sets.
 
Then, we have that $X:= \Spec(R)$ is the limit of the inverse system $X_\bullet$, and the canonical morphism of the limit structure of $X$ are the morphisms $\pi_a : X \lra X_a  \ (a \in A)$ induced by the homomorphisms $\pi_a^*$ through the $\Spec$ functor.
We collect below some easy properties that will be of importance for the rest of the paper.

\begin{Lem}
    \label{lem: limits algebraic}
    Let $X_\bullet$ be an inverse system of affine $K$-schemes as above, and let $A' \subseteq A$ be a cofinal subset. The following hold.
    \begin{enumerate}
        \item If the rings $R_a \ (a \in A)$ are integral domains, then $R$ is an integral domain.
        \item If the homomorphisms $\pi_{b; \ a}^* \ (a \leq b)$ are injective, then the homomorphisms $\pi_a^* \ (a \in A)$ are injective.
        \item The ring $R$ is the union of the subalgebras $\pi_{a'}^*(R_{a'})$ for $a' \in A'.$
    \end{enumerate}
\end{Lem}

\begin{proof}
    Statement one is \cite[Chapter II, Ex. 22]{AM}.
    For statement 2, let $a \in A$ and $r_a \in R_a$ such that $\pi_a^*(r_a)=0$.
    By \cite[Chapter II, Ex. 15]{AM} there exists $b \in A$ such that $b \geq a$ and $\pi_{b; \ a}^*(r_a)=0.$
    Statement two follows.
    For statement three, let $r \in R$. 
    By \cite[Chapter II, Ex. 15]{AM} we have that $r=\pi_a^*(r_a)$ for some $a \in A$ and $r_a \in R_a$. 
    Let $a' \in A$ such that $a' \geq a$. 
    Then, we have that $r= \pi_{a'}^* ( \pi_{a'; \ a}^*(r_a))$, which completes the proof of the lemma.
    \cqfd
\end{proof}

Next, assume that $Y_\bullet$ is a second inverse system of affine $K$-schemes over $A$ with limit $Y$.
Let $\psi_\bullet: X_\bullet \lra Y_\bullet$ be a morphism of inverse systems.

\begin{Lem}
\label{lem: limit closed}
If the morphisms $\psi_a : X_a \lra Y_a \ (a \in A)$ are closed embeddings, then the morphism $\psi: X \lra Y $ induced by $\psi_\bullet$ is a closed embedding.
\end{Lem}

\begin{proof}
    Assume that $Y_a= \Spec(S_a) \ (a \in A)$ and let $S$ be the colimit of the direct system of algebras $S_\bullet.$ 
    By hypothesis, the pullback homomorphisms $\psi_a^* : S_a \lra R_a \ (a \in A)$ are surjective. 
   Then, \cite[Chapter II, Ex. 18-19]{AM} imply that the pullback homomorphism $\psi^*:S \lra R$ is surjective. Hence, $\psi$ is a closed embedding.
    \cqfd
\end{proof}

We conclude this section with few reminders on a prototypical example of pro-objects: infinite products of affine schemes.  
Let $J$ be an infinite set and $\P_f(J)$ be the poset of finite subsets of $J$ ordered by inclusion.
Let $R_j \ (j \in J)$ be a collection of $K$-algebras, and let $X_j:= \Spec(R_j) \ (j \in J)$ be the associated collection of affine $K$-schemes.
Let us consider the data
$$
X_S:= \prod_{j \in S}X_j, \qquad R_S:= \bigotimes_{j \in S}R_j \qquad(S \in \P_f(J)),
$$
so that $X_S= \Spec(R_S).$
To any $S \subseteq S' \subseteq J$ we can associate the natural map $\pi_{S'; \ S} : X_{S'} \lra X_S$ that forgets the components (of points) not in $S$.
The previous data is an inverse system of affine schemes over $\P_f(J)$. 
Its limit is an affine scheme that we denote by $X_J$ and that we call the \textit{product} of the $X_j$.
This terminology is justified by the fact that the functor of points of $X_J$ is the product of the functor of points of the $X_j$. 
More precisely, we have that 
\begin{equation}
    \label{eq: functor product}
    X_J(Z)= \{ x=(x_j)_{j \in J} \mid x_j \in X_j(Z)\} \qquad (Z \in \Schk).
\end{equation}
The natural morphisms $\pi_S: X_J \lra X_S \ (S \subseteq J)$ and $\pi_j : X_J \lra X_j \ (j \in J)$ of the limit structure of $X_J$ are the maps that forget the appropriate components of points.
Moreover, the pullback homomorphism $\pi^*_{S'; \ S} : R_S \lra R_{S'} \ (S \subseteq S')$ are the obvious maps. 
They send an element to itself tensored with the unities of the rings $R_j \ (j \in S' \setminus S)$.
The colimit $R$ of the direct system of algebras $R_\bullet$ is called the \textit{tensor product} of the algebras $R_j$.

Finally, assume that the schemes $X_j \ (j \in J)$ are varieties.
Then, the schemes $X_S \ (S \in \P_f(J))$ are also varieties.
Thus, $X_J$ is a pro-variety. 
Moreover, the homomorphisms  $\pi^*_{S'; \ S} \ (S \subseteq S')$ are injective. Therefore, because of Lemma \ref{lem: limits algebraic}, the algebra $R$ can be thought as the union of the finite tensor products of the algebras $R_j$.

\subsection{Finite $(G,c)$-bands}

Before giving the definition of the spaces of finite bands, we recall few geometric notions.
A morphism of $K$-schemes $\iota: X \lra Y$ is a \textit{monomorphism} if for every $K$-scheme $Z$ the function $X(Z) \lra Y(Z)$ sending a point $x$ to $\iota \circ x$ is injective. 
It is well known that closed embeddings are monomorphisms. 
Therefore, whenever a closed embedding $\iota:X \lra Y$ is given, the functor of points of $X$ can be naturally identified with a subfunctor of the one of $Y$ by means of $\iota$.

\medskip

Let us consider the varieties
$$
G(M,N):= \prod_{M \leq s \leq N} G \qquad (M,N \in \Z),
$$
and their subvarieties
$$
Z(G,c,M,N):= \prod_{M \leq s \leq N} \bigl( \Ucc \bigr). 
$$
We have that $G(M,N)$ (\resp $Z(G,c,M,N)$) is the product of $N-M+1$ copies of $G$ (\resp $\Ucc$) if $M \leq N$ and it is equal to $\Spec(K)$ otherwise. 
Assume from now on that $M \leq N$.
Let us consider the morphisms of algebraic varieties 
$$
\sigma_{M,N} : G(M,N) \lra G(M, N-1) 
$$
sending a point $g=(g(k))_{M \leq k \leq N}$ of $G(M,N)$ to the point $\sigma_{M,N}(g)$ defined by 
$$
\sigma_{M,N}(g)(s)= g(s) g(s+1)^{-1}, \qquad (M \leq s \leq N-1).
$$

\begin{Def}
    \label{def:finite bands}
    Let $M\leq N$. 
    The space of finite $(G,c,M,N)$-bands $B(G,c,M,N)$ is the schematic preimage of $Z(G,c,M,N-1)$ under the morphism $\sigma_{M,N}.$
\end{Def}
From its definition, we have that $B(G,c,M,N)$ is a closed subscheme of $G(M,N)$, and it is therefore an affine scheme. 
Its functor of points is naturally described as a subfunctor of the functor of points of $G(M,N)$ through the natural inclusion.
In particular, for any $K$-scheme $Z$ we have that the set  $B(G,c,M,N)(Z) $ is given by the following formula 
\begin{equation}
   \label{eq:points part bands}
\begin{array}{rrc}
    \bigl\{ g=\bigl (g(k) \bigr)_{M \leq k \leq N} \mid & g(k) \in G(Z), \ g(s)g(s+1)^{-1} \in (\Ucc)(Z) \\[0.5em]
     &  (M \leq k \leq N, \ M \leq s \leq N-1)& \hspace{-0.9em}\bigr\}.
\end{array}
\end{equation}

\begin{example}
    {\rm We consider two simple but instructive examples. Let $s \in \Z$.
    \begin{enumerate}
        \item  The space of finite bands $B(G,c,s,s)$ is isomorphic to $G$. 
        Indeed, we have that 
        $$
        Z(G,c,s,s-1)= \Spec(K)= G(s,s-1).
        $$
        \item   A $Z$-point $\bigl(g(s-1), \ g(s), \ g(s+1) \bigr)$ of $B(G,c,s-1, s+1)$ is completely determined by the the data of $g(s-1)g(s)^{-1}$, $g(s)$ and $g(s)g(s+1)^{-1}.$ 
       More precisely, the morphism 
       $$
       B(G,c,s-1,s+1) \lra \Ucc \times G \times U(c^{-1})\overline{c}
       $$
       sending a $Z$-point as above to the $Z$-point $\bigl( g(s-1)g(s)^{-1}, \ g(s), \ g(s)g(s+1)^{-1} \bigr)$ is an isomorphism.
        Indeed, its inverse is the morphism sending a $Z$-point 
        $$
        \bigl( x, g,y) \in (\Ucc)(Z) \times G(Z) \times (\Ucc)(Z)
        $$ to $(xg, \ g, \ y^{-1}g).$
    \end{enumerate}
    }
\end{example}

Generalising the previous examples, we prove below that the spaces of finite $(G,c)$-bands are isomorphic (as schemes) to the product of a copy of $G$ and a finite number of copies of $U(c^{-1}) \overline{c}$.

\medskip

First, observe that $B(G,c,M,N)$ and $B(G,c,M',N')$ are naturally isomorphic if $N-M=N'-M'$. 
Therefore, when showing geometric properties of the spaces of finite bands we can always reduce to the case where $M \leq 0 \leq N.$
Let us consider the varieties defined by the formula
$$
X_n:=
\left\{
\begin{array}{ll}
   G  & {\rm if} \qquad n=0,  \\[0.5em]
   \Ucc  &  {\rm if} \qquad n \in \Z \setminus \{0\},
\end{array} \right.
$$
and their products
$$
X(G,c,M,N):= \prod_{M \leq s \leq N} X_s \qquad(M \leq 0 \leq N).
$$
For a $K$-scheme $Z$ and a point $x=\bigl( x(s) \bigr)_{M \leq s \leq N} \in X(G,c,M,N)(Z)$, we define the following elements of $G(Z)$:
\begin{equation}
\label{eq: products of x}
    x(0\mid s):= \prod_{0< k \leq s} x(k), \qquad x(s \mid 0):= \prod_{s \leq k < 0} x(k) \qquad (M \leq s \leq N).
\end{equation}
We stress that the products in Eq. (\ref{eq: products of x}) are taken in increasing order. 
For example, we have that 
$$
x(0 \mid 3)= x(1)x(2)x(3), \qquad x(-2 \mid 0)= x(-2)x(-1), \quad x(4 \mid 0)= e, \qquad x(0 \mid -3)=e,
$$
where $e$ denotes the identity of the group $G(Z)$. 
Observe that the condition $x(0 \mid s) \neq e$ implies $s\geq 1$. 
Similarly, if $x(s \mid 0) \neq e$ then $s \leq -1$.
In particular, for every $s \in \Z$, at most one between $x(0 \mid s)$ and $x(s \mid 0)$ is not the identity.

Then, let 
$$ 
\tau_{M,N} : X(G,c,M,N) \lra B(G,c,M,N)
$$
be the morphism sending a $Z$-point $x$ of $X(G,c,M,N)$ to the point $\tau_{M,N}(x)$ of $B(G,c,M,N)$ defined by the formula:
\begin{equation}
    \label{eq: tau MN}
    \tau_{M,N}(x)(s):= x(s \mid 0) x(0 \mid s)^{-1} x(0).
\end{equation}

\begin{Lem}
    \label{lem: partial bands aff space}
    The morphism of $K$-schemes $\tau_{M,N} \ (M \leq 0 \leq N)$ is an isomorphism.
\end{Lem}

\begin{proof}
The function induced by the morphism $\tau_{M,N}$ on the $Z$-points $(Z \in \Schk)$ is bijective. 
Its inverse is the function sending a $Z$-point $g=(g(s))_{M \leq s \leq N}$ to the point $x$ defined by
$$
x(s)= \left\{
\begin{array}{ll}
 g(s)g(s+1)^{-1} &  {\rm if} \qquad s<0,\\
    g(0) & {\rm if} \qquad s=0,  \\
    g(s-1)g(s)^{-1} &  {\rm if} \qquad s>0.\\
\end{array}
\right.
$$
    \cqfd
\end{proof}

\begin{remark}
\label{rem: zero compo tau}
{\rm
Notice that the conditions satisfied by the components of points of  $B(G,c,M,N)$, viewed as a subscheme of $G(M,N)$, are invariant under translation (see Eq. (\ref{eq:points part bands})).
The morphisms $\tau_{M,N}$ break this symmetry by privileging the zero component. 
Indeed, the zero-component of a point $x$ of $X(G,c,M,N)$ is the same as the  one of $\tau_{M,N}(x)$.
Obviously, this choice of the zero component as privileged is arbitrary. Any other component could replace the zero-one by appropriately modifying the morphism $\tau_{M,N}$.
}

\end{remark}

\begin{Cor}
    \label{cor: geometric prop of finite bands}
Let $M \le N$. The scheme  $B(G,c,M,N)$ is an irreducible smooth affine variety of dimension $\dim(G) + (N-M) |I|$.
Its coordinate ring $R(G,c,M,N)$ is isomorphic to a polynomial ring  in $(N-M)  |I|$ variables with coefficients in $K[G]$. 
In particular, $R(G,c,M,N)$ is a unique factorization domain and the isomorphism class of the variety $B(G,c,M,N)$ does not depend on the Coxeter element $c$.
\end{Cor}
\begin{proof}
We can assume without loss of generality that $M \leq 0 \leq N$. 
Since the variety $U(c^{-1})\oc$ is isomorphic to an affine space $\A^{|I|}$ of dimension $|I|$ (see \S\ref{sec: alg gps}), Lemma \ref{lem: partial bands aff space} implies that the scheme $B(G,c,M,N)$ is isomorphic to $G \times \A^{(M-N) |I|}$. The latter scheme is an irreducible smooth affine variety  of dimension $\dim(G) + (N-M) |I|$, and its coordinate ring is a polynomial ring in $(N-M)|I|$ variables with coefficients in $K[G].$
Moreover, the coordinate ring $K[G]$ of $G$ is a  unique factorization domain as $G$ is simply connected \cite[Corollary p.~303]{P}.
The ring $R(G,c,M,N )$ is then a unique factorization domain as it is isomorphic to a polynomial ring over $K[G]$.
\cqfd

\end{proof}

For any triple of integers $M \leq s \leq N$, let us consider the natural morphism 
\[
\pi_s: B(G,c,N,M) \lra G
\]
that projects a point of the space of $(G,c,M,N)$-bands to its $s$-component.
Notice that in the notation $\pi_s$ we drop the dependence on $N$ and $M$.
\begin{Lem}
    \label{lem: generators part bands}
   Let $M \le N$.
   The pullback homomorphisms $\pi_s^* : K[G] \lra R(G,c,M,N) \ (M \leq s \leq N)$ are injective.
   The $K$-algebra $R(G,c,M,N)$ is generated by its subalgebras $\pi_s^*(K[G])$.
\end{Lem}

\begin{proof}
    Assume without loss of generality that $M\leq 0 \leq N$.
    Consider the morphism $G \lra X(G,c,M,N)$ sending a point $g$ of $G$ to the point of $X(G,c,M,N)$ whose zero component is $g$ and whose non-zero components are $\overline{c}$. 
    Let $\iota_0$ be composition of the previous morphism with $\tau_{M,N}$. 
    We clearly have that $\pi_0 \circ \iota_0$ is the identity of $G$.
    This implies that $\pi_0^*$ is injective. 
    The fact that the homomorphism $\pi_s^*$ is injective for a general $s$ can be deduced similarly (see Remark \ref{rem: zero compo tau}).
    Then, notice that the pullback homomorphism of regular functions under the natural inclusion $B(G,c,M,N) \lra G(M,N)$ is surjective.
    Indeed, this property holds since $B(G,c,M,N)$ is a closed subscheme of the affine variety $G(M,N)$.
    But the coordinate ring of $G(M,N)$ is the tensor product of $M-N+1$ copies of $K[G]$, and these copies pullback to the subalgebras $\pi_s^*(K[G]) \ (M \leq s \leq N)$.
    Thus, we deduce that $R(G,c,M,N)$ is generated by its subalgebras $\pi_s^*(K[G]).$
    \cqfd
\end{proof}

\subsection{$(G,c)$-bands}
\label{sec: Gc bands}
To define the space of $(G,c)$-bands $B(G,c)$ it is practical to introduce a natural inverse system of $K$-schemes involving the spaces of finite bands. 
The limit of this system will be the scheme $B(G,c)$.  

\medskip

Let $\P:= \{ (M,N) \in \Z^2 \mid M \le N\}$
endowed with the poset structure defined by $(M,N) \leq (M',N')$ if and only if $[M,N] \subseteq [M',N']$.
Notice that $\P$ is a directed set. 
It can be seen as a cofinal subset of the poset of finite subsets of $\Z$ by assigning to an element $(M,N)$ of $\P$ the finite set $[M,N] \cap \Z.$

\medskip

Consider the natural projection morphisms $ G(M',N') \lra G(M,N) \ \bigl( (M,N) \leq (M',N') \bigr)$. 
These maps form an inverse system of algebraic varieties over the poset $\P$. 
Since $\P$ is cofinal in the poset of finite subsets of $\Z$, the limit $G(\infty)$ of the varieties $G(M,N)$ is the product of countably many copies of $G$ (see Remark \ref{rem: cofinal subset}).
The projection $G(M',N') \lra G(M,N)$ induces by restriction a morphism
$$
\pi_{M',N'; \ M,N} : B(G, c, M',N') \lra B(G,c,M,N).
$$

The collection of morphisms $\pi_{M',N'; \ M,N} \ \bigl( (M,N) \leq (M',N') \bigr)$ turns the spaces of finite bands into an inverse system of affine $K$-schemes over $\P$.

\begin{Def}
   The limit of the inverse system of the schemes of finite bands is the space of $(G,c)$-bands, denoted by $B(G,c)$.  
  The scheme $B(G,c)$ is an affine pro-variety whose coordinate ring is denoted by $R(G,c)$.
\end{Def}

The collection of the natural inclusion morphisms  $B(G,c,M,N) \lra G(M,N) \  (M \leq N)$ is a morphism of inverse systems of $K$-schemes over $\P$. 
Therefore, we can consider the induced limit morphism $\iota : B(G,c) \lra G(\infty).$

\begin{Lem}
    \label{lem: bands closed in Ginf}
    The morphism $\iota$ is a closed embedding. 
\end{Lem}

\begin{proof}
This follows directly from Lemma \ref{lem: limit closed}.
    \cqfd
\end{proof}

By Lemma \ref{lem: bands closed in Ginf}, we can identify the functor of points of $B(G,c)$ as a subfunctor of the one of $G(\infty)$ by means of the morphism $\iota$.
In particular, by means of this identification, we have that for any $K$-scheme $Z$:
\begin{equation}
    \label{eq: points bands}
       B(G,c)(Z) = \bigl\{ g=\bigl (g(s) \bigr)_{s \in \Z} \mid  g(s) \in G(Z), \ g(s)g(s+1)^{-1} \in (U(c^{-1}) \overline{c})(Z) \quad (s \in \Z) \bigr\}.
\end{equation}
We denote by $\pi_{M,N} : B(G, c) \lra B(G,c,N,M) \ ( M \leq N)$ the canonical morphisms of the limit structure of $B(G,c)$.
Moreover, as for finite bands, we  consider the projection maps $\pi_s : B(G,c) \lra G \ (s \in \Z)$ sending a point of $B(G,c)$ to its $s$-component.

\begin{Lem}
    \label{lem: generators bands}
    The following hold.
    \begin{enumerate}
        \item The pullback homomorphism $\pi_{M,N}^* : R(G,c,M,N) \lra R(G,c)$ is injective for every $(M,N) \in \P$. 
        \item The $K$-algebra $R(G,c)$ is the union of its subalgebras $\pi_{M,N}^*(R(G,c,M,N))$ for $(M,N) \in \P$ such that $M \leq 0 \leq N.$
        \item The pullback homomorphism $\pi_s^* : K[G] \lra R(G,c)$ is injective for every $s \in \Z$.
        \item The $K$-algebra $R(G,c)$ is generated by its subalgebras $\pi_s^*(K[G])$ for $s \in \Z$.
    \end{enumerate}
\end{Lem}
\begin{proof}
Let $(M,N) \leq (M',N')$ be elements of $\P$. 
We first show that the pullback homomorphism $\pi^*_{M',N'; \ M,N}$ is injective. 
Without loss of generality, we can assume that $M \leq 0 \leq N.$
Let 
$$
p_{M',N'; \ M,N} : X(G,c,M',N') \lra X(G,c,M,N)
$$
be the map that forgets the components (of points) not belonging to the interval $[M,N]$. 
The morphism $\pi_{M',N'; \ M,N}$ identifies with $p_{M',N'; \ M,N}$ through the isomorphisms $\tau_{M',N'}$ and $\tau_{M,N}$.
Since the pullback homomorphism $p^*_{M',N'; \ M,N}$ is obviously injective, it follows that $\pi^*_{M',N'; \ M,N}$ is also injective.
Then, statements 1 and 2 follow from Lemma \ref{lem: limits algebraic}.
Statement 3 is a consequence of statement 1. 
Indeed, recall that $B(G,c,s,s)=G$ and observe that the morphism $\pi_{s,s}: B(G,c) \lra B(G,c,s,s)=G$ is $\pi_s.$
Because of Lemma \ref{lem: bands closed in Ginf}, statement 4 can be proved as its analogue statement in Lemma \ref{lem: generators part bands}, using that the coordinate ring of $G(\infty)$ is the tensor product of countably many copies of $K[G].$
    \cqfd
\end{proof}

Let $\P_0$ be the subset of $\P$ defined by the formula.
\begin{equation}
    \label{eq: def P0}
    \P_0:=\{(M,N) \in \P \mid M \leq 0 \leq N\}.
\end{equation}
Note that $\P_0$ is a cofinal subset of the poset $\P$ and of the poset of finite subsets of $\Z$. 
Hence, taking limits of an inverse systems over $\P_0$, or over $\P$, or over the finite subsets of $\Z$ (whenever this makes sense) does not affect the limit (see Remark \ref{rem: cofinal subset}).

Observe that the varieties $X(G,c,M,N) \ ((M,N) \in \P_0)$ form an inverse system of $K$-schemes over $\P_0$ by means of the projection maps $p_{M',N'; \ M,N} \ ((M,N) \leq (M',N'))$ defined in the proof of Lemma \ref{lem: generators bands}. 
The limit of this inverse system is the product $X(G,c)$ of the varieties $X_n \ (n \in \Z)$. 
Moreover, as remarked during the proof of Lemma \ref{lem: generators bands}, the collection of morphisms $\tau_{M,N} : X(G,c,M,N) \lra B(G,c,M,N) \ ((M,N) \in \P_0)$ is an isomorphisms of inverse systems. 

\begin{Cor}
\label{cor: bands product}
    The morphism $\tau : X(G,c) \lra B(G,c)$ induced by the $\tau_{M,N} \ ((M,N) \in \P_0)$ is an isomorphism.
    In particular, the algebra $R(G,c)$ is isomorphic to a polynomial ring in countably many variables with coefficients in $K[G].$ Hence, the ring $R(G,c)$ is a unique factorization domain.
\end{Cor}

\begin{proof}
    The fact that $\tau$ is an isomorphism is obvious as it is the limit of isomorphisms.
    Since the coordinate ring of $U(c^{-1}) \overline c$ is a polynomial ring in $|I|$ variables, it follows that $R(G,c)$ is a polynomial ring in countably many variables with coefficients in $K[G]$. 
    Since the ring $K[G]$ is a unique factorization domain, the analogue properties of $R(G,c)$ follow.
\end{proof}

\subsection{The case of $SL(n)$ and $c_{st}$ revisited}
\label{Sect-4.5}

In this section, $G$ denotes the algebraic group $SL(n) \ (n \in \N_{\geq 2})$ and we fix $c=\wt c= c_{st}=s_1 \cdots s_{n-1}.$
 We prove that the scheme $\B_n$ described in \S\ref{sec:SL(n) cst} can be naturally identified with the scheme of $(SL(n), c_{st})$-bands defined in \S\ref{sec: Gc bands}.
 For this, it is sufficient to extend the discussion of  \S\ref{sec:SL(n) cst} on $K$-rational points to  the level of the functor of points. 
 We start by recalling a well known fact about schemes. 
 
 Let $\Kalg$  be the category of associative, unitary and commutative $K$-algebras. 
A $K$-scheme $X$ can be seen as a functor $X : \Kalg \ra \Set$ by composing the functor of points of $X$ with the $\Spec$ functor.
This procedure gives a fully faithful embedding of the category of $K$-schemes in the category of functors from $\Kalg$ to $\Set$.  
 In particular, to show that two $K$-schemes $X$ and $Y$ are isomorphic, it is sufficient to exhibit a natural isomorphism between the functors $X: \Kalg \ra \Set$ and $Y : \Kalg \ra \Set.$

  Let $R$ be a $K$-algebra.  
Notice that the set $\B_n(R)$ consists of $(\infty \times n)$-arrays $\mathrm{B}$ as in Eq.~(\ref{eq:inf arrays1}) where the field $K$ is replaced by $R$, and satisfying the conditions of Eq.~(\ref{eq:inf arrays2}).
Therefore, the natural transformation $\iota_n $ defined by
$$
\begin{array}{rclc}
    \iota_n : \B_n(R)  & \lra &  \displaystyle G(\infty)(R) & (R \in \Kalg)  \\[0.5em]
      \mathrm{B} & \longmapsto & (\mathrm{B}(s))_{s \in \Z} 
\end{array}
$$
is well defined. Recall that we denote $G(\infty)= \prod_{s \in \Z}G$.

\begin{Lem}
    \label{lem: SLn and general bands}
    The image of  $\iota_n$ is $B(SL(n), c_{st})$. 
    The natural transformation $\bar \iota_n : \B_n \ra B(SL(n), c_{st})$ induced by $\iota_n$ is an isomorphism.
\end{Lem}

\begin{proof}
    let $R$ be a $K$-algebra. Recall that, by \S \ref{sec:first decription}, two matrices $g,h \in SL(n,R)$ satisfy the property that $gh^{-1}$ belongs to $(U(c_{st}^{-1}) \oc_{st})(R)$ if and only if the last $n-1$ rows of $g$ coincide with the first $n-1$ rows of $h$.
This implies that the image of the function $\iota_n : \B_n(R) \lra G(\infty)(R)$ is the subset $B(SL(n), s_{st})(R)$ of $G(\infty)(R)$.
Since $\iota_n$ is obviously injective, the statement follows.
\cqfd
\end{proof}

\begin{remark}
{\rm
    The previous discussion can be easily adapted for finite $(SL_n, c_{st})$-bands.
    }
\end{remark}

\subsection{Actions of $G$ and $T$ on $B(G,c)$}
\label{subsec-action}
We now go back to the general case.
The scheme of $(G,c)$-bands carries two relevant actions described below. 
Namely, we will describe a right action of the group $G$ and a left action of the torus~$T$
on $B(G,c)$. 

\medskip

The action of $G$ on $G(\infty)$ by right diagonal multiplication stabilises the closed subscheme $B(G,c)$. 
Recalling the description of the functor of points of the scheme  $B(G,c)$ (Eq.~(\ref{eq: points bands})), this action is given by the formula:
$$
g' \cdot g= \bigl( g'(s) g \bigr)_{s \in \Z} \qquad \bigl( g'=\bigl(g'(s) \bigr)_{s \in \Z} \in B(G,c)(Z), \ g \in G(Z), \ Z \in \Schk \bigr).
$$
This action corresponds, through the isomorphism $\tau$ of Corollary \ref{cor: bands product}, to the action of $G$ on $X(G,c)$ induced by the right multiplication of $G$ on $X_0$.
Recall that $X_0$ is a copy of $G$. 
This last observation has the following remarkable obvious consequence.

\begin{Cor}
    \label{cor: inv function on bands}
    The algebra $R(G,c)^G$ of $G$-invariant functions on $B(G,c)$ is isomorphic to a polynomial ring in countably many variables with coefficients in $K$.
\end{Cor}

Notice that we have similar right actions of $G$ on any space of finite bands, with respect to which the morphisms $\pi_{M,N} \ ( (M,N) \in \P)$ and $\pi_{M',N'; \ M,N} \ ( (M,N) \leq (M',N'))$ are equivariant.

\medskip

We now switch to the description of the $T$-action. Let $Z$ be a $K$-scheme and $t$ be a $Z$-point of $T$. Consider the element 
$$
t(s):= (\oc)^{-s} t (\oc)^s \in T(Z) \qquad (s \in \Z).
$$
Then, the torus $T$ acts on the left on the scheme $B(G,c)$ by means of the formula
$$
t \cdot g:= \bigl(t(s) g(s)\bigr)_{s \in \Z}, \qquad (t \in T(Z), \ g \in B(G,c)(Z), \ Z \in \Schk).
$$
Notice that this action of $T$ commutes with the previously defined right $G$-action.

\begin{remark}
\label{rem: bands group}
    {\rm  
    Observe that the algebraic group $U(c^{-1})$ is canonically identified with the variety $U(c^{-1}) \oc$  by means of the right multiplication by $\oc$.
This observation and Corollary \ref{cor: bands product} allow to identify the scheme of $(G,c)$-bands with the product of a copy of $G$ and countably many copies of $U(c^{-1})$.  
This identification gives to $B(G,c)$ the structure of an infinite dimensional affine group scheme.
Moreover, since the group $U(c^{-1})$ is stable under the conjugation action of $T$ on $G$, the scheme $B(G,c)$ also inherits an action of the product of countably many copies of $T$.
   Thus one could consider more general algebraic structures on $B(G,c)$ than the one described in this section. We focused on the description of these $G$ and $T$-actions  because they are related to the cluster structure of $B(G,c)$, and because of this they are connected with the representation theory of shifted quantum affine algebras. 
    }
\end{remark}

\begin{example}
    \rm{
Consider the case of $G=SL(n)$ and  ${c} =c_{st}.$
In this example, we use the explicit description of the set $B(SL(n), c_{st})(K)$  given in \S\ref{sec:SL(n) cst} in terms of $(\infty \times n)$-arrays with entries in $K$.

Let $\mathrm{B} \in B(SL(n), c_{st})(K)$ and $g \in SL(n,K)$. 
The band $\mathrm{B} \cdot g$ is obtained by standard matrix multiplication of $B$ and $g$.

Similarly, for $t = \mathrm{diag}(t_i)\in T(K)$, the maximal torus of diagonal matrices in $SL(n,K)$, the band 
$t\cdot \mathrm{B}$ is obtained by matrix multiplication $\overline{t} \mathrm{B}$, where $\overline{t}$ is the infinite periodic diagonal matrix
\[
\overline{t} = \mathrm{diag}(\overline{t}_s\mid s\in \Z) \quad\mathrm{ with }\quad \overline{t}_s = t_{(s\ \mod\ n)}.
\]
    }
\end{example}

\section{Cluster structure}
\label{sec_cluster_bands}

In this section we fix $G$ and a Coxeter element $c$. Recall that we denote by $\wt c$ the Coxeter element
\[
 \wt c := w_0 c^{-1}w_0.
\]
Note that for every $i\in I$ the integer $m_i$ defined in \S\ref{sec_notation}
takes the same value for $c$ and $\wt{c}$. 
Let 
$$
\NC:= \{ (i,k) \ \mid \ i \in I, \ 1 \leq k \leq m_i \}.
$$
Write $ c=s_{i_1}\cdots s_{i_n}$. Let $Q$ denote the corresponding orientation of the Dynkin diagram of $G$. 
It is defined as follows. For a quiver $Q'$ with vertex set $I$, and a vertex $i\in I$, we denote by $s_i(Q')$ the quiver obtained from
$Q'$ by changing the orientation of all arrows incident to $i$.
Then $Q$ is the quiver uniquely determined by the fact that $i_1$ is a source of $Q$, $i_2$ is a source of the quiver $s_{i_1}(Q)$,  and similarly for $2\le k \le n$, $i_k$ is a source of $s_{i_{k-1}}\cdots s_{i_1}(Q)$.

\medskip

Let $\pi_s: B(G,c) \lra G \ (s \in \Z)$ be the natural projection on the $s$-component.
In what follows, we will frequently use the following shorthand notation for certain elements of the ring $R(G,c)$:
\[
\Delta^{(s)}_{u\varpi_i,v\varpi_i} := \pi_s^*\left(\Delta_{u(\varpi_i),\,v(\varpi_i)}\right),\qquad (i \in I, \ u,v\in W,\ s\in \Z).
\] 

\subsection{Statement of the theorem}
\label{sec: statement of main theorem}

 We define the integers $\xi_i\ (i\in I)$ inductively by setting $\xi_{i_1} = 0$ and $\xi_j = \xi_i - 1$ whenever there is an arrow $i \to j$ in $Q$. We also define integers $\wt{\xi_i}\ (i\in I)$ associated in the same way with $\wt{c}$.

In \cite[{\S}3.4]{GHL} an infinite quiver $\G_{\wt{c}}$ is associated with the datum $(G, \wt{c})$. 
This quiver $\G_{\wt{c}}$ can be divided into three segments : a finite connected middle part containing green and red vertices, and two infinite upper and lower parts containing the remaining black vertices. 
The columns of $\G_{\wt{c}}$ are in one-to-one correspondence with $I$. The labels $(i,r)$ of the vertices of $\G_{\wt{c}}$ are a priori defined only up to a uniform shift of the second component $r$. 
So to fix this we adopt the same convention as in \cite[{\S}4.2]{GHL} and we denote by $(\nu(i_n),0)$ the highest red vertex in column $\nu(i_n)$, where  $\nu$ is the involution of $I$ defined by $w_0(\varpi_i) = -\varpi_{\nu(i)}$.
Then the red vertices have labels of the form
\[
 (i, \wt{\xi_i} - 4k),\quad (i\in I,\ 0\le k< m_i),
\]
the green vertices have labels of the form
\[
 (i, \wt{\xi_i} - 4k-2),\quad (i\in I,\ 0\le k< m_i),
\]
the black vertices of the upper part have labels of the form
\[
 (i,r),\quad (i\in I,\ r\equiv \wt{\xi_i}\ (2),\ r>\wt{\xi_i}),  
\]
and the black vertices of the lower part have labels of the form
\[
 (i,r),\quad (i\in I,\ r\equiv \wt{\xi_i}\ (2),\ r\le\wt{\xi_i} - 4m_i).
\]
We denote by $\VV$ the vertex set of $\G_{\wt{c}}$.
We now attach to each $(i,r)\in\VV$ an element of $R(G, c)$. 

\begin{Def}[Initial cluster variables]\label{def-initial-cluster-variables}
For $(i,r)\in \VV$, set $r_i := (r-\wt {\xi_i})/2$. An element of $R(G, c)$ is associated with $(i,r)$ as follows:
\begin{itemize}
 \item to a red vertex $(i, \wt{\xi_i} - 4k)$ we attach the element $\Delta^{(0)}_{c^{m_i-1-k}(\varpi_i),\, \wt{c}^{\,k}(\varpi_i)}$;
 \item to a green vertex $(i, \wt{\xi_i} - 4k-2)$ we attach the element $\Delta^{(0)}_{c^{m_i-1-k}(\varpi_i),\, \wt{c}^{\,\,k+1}(\varpi_i)}$;
 \item to a black vertex $(i,r)$ in the upper part of $\G_{\wt{c}}$ we attach the element 
 \[
\Delta^{(r_i-1)}_{c^{m_i}(\varpi_i),\varpi_i} = \Delta^{(r_i-1)}_{w_0(\varpi_i),\varpi_i};
\]
 \item to a black vertex $(i,r)$ in the lower part of $\G_{\wt{c}}$ we attach the element 
\[
 \Delta^{\left(r_i+2m_i-1\right)}_{\varpi_i,\,\wt{c}^{\,m_i}(\varpi_i)} = \Delta^{(r_i+2m_i-1)}_{\varpi_i,\,w_0(\varpi_i)}.
\] 
\end{itemize}
\end{Def}

\begin{figure}[t!]
\[
\def\objectstyle{\scriptstyle}
\def\lablestyle{\scriptstyle}
\xymatrix@-1.0pc{
&&
\\
&{}\save[]+<0cm,1.5ex>*{\vdots}\restore&{}\save[]+<0cm,1.5ex>*{\vdots}\restore  
&{}\save[]+<0cm,1.5ex>*{\vdots}\restore
\\
&{(1,3)}\ar[rd]\ar[u]&
&\ar[ld] (3,3) \ar[u]
\\
&&\ar[ld] (2,2) \ar[rd]\ar[uu]&&
\\
&\ar[uu]{(1,1)}\ar[rd]&
&\ar[ld] (3,1) \ar[uu]
\\
&&\ar[uu] \mathbf{\red(2,0)}\ar[d] &&
\\
&&\ar[ld]\ar[rd] \mathbf{\green(2,-2)} &&
\\
&\ar[uuu]\mathbf{\red(1,-1)}\ar[d]  && \mathbf{\red(3,-1)}\ar[uuu]\ar[d]
\\
&\mathbf{\green(1,-3)} \ar[rd] &&\ar[ld] \mathbf{\green(3,-3)}
\\
&& \ar[d]\ar[uuu]\mathbf{\red(2,-4)} &&
\\
&&\ar[ld] \mathbf{\green(2,-6)} \ar[rd]&&
\\
&\ar[uuu]\mathbf{\red(1,-5)}\ar[d]  && \mathbf{\red(3,-5)}\ar[d]\ar[uuu]
\\
&\mathbf{\green(1,-7)} \ar[rd] &&\ar[ld] \mathbf{\green(3,-7)}
\\
&& \ar[uuu](2,-8) \ar[ld]\ar[rd]&&
\\
&\ar[uu](1,-9) &&\ar[uu] (3,-9) 
\\
&{}\save[]+<0cm,0ex>*{\vdots}\ar[u]\restore&{}\save[]+<0cm,0ex>*{\vdots}\ar[uu]\restore  
&{}\save[]+<0cm,0ex>*{\vdots}\ar[u]\restore
\\
}
\qquad
\def\objectstyle{\scriptscriptstyle}
\def\lablestyle{\scriptscriptstyle}
\xymatrix@-1.0pc{
&\De^{(1)}_{c^2\varpi_1,\,\varpi_1}\ar[rd]
&
&\ar[ld] \De^{(1)}_{c^2\varpi_3,\,\varpi_3} 
\\
&&\ar[ld] \De^{(0)}_{c^2\varpi_2,\,\varpi_2} \ar[rd]
&&
\\
&\ar[uu] \De^{(0)}_{c^2\varpi_1,\,\varpi_1} \ar[rd]&
&\ar[ld] \De^{(0)}_{c^2\varpi_3,\,\varpi_3} \ar[uu]
\\
&&\ar[uu] {\red \De^{(0)}_{c\varpi_2,\,\varpi_2}}\ar[d] &&
\\
&&\ar[ld]\ar[rd] {\green \De^{(0)}_{c\varpi_2,\,\wt{c}\varpi_2}} &&
\\
&\ar[uuu]{\red \De^{(0)}_{c\varpi_1,\,\varpi_1}}\ar[d]  && {\red \De^{(0)}_{c\varpi_3,\,\varpi_3}}\ar[uuu]\ar[d]
\\
&{\green \De^{(0)}_{c\varpi_1,\,\wt{c}\varpi_1}} \ar[rd] &&\ar[ld] {\green \De^{(0)}_{c\varpi_3,\,\wt{c}\varpi_3}}
\\
&& \ar[d]\ar[uuu]{\red \De^{(0)}_{\varpi_2,\,\wt{c}\varpi_2}} &&
\\
&&\ar[ld] {\green \De^{(0)}_{\varpi_2,\,\wt{c}^2\varpi_2}} \ar[rd]&&
\\
&\ar[uuu]{\red \De^{(0)}_{\varpi_1,\,\wt{c}\varpi_1}}\ar[d]  && {\red \De^{(0)}_{\varpi_3,\,\wt{c}\varpi_3}}\ar[d]\ar[uuu]
\\
&{\green \De^{(0)}_{\varpi_1,\,\wt{c}^2\varpi_1}} \ar[rd] &&\ar[ld] {\green \De^{(0)}_{\varpi_3,\,\wt{c}^2\varpi_3}}
\\
&& \ar[uuu]\De^{(-1)}_{\varpi_2,\,\wt{c}^2\varpi_2} \ar[ld]\ar[rd]
&&
\\
&\ar[uu] \De^{(-1)}_{\varpi_1,\,\wt{c}^2\varpi_1} &&\ar[uu] \De^{(-1)}_{\varpi_3,\,\wt{c}^2\varpi_3} 
\\
}
\]
\caption{\label{Fig3} {\it The quiver $\G_{\wt{c}}$ for $\wt{c} = s_2s_1s_3$ in type $A_3$, and the corresponding initial seed of $R(G, \wt c)$.}}
\end{figure}
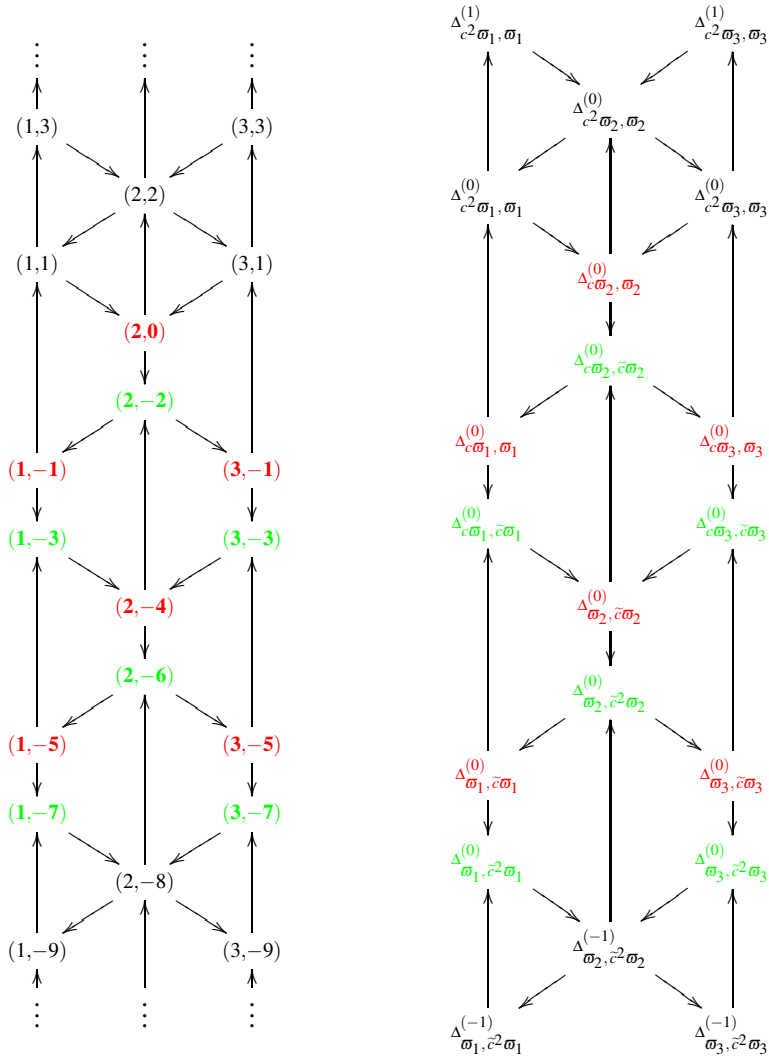 

\begin{example}
{\rm
Let $G$ be of type $A_3$, and $\wt{c} = s_2s_1s_3$. Then $c = s_1s_3s_2$, and $m_1 = m_2 = m_3 = 2$. 
We also have 
$
\wt{\xi_1} = -1,\ \wt{\xi_2} = 0,\ \wt{\xi_3} = -1. 
$
The quiver $\G_{\wt{c}}$ is displayed 
in Figure~\ref{Fig3}, together with a second copy in which vertices have been replaced by the corresponding elements of $R(G, c)$.
}
\end{example}

We can now state one of the main theorems of this paper.
Let $\AA$ denote the cluster algebra with initial seed $(\G_{\wt{c}}, \{z_{(i,r)}\mid (i,r)\in \VV\})$.
For $(i,r)\in\VV$, let $\De_{(i,r)}$ be the corresponding element of $R(G, c)$ introduced in Definition~\ref{def-initial-cluster-variables}.
\begin{Thm}\label{Thm4.4}
The assignment 
\[
z_{(i,r)}\mapsto \De_{(i,r)},\qquad ((i,r)\in\VV)
\]
extends to a unique algebra isomorphism from
$K\otimes \AA$ to $R(G, c)$.
\end{Thm}

\begin{remark}
{\rm
The reader can easily verify that in the case of $G=SL(n)$ and the standard Coxeter element $c_{st}$, Theorem \ref{Thm4.4} recovers Theorem \ref{thm: cluster structure SL cst bands} through the isomorphism $\bar\iota_n$ of Lemma \ref{lem: SLn and general bands}. 
}
\end{remark}
The rest of this section will be devoted to the proof of Theorem~\ref{Thm4.4}.
We will denote by
\[
 \Sigma := (\G_{\wt{c}}, \{\De_{(i,r)}\mid (i,r)\in \VV\}).
\]
our candidate initial seed of $R(G, c)$.

\medskip

A key ingredient of the proof of Theorem \ref{Thm4.4} is provided by the gluing relations of Proposition~\ref{lem: glueing formulas}. 
These relations imply that our candidate initial seed satisfies the translation invariance property described by Proposition \ref{prop translation}.
In order to prove Proposition~\ref{lem: glueing formulas}, we need to introduce certain elements of the Weyl group of $G$ parametrised by the set $\NC$.

\subsection{The elements $w_{i,k}$}
\label{subsubsec:w_{i,k}}
We need some preliminary results concerning the action of the Coxeter element $ c$ on the weight lattice $P$. 
To do this, we find it convenient to use the categorical interpretation of this
action coming from quiver representations.

\medskip
 
Let $\GG_{Q}$ be the Auslander-Reiten quiver of the quiver $Q$ associated with $c$ \cite{ARS}. 
By definition, the vertices of $\GG_{Q}$ are the isomorphism classes of indecomposable representations of $Q$, which are in bijection with positive roots by taking dimension vectors.
For instance, the simple representation $S_i$ of $Q$ supported on vertex $i\in I$ has dimension vector $\a_i$, the injective hull $I_i$ of $S_i$ has dimension vector $\varpi_i - c(\varpi_i)$, and the projective cover $P_i$ of $S_i$ has dimension vector $\varpi_i - c^{-1}(\varpi_i)$.
The vertices of $\GG_{Q}$ are displayed on $n$ rows corresponding to the vertices of $Q$.
The rightmost vertex on row $i$ is $I_i$, and the leftmost vertex is $P_{\nu(i)}$, where we recall that $\nu$ is the involution of $I$ defined by $w_0(\varpi_i) = -\varpi_{\nu(i)}$. 
The full subquiver of $\GG_{Q}$ supported on the indecomposable injectives is isomorphic to $Q^{\mathrm{op}}$, the quiver  with orientation opposite to $Q$. The Auslander-Reiten translation $\tau$ maps a non-projective indecomposable $x$ to the indecomposable $y$ placed immediately to its left. 
If $x$ has dimension vector $\beta$, then $\tau(x)$ has dimension vector $c(\beta)$.

\begin{example}\label{Example.5.5}
{\rm
Let $G$ be of type $D_5$. Take $ c = s_2s_4s_1s_3s_5$, so that $\wt c=s_4s_3s_1s_5s_2$. The Coxeter element $c$  corresponds to the quiver $Q$:
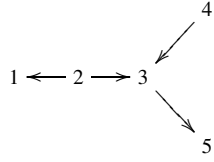
\[
\def\objectstyle{\scriptstyle}
\def\lablestyle{\scriptstyle}
\xymatrix@-0.8pc{
&&&\ar[ld]4
\\
1&\ar[l]2\ar[r]&3\ar[rd]
\\
&&&5
}
\]
The Auslander-Reiten quiver $\GG_{Q}$ of $Q$ is:

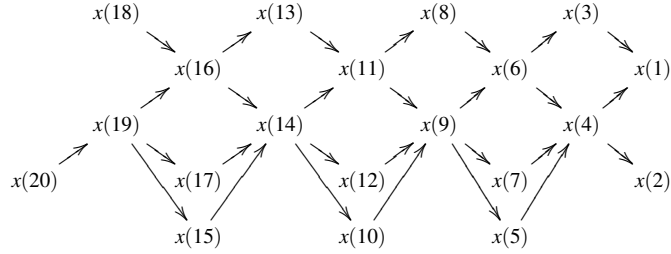
\[
\def\objectstyle{\scriptstyle}
\xymatrix@-1.4pc @l{
&&x(5)\ar[ldd]&&x(10)\ar[ldd]&&x(15)\ar[ldd]\\
x(2)&&\ar[ld]x(7)&&\ar[ld]x(12) &&\ar[ld]x(17)&&x(20)\ar[ld]\\
 &\ar[lu]\ar[ld]x(4)&&\ar[lu]\ar[luu]\ar[ld]x(9)&&\ar[lu]\ar[luu]\ar[ld]x(14)&&\ar[lu]\ar[luu]\ar[ld]x(19)\\
x(1)&&\ar[lu]\ar[ld]x(6)&&\ar[lu]\ar[ld]x(11)&&\ar[lu]\ar[ld]x(16)\\
&x(3)\ar[lu]&&\ar[lu]x(8)&&\ar[lu]x(13)&&\ar[lu]x(18)
}
\]
We have, for instance, $\tau(x(1)) = x(6)$, $\tau(x(2))=x(7)$, $\tau(x(5)) = x(10)$. The injective indecomposable representations are $x(1), \ldots , x(5)$, with dimension vectors (expressed as positive roots):
\[
\dim(x(1))=\alpha_2,\quad 
\dim(x(2))=\alpha_4,\quad
\dim(x(3))=\alpha_1+\alpha_2,\ 
\]
\[
\dim(x(4))=\alpha_2+\alpha_3+\alpha_4,\quad
\dim(x(5))=\alpha_2+\alpha_3+\alpha_4+\alpha_5, 
\]
The dimension vectors of all other vertices can be obtained from these initial ones by a \emph{knitting algorithm}, for instance:
\[
 \dim(x(6)) = \dim(x(3)) + \dim(x(4)) - \dim(x(1)) = \alpha_1+\alpha_2+\alpha_3+\alpha_4,
\]
\[
 \dim(x(7)) =  \dim(x(4)) - \dim(x(2)) = \alpha_2+\alpha_3,
\]
\[
 \dim(x(8)) =  \dim(x(6)) - \dim(x(3)) = \alpha_3+\alpha_4,
\]
\[
 \dim(x(9)) = \dim(x(5)) + \dim(x(6)) +\dim(x(7)) - \dim(x(4)) = \alpha_1+2\alpha_2+2\alpha_3+\alpha_4+\alpha_5.
\]
The relation $\tau(x(1)) = x(6)$ translates into $c(\alpha_2) = \alpha_1+\alpha_2+\alpha_3+\alpha_4$.
}
\end{example}

We fix an \emph{adapted} numbering of the vertices of $\GG_{Q}$. 
In other words, the chosen numbering satisfies the condition that for every arrow $x(i) \to x(j)$, we have $i>j$.
With this choice is associated the sequence $\bi := (i_1, \ldots, i_N)$, where $i_k$ is the row number of $x(k)$ in $\GG_{Q}$. 
Then $s_{i_1}\ldots s_{i_N}$ is a reduced expression for $w_0$. 

\begin{Lem}
\label{Lem3.10}
   Let $i \in I$. We have $m_i = \sharp\{j\in [1,N] \mid i_j = i\}$. In other words, $m_i$ is equal to the number of vertices on row $i$ of $\GG_{Q}$. 
\end{Lem}

\begin{proof}
Let $n_i$ denote the number of vertices on row $i$ of $\GG_{Q}$.
The list of dimension vectors on this row, read from right to left, is :
\[
\dim(I_i) = \varpi_i - c(\varpi_i),\ c(\varpi_i) - c^2(\varpi_i),\ \ldots, \ c^{n_i-1}(\varpi_i) - c^{n_i}(\varpi_i),  
\]
where the last one is equal to 
\[
\dim(P_{\nu(i)}) = \varpi_{\nu(i)}-c^{-1}(\varpi_{\nu(i)}) = -c^{m_i}(\varpi_i)+ c^{m_i-1}(\varpi_i). 
\]
By minimality of $m_i$, it follows that $m_i = n_i$.
    \cqfd
\end{proof}

\begin{remark}\label{rem-3.11}
{\rm
It easy to see using Lemma~\ref{Lem3.10} that if $m_i = 1$ for some $i$, then the group $G$
is of type $A_n$, and either $c = c_{st} = s_1\cdots s_n$ and $i=n$, or $c = c_{st}^{-1}$
and $i=1$. Indeed if $m_i = 1$ then we have that $I_i = P_{\nu(i)}$ is a 
projective-injective indecomposable representation of $Q$. This implies that $i$ is a sink of $Q$,
$\nu(i)$ is a source of $Q$, and there is an oriented path from $\nu(i)$ to $i$ in $Q$. Moreover
there can only be one outgoing arrow at vertex $\nu(i)$, otherwise $P_{\nu(i)}$ would not have a simple socle. Similarly, there can only be one ingoing arrow at $i$, otherwise $I_i$ would not have a simple head.
Hence $i$ and $\nu(i)$ are two terminal vertices of $Q$. 
Finally, there can not be a branching node on the path from $\nu(i)$ to $i$ since again this would imply that $I_i = P_{\nu(i)}$ has not a simple socle or a simple head. So $Q$ is the equi-oriented quiver $\nu(i)\to \cdots \to i$. 
This remark will be used in \S\ref{subsec-one-step-mutations} below.
}
\end{remark}

We can always assume that $\{i_1,\ldots, i_n\} = \{1,\ldots,n\}$, that is, that the adapted numbering starts with the $n$ rightmost vertices on each row, corresponding to the $n$ indecomposable injectives.
In this case, we also have that $c = s_{i_1}\cdots s_{i_n}$. In the sequel, we will always assume that this additional property is fulfilled.

Let $(i,k) \in \NC$. By Lemma~\ref{Lem3.10}, we can associate with $(i,k)$ a unique vertex $x(a)$ of $\GG_Q$, namely, the $k$-th vertex on row $i$, starting from the right. We shall then denote $a(i,k) := a$. In other words, $i_{a(i,k)}$ is the $k$-th element equal to $i$ in the list $\bi$.  Let $w_{i,k}$ be the element of $W$ defined by the formula
$$
w_{i,k}= \left\{\begin{array}{ll}
    s_{i_{n+1}}\cdots s_{i_{a(i,k)}} &  \mbox{if} \quad k \geq 2, \\
    e & \mbox{if} \quad k=1.
\end{array} \right.
$$

\begin{example}
{\rm
We continue the previous example. The numbering of $\GG_{Q}$ is adapted, and we have: 
\[
 \bi= (2,4,1,3,5,2,4,1,3,5,2,4,1,3,5,2,4,1,3,4).
\]
Also $s_{i_1}s_{i_2}s_{i_3}s_{i_4}s_{i_5} = s_2s_4s_1s_3s_5 = c$.
We have $m_1 = 4,\ m_2 = 4,\ m_3 = 4,\ m_4 = 5,\ m_5 = 3$, and
\begin{eqnarray*}
&& w_{1,2} = s_2s_4s_1,\ w_{1,3} = s_2s_4s_1s_3s_5s_2s_4s_1,\ w_{1,4} = s_2s_4s_1s_3s_5s_2s_4s_1s_3s_5s_2s_4s_1,  
\\[2mm]
&& w_{2,2} = s_2,\ w_{2,3} = s_2s_4s_1s_3s_5s_2,\ w_{2,4} = s_2s_4s_1s_3s_5s_2s_4s_1s_3s_5s_2,  
\\[2mm]
&& w_{3,2} = s_2s_4s_1s_3,\ w_{3,3} = s_2s_4s_1s_3s_5s_2s_4s_1s_3,\ w_{3,4} = s_2s_4s_1s_3s_5s_2s_4s_1s_3s_5s_2s_4s_1s_3,  
\\[2mm]
&& w_{4,2} = s_2s_4,\ w_{4,3} = s_2s_4s_1s_3s_5s_2s_4,\ w_{4,4} = s_2s_4s_1s_3s_5s_2s_4s_1s_3s_5s_2s_4,\
\\
&&\qquad w_{4,5} = s_2s_4s_1s_3s_5s_2s_4s_1s_3s_5s_2s_4s_1s_3s_4,
\\[2mm]
&& w_{5,2} = s_2s_4s_1s_3s_5,\ w_{5,3} = s_2s_4s_1s_3s_5s_2s_4s_1s_3s_5.  
\end{eqnarray*}
}
\end{example}

\begin{Lem} 
\label{Lem2-4}
The equalities
\[
w_{i,k}(\varpi_i) = c^{k-1}(\varpi_i)\quad \mbox{and}\quad \ell(c w_{i,k}) = \ell(c) + \ell(w_{i,k})
\]
hold for any $(i,k) \in \NC.$
\end{Lem}

\begin{proof}
We first prove that, for $1\le j\le N$, if $x(j)$ is the $k$-th vertex on row~$i_j$ starting from the right,
we have 
\begin{equation}\label{Eq-15}
s_{i_1}s_{i_2}\cdots s_{i_{j-1}}(\varpi_{i_j}) = c^{k-1}(\varpi_{i_j}),
\qquad
s_{i_1}s_{i_2}\cdots s_{i_j}(\varpi_{i_j}) = c^k(\varpi_{i_j}).
\end{equation}
This can be deduced by induction on $k$ using the the well-known 
formula 
\[
\dim(x(j)) = s_{i_1}\cdots s_{i_{j-1}}(\a_{i_j}).
\]
Indeed, if $k=1$ then  $x(j)$ is equal to the injective indecomposable $I_{i_j}$.
So
\[
\dim x(j) = \varpi_{i_j} - c(\varpi_{i_j}) = s_{i_1}\cdots s_{i_{j-1}}(\a_{i_j}) 
= s_{i_1}\cdots s_{i_{j-1}}(\varpi_{i_j}) - s_{i_1}\cdots s_{i_{j-1}}s_{i_j}(\varpi_{i_j}).
\]
On the other hand, if $k=1$ then $j$ is the smallest index $p$ such that $i_p = i_j$, so 
$s_{i_1}\cdots s_{i_{j-1}}(\varpi_{i_j}) = \varpi_{i_j}$ and therefore 
$s_{i_1}\cdots s_{i_j}(\varpi_{i_j}) = c(\varpi_{i_j})$. So Eq.~(\ref{Eq-15}) holds for $k=1$.
Suppose that $k\ge 2$, and let $l$ be such that $x(l)$ is the $(k-1)$th vertex on row $i_j$ 
starting from the right. By induction on $k$ we may assume that 
$s_{i_1}s_{i_2}\cdots s_{i_l}(\varpi_{i_j}) = c^{k-1}(\varpi_{i_j})$. 
Since for every $l<s<j$ we have $i_s \not = i_j$, we deduce that 
$s_{i_1}s_{i_2}\cdots s_{i_{j-1}}(\varpi_{i_j}) = c^{k-1}(\varpi_{i_j})$. 
Now we have
\[
 \dim x(j) = c^{k-1}(\dim(I_{i_j})) = c^{k-1}\varpi_{i_j} - c^{k}\varpi_{i_j}, 
\]
and also
\[
 \dim x(j) = s_{i_1}\cdots s_{i_{j-1}}(\a_{i_j}) = s_{i_1}s_{i_2}\cdots s_{i_{j-1}}(\varpi_{i_j})
 - s_{i_1}s_{i_2}\cdots s_{i_{j}}(\varpi_{i_j}),
\]
so we get that $c^{k}\varpi_{i_j} = s_{i_1}s_{i_2}\cdots s_{i_{j}}(\varpi_{i_j})$. Thus 
Eq.~(\ref{Eq-15}) holds for every $k$.

Hence, for $2\le k \le m_{i_j}$, we have $c^k(\varpi_{i_j}) = c w_{i_j,k}(\varpi_{i_j})$, and therefore,
$w_{i_j,k}(\varpi_i) = c^{k-1}(\varpi_{i_j})$.
Moreover, since $c w_{i_j,k} = s_{i_1}s_{i_2}\cdots s_{i_j}$ consists of the first $j$ letters of the reduced word  $\bi$ of $w_0$, we have $\ell(c w_{i_j,k}) = j = \ell(c) + \ell(w_{i_j,k})$.
Finally, the statement is obviously true if $k=1$.
\cqfd
\end{proof}

Note that the elements $w_{i,k}$ depend on the choice of an adapted numbering of $\GG_Q$, but it follows from Lemma~\ref{Lem2-4} that the weights $w_{i,k}(\varpi_i)$ are independent of this choice.

\subsection{Gluing formulas}

We start with a simple lemma.

\begin{Lem}
    \label{lem: week order}
    Let $(i,k) \in \NC$. The group $w_{i,k}^{-1} \bigl( U^- \cap c^{-1} U c \bigr) w_{i,k}$ is contained in $U^-.$
\end{Lem}

\begin{proof}
Let $\Phi^+$  (\resp $\Phi^-$) be the set of positive (\resp negative) roots determined by the Borel subgroup $B$.
The group $U^- \cap c^{-1} U c$ is a $T$-stable subgroup of $U^-$. 
Its Lie algebra is the direct sum  of the root spaces corresponding to the set of roots $\Phi^- \cap c^{-1} \Phi^+$. 
As the group $U^- \cap c^{-1} U c$ is connected, it is sufficient to prove that 
\begin{equation}
    \label{eq: inversions}
    w_{i,k}^{-1} \bigl( \Phi^- \cap c^{-1} \Phi^+ \bigr) \subseteq \Phi^-.
\end{equation}
Lemma \ref{Lem2-4} implies that $\ell(w_{i,k}^{-1}c^{-1})= \ell(w_{i,k}^{-1})+ \ell(c^{-1})$, 
 which is equivalent to the fact that a reduced expression of $w_{i,k}^{-1}c^{-1}$ can be obtained by concatenating a reduced expression of $w_{i,k}^{-1}$ with one of $c^{-1}$. By \cite[\S 5.6, Exercise 1]{Hum}, this
implies that if a root $\alpha \in \Phi^+$ satisfies $c^{-1} (\alpha) \in \Phi^-$, then $w_{i,k}^{-1}c^{-1} (\alpha) $ belongs to $\Phi^-.$
In other words, we have proved that 
$$
w_{i,k}^{-1}  c^{-1} \bigl( \Phi^+ \cap c \Phi^- \bigr) \subseteq \Phi^-,
$$
which is a reformulation of Eq. (\ref{eq: inversions}).
    \cqfd
\end{proof}

\begin{Prop}
    \label{lem: glueing formulas}
    The following formulas hold in the ring $R(G, c)$:
    \begin{equation}
    \label{glue-eq}
 \Delta^{(s)}_{c^k(\varpi_i),\,w(\varpi_i)} = \Delta^{(s+1)}_{c^{k-1}(\varpi_i),\,w(\varpi_i)},
\quad (i\in I,\ 1\le k\le m_i,\ w\in W,\ s\in\Z).
\end{equation}
\end{Prop}

\begin{proof}
Fix a choice of $i,k,w$ and $s$ as in the statement of the Proposition.
Equation~\ref{glue-eq}
can be interpreted as a relation between functions on the space of finite bands $B(G,c,s,s+1)$. 
Since $B(G,c,s,s+1)$ is a variety  (see Corollary \ref{cor: geometric prop of finite bands}), it is sufficient to prove that for any closed point $g=(g(s), g(s+1))$ of  $B(G,c,s,s+1)$, then 
$$
\Delta_{c^k(\varpi_i),\,w(\varpi_i)}\bigl(g(s)\bigr) = \Delta_{c^{k-1}(\varpi_i),\,w(\varpi_i)}\bigr(g(s+1) \bigr).
$$
By the definition of the spaces of finite bands, there exists $x \in U(c^{-1}) \oc $ such that $g(s)=x g(s+1)$. 
This remark together with Lemma \ref{Lem2-4} and the properties of generalised minors \cite[Definition 1.4, Proposition 2.1]{FZ} allow to compute that
\begin{equation}
\label{eq:glu 1}
\Delta_{c^k(\varpi_i),\,w(\varpi_i)}\bigl(g(s)\bigr)  =  \Delta_{\varpi_i,\,w(\varpi_i)}\bigl(\overline{w_{i,k}}^{-1}\oc^{-1} x g(s+1)\bigr).
\end{equation}
But, by the definition of $U(c^{-1})$, we have that $\oc^{-1} x$ belongs to $U^- \cap c^{-1} U c.$
Therefore, Lemma \ref{lem: week order} implies that 
\begin{equation}
    \label{eq: glu 2}
\Delta_{\varpi_i,\,w(\varpi_i)}\bigl(\overline{w_{i,k}}^{-1}\oc^{-1} x g(s+1)\bigr)=
\Delta_{\varpi_i,\,w(\varpi_i)}\bigl(u\overline{w_{i,k}}^{-1}g(s+1)\bigr)
\end{equation}
for some $u \in U^-.$ 
Using that the generalised minor $\Delta_{\varpi_i,\,w(\varpi_i)}$ is invariant under left multiplication by $U^-$ and Lemma \ref{Lem2-4}, we deduce that
\begin{equation}
    \label{eq: glu 3}\Delta_{\varpi_i,\,w(\varpi_i)}\bigl(u\overline{w_{i,k}}^{-1}g(s+1)\bigr)= \Delta_{c^{k-1}(\varpi_i),\,w(\varpi_i)}\bigl(g(s+1)\bigr).
\end{equation}
Putting Equations (\ref{eq:glu 1}), (\ref{eq: glu 2}) and (\ref{eq: glu 3}) together, the statement follows.
    \cqfd
\end{proof}

\begin{remark}
\rm{
    In this paper, we will be frequently dealing with the problem of proving that two elements $\varphi$ and $\psi$ of $R(G,c)$ are equal.     
    We claim that if the equality $\varphi(b)= \psi(b)$ holds for any $K$-rational point $b$ of $B(G,c)$, then we have that $\varphi=\psi$ as elements of $R(G,c)$. 
    
    Indeed, since $\varphi$ and $\psi$ belong to $R(G,c,-n,n)$ for some sufficiently large $n \in \Z_{\ge 0}$, we can interpret the equality $\varphi=\psi$ as an equality of regular functions on $B(G,c,-n,n)$. 
    Since $B(G,c,-n,n)$ is a variety, the desired equality can be verified on $K$-rational points of $B(G,c,-n,n)$. 
    The claim follows since, by the description of $B(G,c)$ given in Corollary \ref{cor: bands product}, we have that for every $(M,N) \in \P$ the projection $\pi_{M,N}$ surjects the $K$-rational points of $B(G,c)$ on the $K$-rational points of $B(G,c,M,N)$. This is actually the argument used in the proof of Proposition~\ref{lem: glueing formulas}.
    In the sequel we will freely use this method without further comment. 
}    
\end{remark}

\subsection{Translation invariance}

There are no arrows between red (\resp green) vertices of $\G_{\wt{c}}$. Therefore the mutations at these vertices pairwise commute, and we can consider their product $\tau_{\mathrm{red}}$ (\resp $\tau_{\mathrm{green}})$ in an arbitrary  order. It was explained in \cite{GHL} that the quivers $\tau_{\mathrm{red}}(\G_{\wt{c}}))$ (\resp $\tau_{\mathrm{green}}(\G_{\wt{c}}))$ are isomorphic to $\G_{\wt{c}}$. More precisely, they are obtained by translating one step above (\resp one step below) the finite middle part of $\G_{\wt{c}}$ consisting of red and green vertices. We will now check that a similar property is satisfied by our set of candidate initial cluster variables $\{\De_{(i,r)}\mid (i,r)\in \VV\}\subset R(G, c)$. 

Using Fomin-Zelevinsky mutation formulas, 
for every $(i,r)\in \VV$ we can calculate the element $\tau_{\mathrm{red}}(\Delta_{(i,r)})$ (\resp $\tau_{\mathrm{green}}(\Delta_{(i,r)})$), 
which is a priori an element of the fraction field of $R(G, c)$, well-defined since $\De_{(i,r)}$ is non-zero. 
The following proposition shows that it belongs to $R(G, c)$. 
Moreover, if we write 
\[
\begin{array}{l}
\tau_{\mathrm{red}}\left(\Sigma\right) := \left(\tau_{\mathrm{red}}(\G_{\wt{c}}),
\ \{\tau_{\mathrm{red}}
 (\Delta_{(i,r)}) \mid (i,r)\in \VV \}\right),\\[0.3em]
\tau_{\mathrm{green}}\left(\Sigma\right) := \left(\tau_{\mathrm{green}}(\G_{\wt{c}}),\ \{\tau_{\mathrm{green}}
(\Delta_{(i,r)}) \mid (i,r)\in \VV \}\right),
\end{array}
\]
and regard these as quivers with vertices labeled by elements of $R(G, \wt c)$, they satisfy the following translation invariance property.
Remember from Definition~\ref{def-initial-cluster-variables} that every $\De_{(i,r)}$ is of the form 
$\Delta^{(s)}_{c^k(\varpi_i),\,\wt{c}^{\,\, l}(\varpi_i)}$ 
for appropriately chosen integers $s, k, l$.

\begin{Prop}\label{prop translation}
The labeled quiver $\tau_{\mathrm{red}}\left(\Sigma\right)$ (\resp $\tau_{\mathrm{green}}\left(\Sigma\right)$) can be obtained from $\Sigma$ by keeping the same quiver $\G_{\wt{c}}$, and replacing at each vertex the corresponding generalized minor $\Delta^{(s)}_{c^k(\varpi_i),\,\wt{c}^{\,\, l}(\varpi_i)}$ by $\Delta^{(s+1)}_{c^k(\varpi_i),\,\wt{c}^{\,\, l}(\varpi_i)}$ 
(\resp $\Delta^{(s-1)}_{c^k(\varpi_i),\,\wt{c}^{\,\, l}(\varpi_i)}$).
\end{Prop}

\begin{proof}
Let us first consider $\tau_{\mathrm{red}}\left(\Sigma\right)$. 
By Definition~\ref{def-initial-cluster-variables}, the $\De_{(i,r)}$ corresponding to red vertices of $\G_{\wt{c}}$ are all of the form $\Delta^{(0)}_{c^{m_i-1-k}(\varpi_i),\, \wt{c}^{\,k}(\varpi_i)}$. 
Therefore, they are all pullback of generalized minors under the morphism $\pi_0: B(G,c) \lra G$. The same is true for all the $\De_{(i,r)}$ corresponding to a vertex connected to a red vertex by an arrow (these vertices are either green vertices, or black vertices from the lowest layer of the upper infinite part). It follows that all exchange relations at red vertices can be regarded as algebraic identities satisfied by generalized minors over $G$. As explained in \cite[\S10]{GHL},
it turns out that all these minor identities are special instances of \cite[Theorem 1.17]{FZ}. Indeed the right hand side 
of the exchange relation for $\Delta^{(0)}_{c^{m_i-1-k}(\varpi_i),\, \wt{c}^{\,k}(\varpi_i)}$, which can be read from the quiver $\G_{\wt{c}}$, is equal to
\begin{equation}\label{equa20}
\Delta^{(0)}_{c^{m_i-k}(\varpi_i),\, \wt{c}^{\,k}(\varpi_i)} \Delta^{(0)}_{c^{m_i-1-k}(\varpi_i),\, \wt{c}^{\,k+1}(\varpi_i)}
\ +\ 
\prod_j \Delta^{(0)}_{c^{m_i-1-k+a_{ij}}(\varpi_i),\, \wt{c}^{\,k+b_{ij}}(\varpi_i)}, 
\end{equation}
where the products runs over vertices $j\in I$ connected to $i$ in the Dynkin diagram, and 
\begin{equation}\label{Eq.25}
 a_{ij} := \delta(\xi_j > \xi_i),\qquad b_{ij} := \delta(\wt{\xi_j} > \wt{\xi_i}).
\end{equation}
(For two integers $p,q$, we define $\delta(p>q):=1$ if $p$ is bigger than $q$ and $\delta(p>q):=0$ otherwise.)
Let us define
\[
 u := c^{m_i-1-k}x,\qquad v:= \wt{c}^{\,k}y,  
\]
where $x$ (\resp $y$) denotes the product of all simple reflections strictly to the left of $s_i$ in $c$ (\resp in $\wt{c}$).
Then, applying \cite[Theorem 1.17]{FZ} with these values of $u$, $v$ and $i$ we get that the polynomial function (\ref{equa20}) is equal to 
\[
\Delta^{(0)}_{c^{m_i-1-k}(\varpi_i),\, \wt{c}^{\,k}(\varpi_i)}\Delta^{(0)}_{c^{m_i-k}(\varpi_i),\, \wt{c}^{\,k+1}(\varpi_i)}. 
\]
It follows that the sequence of mutations $\tau_{\mathrm{red}}$ replaces the red generalized minors
$\Delta^{(0)}_{c^{m_i-1-k}(\varpi_i),\, \wt{c}^{\,k}(\varpi_i)}$ by
\[
\Delta^{(0)}_{c^{m_i-k}(\varpi_i),\, \wt{c}^{\,k+1}(\varpi_i)} = \Delta^{(1)}_{c^{m_i-k-1}(\varpi_i),\, \wt{c}^{\,k+1}(\varpi_i)}, 
\]
where the equality follows from the gluing formulas of Proposition~\ref{lem: glueing formulas}. Let us change their color and turn them green. 

On the other hand we can also use Proposition~\ref{lem: glueing formulas} to rewrite all green generalized minors
\[
\Delta^{(0)}_{c^{m_i-1-k}(\varpi_i),\, \wt{c}^{\,\,k+1}(\varpi_i)}\ \ \mbox{such that}\ \ m_i-1-k \ge 1
\]
(that is, all except the lowest green minor in each column) as
\[
\Delta^{(0)}_{c^{m_i-1-k}(\varpi_i),\, \wt{c}^{\,\,k+1}(\varpi_i)} = \Delta^{(1)}_{c^{m_i-2-k}(\varpi_i),\, \wt{c}^{\,\,k+1}(\varpi_i)}. 
\]
Let us turn these generalized minors red, and let us turn black the lowest green minors of the form $\Delta^{(0)}_{\varpi_i,\, \wt{c}^{\,\,m_i}(\varpi_i)}$.
Finally the black generalized minors lying on the lowest layer of the upper infinite part are of the form 
$\Delta^{(0)}_{c^{m_i}(\varpi_i),\, \varpi_i}$, and we can rewrite them as
\[
\Delta^{(0)}_{c^{m_i}(\varpi_i),\, \varpi_i} = \Delta^{(1)}_{c^{m_i-1}(\varpi_i),\, \varpi_i}, 
\]
and turn them red. Taking into account the mutation $\tau_{\mathrm{red}}(\G_{\wt{c}})$ of the quiver $\G_{\wt{c}}$
described in \cite{GHL}, it is then immediate to check that $\tau_{\mathrm{red}}\left(\Sigma\right)$ can alternatively be described as being obtained from $\Sigma$ by replacing at every vertex of the quiver $\G_{\wt{c}}$ the generalized minor $\Delta^{(s)}_{c^k(\varpi_i),\,\wt{c}^{\,\, l}(\varpi_i)}$ by $\Delta^{(s+1)}_{c^k(\varpi_i),\,\wt{c}^{\,\, l}(\varpi_i)}$ without changing the colour of the vertex. 

The reasoning for $\tau_{\mathrm{green}}(\Sigma)$ is similar, but the steps should be reversed.
First note that the $\Delta_{(i,r)}$ sitting at a green or a red vertex of $\G_{\wt{c}}$ are all of the form 
$\Delta^{(0)}_{c^k(\varpi_i),\,\wt{c}^{\,\, l}(\varpi_i)}$ with $k<m_i$. Therefore, by Proposition~\ref{lem: glueing formulas} we can rewrite all of them as
\[
\Delta^{(0)}_{c^k(\varpi_i),\,\wt{c}^{\,\, l}(\varpi_i)} = \Delta^{(-1)}_{c^{k+1}(\varpi_i),\,\wt{c}^{\,\, l}(\varpi_i)}. 
\]
It follows that all the exchange relations of $\tau_{\mathrm{green}}$ can be regarded as the pullback under the morphism $\pi_{-1}$ of algebraic identities satisfied by generalized minors over $G$. We can then check in the same manner that these exchange relations are instances of 
\cite[Theorem 1.17]{FZ}, and finish the proof as above.
\cqfd
\end{proof}

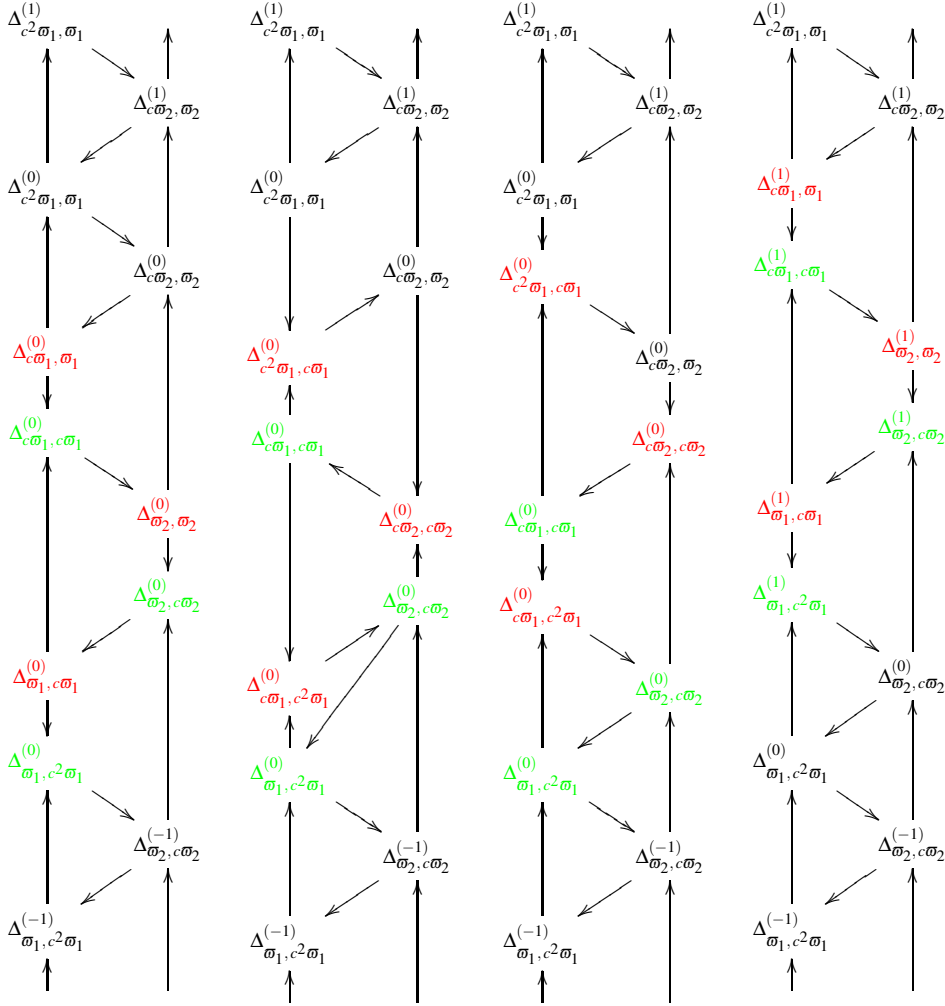
\begin{figure}[t!]
\[
\def\objectstyle{\scriptstyle}
\def\lablestyle{\scriptstyle}
\xymatrix@-1.0pc{
\De^{(1)}_{c^2\varpi_1,\,\varpi_1}\ar[rd]&
\\
&\ar[ld] \De^{(1)}_{c\varpi_2,\,\varpi_2} \ar[u]
\\
\De^{(0)}_{c^2\varpi_1,\,\varpi_1}\ar[rd]\ar[uu]&
\\
&\ar[ld] \De^{(0)}_{c\varpi_2,\,\varpi_2} \ar[uu]
\\
\ar[uu]{{\red \De^{(0)}_{c\varpi_1,\,\varpi_1}}}\ar[d]&
\\
{\green \De^{(0)}_{c\varpi_1,\,c\varpi_1}}\ar[rd]
\\
&\ar[uuu] {\red \De^{(0)}_{\varpi_2,\,\varpi_2}}\ar[d] 
\\
& \ar[ld]{\green \De^{(0)}_{\varpi_2,\,c\varpi_2}}
\\
\ar[uuu]{\red \De^{(0)}_{\varpi_1,\,c\varpi_1}} \ar[d] 
\\
{\green \De^{(0)}_{\varpi_1,\,c^2\varpi_1}}\ar[rd]
\\
& \ar[uuu]\De^{(-1)}_{\varpi_2,\,c\varpi_2}\ar[ld] 
\\
\ar[uu] \De^{(-1)}_{\varpi_1,\,c^2\varpi_1} 
\\
\ar[u] & \ar[uu]
}
\quad
\xymatrix@-1.0pc{
\De^{(1)}_{c^2\varpi_1,\,\varpi_1}\ar[rd]&
\\
&\ar[ld] \De^{(1)}_{c\varpi_2,\,\varpi_2} \ar[u]
\\
\ar[dd]\De^{(0)}_{c^2\varpi_1,\,\varpi_1}\ar[uu]&
\\
&\De^{(0)}_{c\varpi_2,\,\varpi_2} \ar[ddd]\ar[uu]
\\
{{\red \De^{(0)}_{c^2\varpi_1,\,c\varpi_1}}}\ar[ru]&
\\
\ar[u]{\green \De^{(0)}_{c\varpi_1,\,c\varpi_1}}\ar[ddd]
\\
& \ar[lu]{\red \De^{(0)}_{c\varpi_2,\,c\varpi_2}}
\\
& \ar[ldd]{\green \De^{(0)}_{\varpi_2,\,c\varpi_2}}\ar[u]
\\
{\red \De^{(0)}_{c\varpi_1,\,c^2\varpi_1}}\ar[ru] 
\\
\ar[u]{\green \De^{(0)}_{\varpi_1,\,c^2\varpi_1}}\ar[rd]
\\
& \ar[uuu]\De^{(-1)}_{\varpi_2,\,c\varpi_2}\ar[ld] 
\\
\ar[uu] \De^{(-1)}_{\varpi_1,\,c^2\varpi_1} 
\\
\ar[u] & \ar[uu]
}
\quad
\xymatrix@-1.0pc{
\De^{(1)}_{c^2\varpi_1,\,\varpi_1}\ar[rd]&
\\
&\ar[ld] \De^{(1)}_{c\varpi_2,\,\varpi_2} \ar[u]
\\
\ar[uu]{{ \De^{(0)}_{c^2\varpi_1,\,\varpi_1}}}\ar[d]&
\\
{\red \De^{(0)}_{c^2\varpi_1,\,c\varpi_1}}\ar[rd]
\\
&\ar[uuu] { \De^{(0)}_{c\varpi_2,\,\varpi_2}}\ar[d] 
\\
& \ar[ld]{\red \De^{(0)}_{c\varpi_2,\,c\varpi_2}}
\\
\ar[uuu]{\green \De^{(0)}_{c\varpi_1,\,c\varpi_1}} \ar[d] 
\\
{\red \De^{(0)}_{c\varpi_1,\,c^2\varpi_1}}\ar[rd]
\\
& \ar[uuu]{\green \De^{(0)}_{\varpi_2,\,c\varpi_2}}\ar[ld] 
\\
\ar[uu] {\green\De^{(0)}_{\varpi_1,\,c^2\varpi_1}}\ar[rd] 
\\
& \ar[uu]\De^{(-1)}_{\varpi_2,\,c\varpi_2}\ar[ld] 
\\
\ar[uu] \De^{(-1)}_{\varpi_1,\,c^2\varpi_1} 
\\
\ar[u] & \ar[uu]
}
\quad
\xymatrix@-1.0pc{
\De^{(1)}_{c^2\varpi_1,\,\varpi_1}\ar[rd]&
\\
&\ar[ld] \De^{(1)}_{c\varpi_2,\,\varpi_2} \ar[u]
\\
\ar[uu]{{\red \De^{(1)}_{c\varpi_1,\,\varpi_1}}}\ar[d]&
\\
{\green \De^{(1)}_{c\varpi_1,\,c\varpi_1}}\ar[rd]
\\
&\ar[uuu] {\red \De^{(1)}_{\varpi_2,\,\varpi_2}}\ar[d] 
\\
& \ar[ld]{\green \De^{(1)}_{\varpi_2,\,c\varpi_2}}
\\
\ar[uuu]{\red \De^{(1)}_{\varpi_1,\,c\varpi_1}} \ar[d] 
\\
{\green \De^{(1)}_{\varpi_1,\,c^2\varpi_1}}\ar[rd]
\\
& \ar[uuu]{ \De^{(0)}_{\varpi_2,\,c\varpi_2}}\ar[ld] 
\\
\ar[uu] {\De^{(0)}_{\varpi_1,\,c^2\varpi_1}}\ar[rd] 
\\
& \ar[uu]\De^{(-1)}_{\varpi_2,\,c\varpi_2}\ar[ld] 
\\
\ar[uu] \De^{(-1)}_{\varpi_1,\,c^2\varpi_1} 
\\
\ar[u] & \ar[uu]
}
\]
\caption{\label{Fig4}{\it The sequence of mutations $\tau_{\mathrm{red}}$ in type $A_2$.}}
\end{figure}

\begin{example}
{\rm
We illustrate the proof of Proposition~\ref{prop translation} in the case when
$G$ is of type $A_2$, and $c = s_1s_2 = \wt{c}$. In Figure~\ref{Fig4}, the first quiver is the labelled quiver
$\Sigma$. The second quiver is obtained from $\Sigma$ by mutation at the three red vertices. The three new red minors are calculated using the generalized minor identity \cite[Theorem 1.17]{FZ}.
The third quiver is obtained from the second by a mere rearranging of the vertices aimed at showing that its shape is obtained by just translating $\G_{\wt{c}}$ one step up. Finally the fourth quiver is obtained by rewriting certain minors of the third one using the gluing relations provided by Proposition~\ref{lem: glueing formulas}, and changing their colours in the way explained in the proof of Proposition~\ref{prop translation}. Clearly, the fourth quiver can be regarded as the first one in which every minor
$\Delta^{(s)}_{c^k(\varpi_i),\,{c}^{\,\, l}(\varpi_i)}$ has simply been replaced by $\Delta^{(s+1)}_{c^k(\varpi_i),\,{c}^{\,\, l}(\varpi_i)}$.
}
\end{example}

In order to iteratively apply Proposition \ref{prop translation}, it is convenient to change the colours of the vertices of the quiver $\tau_{\rm{red}}(\Gamma_{\wt c})$ and $\tau_{\rm{green}}(\Gamma_{\wt c})$ as explained in its proof. 
More precisely, a vertex of $\tau_{\rm{red}}(\Gamma_{\wt c})$ is red (\resp green) if in the quiver $\Gamma_{\wt c}$ it lies immediately above a red (\resp green) vertex, and it is black otherwise. 
Similarly, a  vertex of $\tau_{\rm{green}}(\Gamma_{\wt c})$ is red (\resp green) if in the quiver $\Gamma_{\wt c}$ it lies immediately below a red (\resp green) vertex, and it is black otherwise. 
We do an obviously analogue change of colours when applying $\tau_{\rm red}$ or $\tau_{\rm green}$ multiple times.

\subsection{Well-definedness of the initial seed}

Let  $K(G, c)$ be the field of fractions of the integral domain $R( G, c)$. 
We verify in this section that the elements $\De_{(i,r)} \ ( (i,r) \in \VV)$ freely generate $K(G, c)$ over the base field $K$. 
In other words, $\Sigma$ is a well-defined seed of the field $K(G, c)$.
The proof of this fact and of Theorem \ref{Thm4.4} requires  reduction to finite bands, for which we need to introduce some notation.

\medskip

Let $(i,r) \in \VV$.
In Definition \ref{def-initial-cluster-variables}, the element  $\Delta_{(i,r)}$ of $R(G,c)$ is defined as $\Delta_{u(\varpi_i),v(\varpi_i)}^{(s)}$ for some $u,v \in W$  and $s \in \Z.$
Let $s_{(i,r)}$ denote this integer $s$.
For $(M,N) \in \P$ we set 
$$
\VV_{M,N}:= \{ (i,r) \in \VV \mid M \leq s_{(i,r)} \leq N\},
$$
and we denote by $\Gamma_{ \wt c, M,N}$ the full subquiver of $\Gamma_{ \wt c}$ consisting of vertices $(i,r)$ belonging to $\VV_{M,N}$. 
This is considered as an iced-quiver by declaring that the highest and lowest vertices of any column of  $\Gamma_{ \wt c, M,N}$ are frozen, and by appropriately removing arrows between frozen vertices. 

\begin{example}
{\rm
Let, $G= SL(3)$ and $c=c_{st}= s_1s_2 = \wt{c}$.
The leftmost (rep. rightmost) part of Figure~\ref{Fig2} displays the iced-quiver $\Gamma_{\wt c, -2,1}$ (\resp $\Gamma_{\wt c, -1,1}$).
}     
\end{example}

The following proposition is a particular instance of a classical result of Berenstein, Fomin and Zelevinsky \cite{BFZ}.

\begin{Prop}
\label{Prop:seed BFZ}
 For $(i,r) \in \VV_{0,0}$, let $\wt \De_{(i,r)} \in K[G]$ be the only Fomin-Zelevinsky generalized minor such that $\De_{(i,r)}= \pi_0^* \big( \wt \De_{(i,r)} \bigr).$
The datum
\[
\left(\Gamma_{ \wt{c}, 0,0},\ \{ \wt \Delta_{(i,r)}\mid (i,r)\in \VV_{0,0}\}\right)  
\]
is an initial seed for the cluster algebra structure on the coordinate ring of the open double Bruhat cell of the group $G$
described in \cite{BFZ}.
\end{Prop}

\begin{proof}
In \cite[\S2]{BFZ}, a cluster algebra structure on the coordinate ring of an arbitrary double Bruhat cell $G^{u,v}$ of $G$ is
obtained. Various initial seeds of this structure parametrized by pairs of reduced expressions of $u$ and $v$ are explicitly described. Here we take $u=v=w_0$, and we choose a reduced expression $\bi = (i_1,\ldots, i_N)$ for $w_0$ coming from an adapted numbering of the Auslander-Reiten quiver $\GG_Q$ of the quiver $Q$ attached to the Coxeter element $c$, as defined in \S\ref{subsubsec:w_{i,k}}.

Let us consider the initial seed of $K[G^{w_0,w_0}]$ associated in \cite{BFZ} with the sequence:
\[
 \ba = (a_j\mid j\in -[1,n]\cup [1,2N]) := (-n,-n+1,\ldots,-1,i_1,-i_1,i_2,-i_2,\ldots,i_N,-i_N).
\]
To this datum, \cite[\S2.3]{BFZ} attaches two sequences of elements of $W$:
\[
u_{\le j} := \prod_{l=1;\ a_l<0}^j s_{|a_l|},\qquad
v_{> j} := \prod_{l=2N;\ a_l>0}^{j+1} s_{a_l},\qquad
(j\in -[1,n]\cup [1,2N]).
\]
It is easy to check that our special choice for $\ba$ implies that $u_{\le 1} = e$, $u_{\le 2N} = w_0$, and
\[
u_{\le 2j} = u_{\le 2j+1} = s_{i_1}\cdots s_{i_j},\qquad (1\le j < N). 
\]
Similarly, we have: 
\[
v_{> 2j-1} = v_{> 2j} = s_{i_N} s_{i_{N-1}}\cdots s_{i_{j+1}}, \qquad (1\le j \le N). 
\]
It follows from Eq.(\ref{Eq-15}) in the proof of Lemma~\ref{Lem2-4} that for $1\le j \le N$ there hold
\[
u_{\le 2j-1}(\varpi_{|a_{2j-1}|})  = s_{i_1}s_{i_2}\cdots s_{i_{j-1}}(\varpi_{i_j}) = c^{k-1}(\varpi_{i_j}),
\qquad 
u_{\le 2j}(\varpi_{|a_{2j}|})  = s_{i_1}s_{i_2}\cdots s_{i_j}(\varpi_{i_j}) = c^k(\varpi_{i_j}),
\]
where $k$ is the number of indices $p\in [1, j]$ such that $i_p = i_j$.
Similarly, taking into account that the reversed sequence $\wt{\bi} = (i_N,i_{N-1},\ldots, i_1)$ is a reduced word 
for $w_0$ adapted to $\wt{c}$, we obtain that 
\[
v_{> 2j-1}(\varpi_{|a_{2j-1}|}) = v_{> 2j}(\varpi_{|a_{2j}|}) = s_{i_N} s_{i_{N-1}}\cdots s_{i_{j+1}}(\varpi_{i_j})
= \wt{c}^{\,\,m_{i_j}-k} (\varpi_{i_j}).
\]
(The fact that $\wt{\bi}$ is adapted to $\wt{c}$ follows from the fact that the Auslander-Reiten quivers $\GG_Q$
and $\GG_{\wt{Q}}$ are mirror images of each other via a left-right symmetry.)
It follows that the cluster variables of the initial seed of $K[G^{w_0,w_0}]$ associated with $\ba$ are the generalized minors
\[
 \Delta_{c^{k-1}(\varpi_{i_j}),\ \wt{c}^{\,\,m_{i_j}-k}(\varpi_{i_j})}, 
 \qquad
 \Delta_{c^{k}(\varpi_{i_j}),\ \wt{c}^{\,\,m_{i_j}-k} (\varpi_{i_j})},
 \qquad
 (1\le j \le N),
\]
together with the variables 
\[
\De_{\varpi_i,w_0(\varpi_i)},\qquad (i\in I)
\]
corresponding to the indices $j\in -[1,n]$.
Comparing with Definition~\ref{def-initial-cluster-variables}, these cluster variables are precisely the elements of the
set $\{ \wt \Delta_{(i,r)}\mid (i,r)\in \VV_{0,0}\}$, as required.

Finally, we leave it as an exercise to the reader to check that the quiver $\wt{\G}(\ba)$ of \cite[Definition 2.2]{BFZ} 
corresponding to the initial datum $\ba$ coincides with $\Gamma_{ \wt{c}, 0,0}$, except that all arrows have opposite orientations. Since a global change of orientation does not affect the definition of the cluster algebra, this finishes the proof of the proposition.
\cqfd

\end{proof}

The cluster structure of \cite{BFZ} on the coordinate ring $K[G^{w_0,w_0}]$ of the open double Bruhat cell $G^{w_0,w_0}$ has invertible frozen variables:
\begin{equation}\label{eq-22}
 \De_{\varpi_i,w_0(\varpi_i)},\ \De_{w_0(\varpi_i), \varpi_i},\qquad (i\in I).
\end{equation}
In fact $G^{w_0,w_0}$ is precisely the open subset of $G$ defined by the non-vanishing of these frozen variables. Since by construction the ring $R(G,c)$ is obtained by gluing copies of $K[G]$ 
(not of $K[G^{w_0,w_0}]$) we will need the following adaptation of the Berenstein-Fomin-Zelevinsky theorem to $K[G]$, recently proved by Oya \cite{O}.

\begin{Thm}[\bf Oya]
\label{thm: Oya}
The ring $K[G]$ has a cluster algebra structure given by the same initial seeds as $K[G^{w_0,w_0}]$, 
except that the frozen variables of Eq.~(\ref{eq-22}) are no longer invertible. Moreover this cluster algebra coincides with its upper cluster algebra.
\end{Thm}

By a slight abuse of notation, we identify from now on the integral domain $R(G,c, M,N)$ with its isomorphic image under the homomorphism $\pi_{M,N}^*$ in the ring $R(G,c)$ (Lemma \ref{lem: generators bands}).
We set 
\begin{equation}
    \label{eq:truncated seed}
    \Sigma_{M,N}:= ( \Gamma_{ \wt c , M, N}, \ \{ (\Delta_{(i,r)} \mid (i,r) \in \VV_{M,N} \}).
\end{equation}
Note that $\De_{(i,r)}$ belongs to $R(G,c, M,N)$ whenever $M \leq s_{(i,r)} \leq N$.
We consider $\Sigma_{M,N}$ as a labeling of the quiver $\Gamma_{ \wt c , M, N}$ with elements of the ring $R(G,c , M,N).$

We denote by  $K(G,c, M,N)$ the field of fractions of the integral domain $R(G, c, M,N)$, regarded as a subfield of 
$K(G,c)$.
The pullback of rational functions under the morphism $\pi_{M,N}$ canonically identifies the field of rational functions of the space of finite bands $B(G, c, M, N)$ with $K(G, c, M, N)$. 
Notice that $K(G, c, s, s) \ (s \in \Z)$ is the field of fractions of the subalgebra $\pi_s^*(K[G])$ of $R(G,c)$.

\begin{Lem}
    \label{lem: gen rational functions}
Let $(M,N) \in \P$.    The following statements hold.
    \begin{enumerate}
        \item The equality $K(G, c)= \bigcup_{n \in \N} K(G, c , - n, +n)$ holds.
        \item The field $K(G, c, M,N)$ is generated by the subfields $\{ K(G, c, s,s) \, \mid \, s \in [M,N] \cap \Z\}.$
        \item The field $K(G, c)$ is generated by the subfields $\{ K(G, c,s,s) \, \mid \, s \in \Z\}.$
    \end{enumerate}
\end{Lem}
\begin{proof}
Notice that the set $\{(-n,n) \mid n \in \N\}$ is cofinal in $\P$. Therefore, statement 1 can be easily deduced from Lemma \ref{lem: limits algebraic}.
   Statement 2 follows from Lemma \ref{lem: generators part bands}.
   Similarly, statement 3 follows from statement 4 of Lemma \ref{lem: generators bands}.
    \cqfd
\end{proof}
\begin{Prop}
    \label{prop: variables free}
Let $(M,N) \in \P_0$. The following statements hold.
\begin{enumerate}
    \item The set $\{ \De_{(i,r)}  \mid  (i,r) \in \VV_{M,N} \}$ freely generates the field $K(G, c, M, N)$ over $K.$
    \item The set $\{ \De_{(i,r)}  \mid  (i,r) \in \VV \}$ freely generates the field $K(G, c)$ over $K$.
\end{enumerate}
\end{Prop}

\begin{proof}
    The second assertion follows from the first one and Lemma \ref{lem: gen rational functions}. 
    For the first statement, let $F_{M,N}$ be the subfield of $K(G, c,M,N)$ generated by the elements of the set $ \{\De_{(i,r)} \mid (i,r) \in \VV_{M,N} \}$ over $K$.
    We first prove that $F_{M,N}=K(G, c, M, N)$.

     \smallskip
     
   Proposition \ref{Prop:seed BFZ} implies that the set $\{ \De_{(i,r)}  \mid (i,r) \in \VV_{0,0}\}$ generates the field $K(G, c, 0,0)$ over~$K$. Therefore, the field $K(G, c, 0,0) $ is contained in $F_{M,N}.$

  Let $(i,r)$ be a mutable vertex of $\Gamma_{ \wt c, M,N}$. By applying Fomin-Zelevinsky mutation formulas to $\Sigma_{M,N}$, we obtain a labeling of the iced-quiver $\mu_{(i,r)}( \Gamma_{ \wt c, M,N})$ that we denote, with slight abuse of notation, by $\mu_{(i,r)}( \Sigma_{M,N})$.
  Note that in the quiver $\Gamma_{ \wt c}$ there is no arrow between $(i,r)$ and any vertex $(j,t)$ satisfying $s_{(j,t)} < M$ or $s_{(j,t)} > N$.
  Thus, the labeled quiver $\mu_{(i,r)}( \Sigma_{M,N})$ can be obtained by forgetting the vertices belonging to the set $\VV \setminus \VV_{M,N}$ from the labeled quiver $\mu_{(i,r)}( \Sigma)$ and appropriately freezing vertices.
Then, Proposition \ref{Prop:seed BFZ} and iteration of Proposition \ref{prop translation}  imply that the field $K(G, c, s,s)$ is contained in $F_{M,N}$ for every $s\in[M,N]$.
  By Lemma \ref{lem: gen rational functions}, we deduce that $F_{M,N}= K(G,c, M, N).$

  \smallskip
  
  For a field extension $F$ of $K$ we denote by $\trdeg_K(F)$ the transcendence degree of $F$ over $K$.
   Corollary \ref{cor: geometric prop of finite bands} and the well known Noether normalization lemma imply that 
  $$
  \trdeg_K\bigl(K(G, c, M, N)\bigr)= \dim(G)+ |I|(N-M).
  $$
 Moreover, Proposition \ref{Prop:seed BFZ} implies that
 $$
 | \VV_{0,0} |= \trdeg_K\bigl ( K(G, c, 0, 0) \bigr)= \dim(G).
 $$
  It follows that 
  $$
  | \VV_{M,N} |= \dim(G)+ |I|(N-M).
  $$
Hence, the cardinality of the set $\{ \De_{(i,r)} \ \mid \ (i,r) \in \VV_{M,N} \}$ equals the transcendence degree of the field it generates (over $K$). 
By well known properties of the transcendence degree, it follows that the elements of the set  $\{ \De_{(i,r)} \ \mid \ (i,r) \in \VV_{M,N} \}$ are algebraically independent over $K$.
  \cqfd
\end{proof}

\subsection{Reduction to finite bands}
\label{sec-reduction-finite-bands}

For $(M,N) \in \P_0$, we denote by $\AA_{M,N}$ the cluster algebra with non-invertible coefficients (frozen variables) and initial seed $(\G_{\wt{c}, M,N}, \{z_{(i,r)}\mid (i,r)\in \VV_{M,N}\})$.
Remember that the frozen vertices of the iced-quiver $\G_{\wt{c}, M,N}$ are the highest and lowest vertices of all columns. 
Let $(j,t) \in \VV$ be a vertex satisfying $s_{(j,t)} < M$ or $s_{(j,t)} > N$. 
In the quiver $\G_{\wt c}$, there is no arrow between $(j,t)$ and any vertex which is a mutable vertex of $\G_{\wt c, M, N}$. 
Hence, the cluster algebra $\AA_{M,N}$ is naturally identified with the sub-algebra of $\AA$ generated (over $\Z$) by the cluster variables of the form $\mu(z_{(i,r)})$, where $(i,r)$ belongs to $\VV_{M,N}$ and $\mu$ is a sequence of 
mutations at mutable vertices of the iced-quiver $\G_{c, M, N}.$

\begin{Thm}\label{Thm4.10}
    Let $(M,N) \in \P_0$.
The assignment 
\[
z_{(i,r)}\mapsto \De_{(i,r)},\qquad ((i,r)\in\VV_{M,N})
\]
extends to a unique algebra isomorphism from
$K\otimes \AA_{ M, N}$ to $R(G, c, M, N)$.
\end{Thm}

Let us check immediately that Theorem \ref{Thm4.4} is an easy consequence of Theorem \ref{Thm4.10}.

\medskip\noindent
\emph{Proof of Theorem \ref{Thm4.4}.---}
Proposition \ref{prop: variables free} implies that the homomorphism of $K$-algebras 
\[
\phi : K \otimes \AA \lra K(G, c)
\]
defined by the assignment of Theorem \ref{Thm4.4} is well defined.
Moreover, the restriction $\phi_{M,N} \ (M \leq 0 \leq N)$ of $\phi$ to the subalgebra $K \otimes \AA_{ M, N}$ of $K \otimes \AA$ is the homomorphism defined by the assignment of Theorem \ref{Thm4.10}.
Since any sequence of mutable vertices of the quiver $\G_{\wt c}$ is also a sequence of mutable vertices of a truncated quiver $\Gamma_{\wt c, -n,n}$ for $n$ sufficiently large, and since tensoring with $K$ commutes with colimits, we have that 
$$
K \otimes \AA= \bigcup_{n \in \N} K \otimes \AA_{- n, n}. 
$$
Then, Theorem \ref{Thm4.10} and the previous discussion imply that $\phi$ is injective and its image is the subalgebra  
$$
\bigcup_{n \in \N} R(G, c, -n, n)
$$
of $K(G, c)$.
But the latter subalgebra is $R(G, c)$ because of Lemma \ref{lem: limits algebraic} and the fact that the set $\{(-n,n) \mid n \in \N\}$ is cofinal in $\P.$
    \cqfd

The rest of this section is devoted to the proof of Theorem \ref{Thm4.10}.

\subsection{One step mutation of initial cluster variables}
\label{subsec-one-step-mutations}

The proof of Theorem \ref{Thm4.10} will follow from the Starfish lemma. 
To apply this lemma we will need the following proposition.

\begin{Prop}\label{prop: one-step-mutations}
For every $(i,r)\in\VV$, let $\Delta_{(i,r)}^*\in K(G,c)$ denote the rational function obtained
from $\De_{(i,r)}$ by a single mutation of the seed $\Sigma$ at vertex $(i,r)$. 
Then we have $\De_{(i,r)}^* \in R(G,c)$. 
\end{Prop}
\begin{proof}
If $(i,r)$ is a red or a green vertex, then it follows from Proposition~\ref{prop translation} that $\Delta_{(i,r)}^*$   
is of the form $\Delta^{(0)}_{c^{m_i-k}(\varpi_i),\, \wt{c}^{\,k+1}(\varpi_i)}$ or
$\Delta^{(-1)}_{c^{m_i-k}(\varpi_i),\, \wt{c}^{\,k+1}(\varpi_i)}$ for some $k$, so it clearly belongs to $R(G,c)$.

So we are reduced to the case when $(i,r)$ is a black vertex.
Let us exclude for the moment the special case when $G$ is of type $A_n$ and $c=c_{st} = s_1\ldots s_n$ or $c=c_{st}^{-1}$. By Remark~\ref{rem-3.11}, this implies that $m_i \ge 2$ for every $i$.

Let us assume that $(i,r)$ is a black vertex of the lower semi-infinite part of the quiver $\G_{\wt{c}}$.
In that case we have 
$\De_{(i,r)} = \De_{\varpi_i,\,\wt{c}^{m_i}(\varpi_i)}^{(s)}$
for some $s<0$. The neighbours of $\De_{(i,r)}$ in the quiver $\G_{\wt{c}}$ are of two types:
\begin{itemize}
 \item vertices of an outgoing arrow: 
 \begin{itemize}
 \item $\De_{\varpi_i,\,\wt{c}^{m_i}(\varpi_i)}^{(s+1)}$;
 \item $\De_{\varpi_j,\,\wt{c}^{m_j}(\varpi_j)}^{(s)}$ if there is an arrow $i\to j$ in $\wt{Q}$;
 \item $\De_{\varpi_j,\,\wt{c}^{m_j}(\varpi_j)}^{(s-1)}$ if there is an arrow $j\to i$ in $\wt{Q}$;
 \end{itemize}
 \item vertices of an ingoing arrow: 
 \begin{itemize}
 \item $\De_{\varpi_i,\,\wt{c}^{m_i}(\varpi_i)}^{(s-1)}$;
 \item $\De_{\varpi_j,\,\wt{c}^{m_j}(\varpi_j)}^{(s)}$ if there is an arrow $j\to i$ in $\wt{Q}$;
 \item $\De_{\varpi_j,\,\wt{c}^{m_j}(\varpi_j)}^{(s+1)}$ if there is an arrow $i\to j$ in $\wt{Q}$.
 \end{itemize}
\end{itemize}
Using the gluing relations of Proposition~\ref{lem: glueing formulas}, we can rewrite 
$\De_{(i,r)} = \De_{c(\varpi_i),\,\wt{c}^{m_i}(\varpi_i)}^{(s-1)}$, and its neighbours:
\begin{itemize}
 \item vertices of an outgoing arrow: 
 \begin{itemize}
 \item $\De_{c^2(\varpi_i),\,\wt{c}^{m_i}(\varpi_i)}^{(s-1)}$;
 \item $\De_{c(\varpi_j),\,\wt{c}^{m_j}(\varpi_j)}^{(s-1)}$ if there is an arrow $i\to j$ in $\wt{Q}$;
 \item $\De_{\varpi_j,\,\wt{c}^{m_j}(\varpi_j)}^{(s-1)}$ if there is an arrow $j\to i$ in $\wt{Q}$;
 \end{itemize}
 \item vertices of an ingoing arrow: 
 \begin{itemize}
 \item $\De_{\varpi_i,\,\wt{c}^{m_i}(\varpi_i)}^{(s-1)}$;
 \item $\De_{c(\varpi_j),\,\wt{c}^{m_j}(\varpi_j)}^{(s-1)}$ if there is an arrow $j\to i$ in $\wt{Q}$;
 \item $\De_{c^2(\varpi_j),\,\wt{c}^{m_j}(\varpi_j)}^{(s-1)}$ if there is an arrow $i\to j$ in $\wt{Q}$;
 \end{itemize}
\end{itemize}
so that the exchange relation at $\De_{(i,r)}$ can be expressed using only pullbacks of functions on the group $G$ under the morphism $\pi_{s-1}$.
(Note that here we have used that $m_i>1$, and also $m_j>1$ for all vertices $j$ of $\wt{Q}$ connected to $i$.)

What we have to show is that the rational function
\[
\De_{(i,r)}^*=
 \frac{\displaystyle\De_{c^2(\varpi_i),\,\wt{c}^{m_i}(\varpi_i)}^{(s-1)}
 \prod_{i\to j}\De_{c(\varpi_j),\,\wt{c}^{m_j}(\varpi_j)}^{(s-1)}
 \prod_{j\to i}\De_{\varpi_j,\,\wt{c}^{m_j}(\varpi_j)}^{(s-1)}
 +
 \De_{\varpi_i,\,\wt{c}^{m_i}(\varpi_i)}^{(s-1)}
 \prod_{j\to i}\De_{c(\varpi_j),\,\wt{c}^{m_j}(\varpi_j)}^{(s-1)}
 \prod_{i\to j} \De_{c^2(\varpi_j),\,\wt{c}^{m_j}(\varpi_j)}^{(s-1)}
 }
 {\De_{c(\varpi_i),\,\wt{c}^{m_i}(\varpi_i)}^{(s-1)}}
\]
is regular, or equivalently that the unique rational function on $G$ of which it is the pullback under the morphism $\pi_{s-1}$ is regular.
To prove this, we will show that this function on $G$ is the result of a one step mutation of a cluster variable of one of the standard initial seeds of \cite{BFZ} and \cite{O} on $K[G]$. 

Let $\bi = (i_1,\ldots, i_N)$ be as in the proof of Proposition~\ref{Prop:seed BFZ}, and put 
\[
 \bb = (b_j\mid j\in -[1,n]\cup [1,2N]) := (-n,-n+1,\ldots,-1,-i_1,\ldots,-i_N,i_1,\ldots,i_N).
\]
Then it is easy to check from \cite[\S 2.3]{BFZ} that the generalized minors of the initial seed corresponding 
to $\bb$ are:
\[
 \De_{c^k(\varpi_i),\,\wt{c}^{m_i}(\varpi_i)},\qquad \De_{c^{m_i}(\varpi_i),\,\wt{c}^{l}(\varpi_i)},
 \qquad (1\le i \le n,\ 0\le k \le m_i,\ 0\le l < m_i).
\]
In particular this initial seed contains all the generalized minors whose pullback occurs in the expression
of $\De_{(i,r)}^*$. 
Moreover, one can also check using \cite[Definition 2.2]{BFZ} that the neighbours of the cluster variable
$\De_{c(\varpi_i),\,\wt{c}^{m_i}(\varpi_i)}$ in this seed are exactly the generalized minors whose pull-back 
give the neighbours of $\De_{(i,r)}$ in $\G_{\wt{c}}$ described above (with all arrows going in the opposite direction). 
So we have proved that $\De_{(i,r)}^*$ is the pullback of a cluster variable of $K[G]$, hence it is
regular. Note however that in general this mutated cluster variable is not a generalized minor.

The case when $(i,r)$ is a black vertex in the upper semi-infinite part of $\G_{\wt{c}}$ is dealt with in the same
manner. In that case we have $\De_{(i,r)} = \De_{c^{m_i}(\varpi_i),\,\varpi_i}^{(s)}$ for some $s\ge 0$.
We can check in the same way that $\De_{(i,r)}^*$ is the pullback of a 
rational function on the group $G$ under the morphism $\pi_{s+1}.$
Then replacing $\bb$ by
\[
  \bc = (c_j\mid j\in -[1,n]\cup [1,2N]) := (-n,-n+1,\ldots,-1,i_1,\ldots,i_N,-i_1,\ldots,-i_N),
\]
we can argue similarly that this rational function on $G$ is the result of a one step mutation of the 
standard initial seed of \cite{BFZ} associated with $\bc$.

Let us now deal with the special case when $G$ is of type $A_n$ and $c = c_{st}$. In this situation we can use the concrete description of the space of bands given in Section~\ref{sec:SL(n) cst}, and deduce our claim from the following cubic identity satisfied by minors of a matrix.

\begin{Lem}\label{lem-4.14}
Let $k\ge 2$, and let $A:=[a_{ij}]$ be a $(k+1)\times k$-matrix with entries in a commutative ring. 
For $P\subset [1,k+1]$ and $Q\subset [1,k]$, we denote by $A_{P,Q}$ the submatrix of $A$ with row indices in $P$ and column indices in $Q$.
Moreover, if $|P|=|Q|$ we put $\De_{P,Q} = \det A_{P,Q}$. Then the following identity holds:
\[
\De_{[2,k],[1,k-1]}
\left(\De_{[1,k-1],[1,k-1]} \De_{[3,k+1],[1,k]\setminus \{k-1\}}
- \De_{[1,k-1],[1,k]\setminus \{k-1\}}\De_{[3,k+1],[1,k-1]}  
      \right)      
\]
\[
= \ \De_{[2,k-1],[1,k-2]}\De_{[3,k+1],[1,k-1]}\De_{[1,k],[1,k]}
+\De_{[3,k],[1,k-2]}\De_{[1,k-1],[1,k-1]}\De_{[2,k+1],[1,k]}. 
\]
\end{Lem}

\medskip\noindent
\emph{Proof of Lemma \ref{lem-4.14}.---}
This identity can be proved by a suitable application of Turnbull's identity (see \cite[Proposition 1.2.2]{L}).
Alternatively, one can first check by hand the identities for $k = 2$
\[
 a_{21}(a_{11}a_{32}-a_{12}a_{31}) = 
 a_{31}
 \left|
 \begin{array}{cc}
 a_{11}&a_{12}\\
 a_{21}&a_{22}
 \end{array}
 \right|
 + 
a_{11}
 \left|
 \begin{array}{cc}
 a_{21}&a_{22}\\
 a_{31}&a_{32}
 \end{array}
 \right|, 
\]
and $k=3$
\[
\left|
 \begin{array}{cc}
 a_{21}&a_{22}\\
 a_{31}&a_{32}
 \end{array}
 \right| 
 \left(
\left|
 \begin{array}{cc}
 a_{11}&a_{12}\\
 a_{21}&a_{22}
 \end{array}
 \right| 
 \left|
 \begin{array}{cc}
 a_{31}&a_{33}\\
 a_{41}&a_{43}
 \end{array}
 \right| 
 -
\left|
 \begin{array}{cc}
 a_{11}&a_{13}\\
 a_{21}&a_{23}
 \end{array}
 \right| 
 \left|
 \begin{array}{cc}
 a_{31}&a_{32}\\
 a_{41}&a_{42}
 \end{array}
 \right|  
 \right)
\]
\[
=
a_{21}
\left|
 \begin{array}{cc}
 a_{31}&a_{32}\\
 a_{41}&a_{42}
 \end{array}
 \right|  
 \left|
 \begin{array}{ccc}
 a_{11}&a_{12}&a_{13}\\
 a_{21}&a_{22}&a_{23}\\
 a_{31}&a_{32}&a_{33}
 \end{array}
 \right|  
+
a_{31}
\left|
 \begin{array}{cc}
 a_{11}&a_{12}\\
 a_{21}&a_{22}
 \end{array}
 \right|  
 \left|
 \begin{array}{ccc}
 a_{21}&a_{22}&a_{23}\\
 a_{31}&a_{32}&a_{33}\\
 a_{41}&a_{42}&a_{43}
 \end{array}
 \right|.  
\]
Then, one can readily deduce the general case $k>3$ using Muir's law of extensible minors (see \cite[\S2.3.3]{KL}). Indeed, one can notice that every submatrix $A_{P,Q}$ involved in the general 
identity contains the common $(k-3)\times (k-3)$-submatrix $A_{[3,k-1],[1,k-3]}$. It follows that the general identity is nothing but the extension of the $k=3$ case, where we use $A_{[3,k-1],[1,k-3]}$ as pivot of the extension.
\cqfd 

Let now $\mathrm{B} = [b_{kl}]$ be an $\infty\times n$-matrix representing an $(SL(n), c_{st})$-band as in Section~\ref{sec:SL(n) cst}, and let $i\in I =\{1,\ldots,n-1\}$.
If we apply Lemma~\ref{lem-4.14} to the $(i+2)\times(i+1)$-submatrix 
\[
A = \mathrm{B}_{[s+n-i,\,s+n+1],\,[1,\,i+1]} 
\]
of $\mathrm{B}$, we get the identity
\[
 \De^{(s+1)}_{w_0(\varpi_{i}),\,\varpi_{i}}
 \left(\De^{(s)}_{w_0(\varpi_{i}),\,\varpi_{i}}\De^{(s+2)}_{w_0(\varpi_{i}),\,s_{i}(\varpi_{i})}
 -
 \De^{(s)}_{w_0(\varpi_{i}),\,s_{i}(\varpi_{i})}\De^{(s+2)}_{w_0(\varpi_{i}),\,\varpi_{i}} 
 \right)
\]
\[
=\ \De^{(s)}_{w_0(\varpi_{i-1}),\,\varpi_{i-1}}\De^{(s+2)}_{w_0(\varpi_{i}),\,\varpi_{i}}
\De^{(s+1)}_{w_0(\varpi_{i+1}),\,\varpi_{i+1}}
+
\De^{(s+1)}_{w_0(\varpi_{i-1}),\,\varpi_{i-1}}\De^{(s)}_{w_0(\varpi_{i}),\,\varpi_{i}}
\De^{(s+2)}_{w_0(\varpi_{i+1}),\,\varpi_{i+1}},
\]
where we use the convention that $\Delta^{(s)}_{u(\varpi_k),v(\varpi_k)} = 1$ if $k=0$ or $k=n$.
The right hand side of this equation is precisely the right hand side of the exchange relation
of the black cluster variable $\De_{(i,r)}= \De^{(s+1)}_{w_0(\varpi_{i}),\,\varpi_{i}}$, hence
reading the left hand side we see that
\[
 \De_{(i,r)}^* = \De^{(s)}_{w_0(\varpi_{i}),\,\varpi_{i}}\De^{(s+2)}_{w_0(\varpi_{i}),\,s_{i}(\varpi_{i})}
 -
 \De^{(s)}_{w_0(\varpi_{i}),\,s_{i}(\varpi_{i})}\De^{(s+2)}_{w_0(\varpi_{i}),\,\varpi_{i}},
\]
which is a regular function. This proves our claim in the case when $(i,r)$ is a black vertex of the upper semi-infinite part of $\G_{c_{st}}$. 

The case when $(i,r)$ is in the lower semi-infinite part can be dealt with similarly, if we replace the identity of Lemma~\ref{lem-4.14} by a similar one in which each left-justified minor $\De_{P,[1,m]}$ is replaced by its right-justified companion $\De_{P,[1+n-m,n]}$. 

The case when $c= c_{st}^{-1}$ is similar and we omit it.  
This finishes the proof of Proposition~\ref{prop: one-step-mutations}. \cqfd
\end{proof}

\subsection{Proof of Theorem \ref{Thm4.10}}

For the proof of Theorem  \ref{Thm4.10} we will need a special version of the well known Starfish lemma, for which we refer the reader to \cite[Corollary 6.4.6]{FWZ}. 

\begin{remark}
\label{rem:starfish lemma}
\rm{
    Let $\R$ and $( \wt x, \wt B)$ be as in the notation of \cite[Corollary 6.4.6]{FWZ}. 
    For the proof of Corollary 6.4.6 given in \cite{FWZ} to be valid, one has to assume that the invertible elements of the ring $\R$ are constant. 
    Indeed, the authors assume in the proof that two non coprime irreducible elements of  $\R$ differ by multiplication by an element of $K^\times.$
    In general, two non coprime irreducible elements of the unique factorization domain $\R$ differ by multiplication by an element of~$\R^\times$.
    Moreover, under the necessary assumption that $\R^\times= K^\times$, the proof of \cite[Corollary 6.4.6]{FWZ} also implies that the upper cluster algebra $\mathcal{U}( \wt x, \wt B)$ of the seed $( \wt x, \wt B)$ is contained in the ring~$\R$.
    }
\end{remark}

\medskip\noindent
\emph{Proof of Theorem \ref{Thm4.10}.---}
Proposition \ref{prop: variables free} implies that the homomorphism of $K$-algebras $\phi_{M,N} : K \otimes \AA_{M,N} \lra K(G, c, M,N)$ defined by the assignment of Theorem \ref{Thm4.10} is well defined and  injective. 
Therefore, we only have to prove that the image $A_{M,N}$ of the homomorphism $\phi_{M,N}$ equals $R(G, c , M,N).$

We first prove that $R(G,  c, M,N)$ is contained in $A_{M,N}.$
By Lemma \ref{lem: generators part bands}, we have that the algebra $R(G, c , M,N)$ is generated by its subalgebras $\pi_s^*(K[G]) \ (M \leq s \leq N)$.
Therefore, it is sufficient to prove that for every integer $s \in [M,N]$, the ring $\pi_s^*(K[G])$ is contained in $A_{M,N}$. 
If $s=0$, the inclusion $\pi_s^*(K[G]) \subseteq A_{M,N}$ is a direct consequence of Theorem \ref{thm: Oya} and Proposition \ref{Prop:seed BFZ}.
For a general integer $s \in [M,N]$, the inclusion $\pi_s^*(K[G]) \subseteq A_{M,N}$ follows from Theorem \ref{thm: Oya}, Proposition \ref{Prop:seed BFZ}, and appropriate iteration of Proposition \ref{prop translation}.

Next, we prove that $A_{M,N}$ is contained in $ R(G, c, M,N)$. 
Let $\mathcal{U}_{M, N}$ be the upper cluster algebra of $\AA_{M,N}$, and $U_{M,N}$ be the image of the homomorphism $K \otimes \mathcal{U}_{ M, N} \lra K(G,  c, M,N)$ induced by the assignment of Theorem \ref{Thm4.10}. 
Note that, by the Laurent phenomenon, the algebra $A_{M,N}$ is contained in $U_{M,N}$.
Therefore, to complete the proof of the theorem, it is sufficient to show that $U_{M,N}$ is contained $ R(G,c, M,N)$. 
The aforementioned inclusion will follow from application of \cite[Corollary 6.4.6]{FWZ}. 
Thus, we now check that the conditions required to apply \cite[Corollary 6.4.6]{FWZ} are fulfilled and that $R(G, c, M,N)^\times = K^\times$ (see Remark \ref{rem:starfish lemma}).

The fact that $R(G, c, M,N)^\times=K^\times$ is a consequence of the fact that $R(G, c, M,N)$ is isomorphic to a polynomial ring with coefficients in $K[G]$ (Corollary \ref{cor: geometric prop of finite bands}) and that $K[G]^\times=K^\times$ since $G$ is semisimple.

The algebra $R(G, c, M,N)$ if clearly finitely generated, and it is a unique factorization domain by Corollary \ref{cor: geometric prop of finite bands}.
The fact that the cluster variables of the seed $\Sigma_{ M,N}$ belong to $R(G, c, M,N)$ is obvious. 
Moreover, Proposition \ref{prop: one-step-mutations} guarantees that the cluster variables obtained by a single mutation of the seed $\Sigma_{M,N}$ belong to $R(G, c, M,N)$.
Note that, we have already proved that $R(G, c, M, N)$ is contained in $A_{M,N}$.
Moreover, notice that $A_{M,N}^\times = K^\times$ because of \cite[Theorem 1.3]{GLS}.
Then, let $x$ be a cluster variable of the seed $\Sigma_{ M,N}$ or of a seed obtained by a single mutation of  $\Sigma_{ M,N}$.
Using that $x$ is an irreducible element of $A_{M,N}$ because of \cite[Theorem 1.3]{GLS} and the fact that $R(G,c,M,N)^\times=A_{M,N}^\times$, it follows that
$x$ is also irreducible in the ring $R(G, c, M, N)$.
Then, \cite[Corollary 6.4.6]{FWZ} implies that $U_{M,N}$ is contained in $R(G,c, M,N)$. 
To sum up, we have proved that $A_{M,N}= U_{M,N}= R(G, c, M,N)$.
\cqfd

The proof of Theorem \ref{Thm4.10} also shows that the following corollary holds.

\begin{Cor} 
\label{cor:A equals U finite}
    Let $(M,N) \in \P_0$. The cluster algebra $K \otimes\AA_{M,N}$ is equal to its upper cluster algebra $K \otimes \mathcal{U}_{M,N}$.
\end{Cor}

The previous result can be readily generalised to the infinite rank cluster algebra $\AA$. 

\begin{Cor}
\label{cor:AU R(G,c)}
    Let $\mathcal{U}$ be the upper cluster algebra with initial seed $\Sigma$. 
    Then $K \otimes\AA= K \otimes \mathcal{U}$.
\end{Cor}

For the sake of clarity, we specify that the upper cluster algebra  $\mathcal{U}$ is defined as the intersection, over all the seeds $\Sigma'$ mutation equivalent to $\Sigma$, of the Laurent polynomial rings in the cluster variables of the seed $\Sigma$.
This follows the standard definition of upper cluster algebras of finite rank.

\begin{proof}
We claim that for every $(M,N) \in \P_0$, the seed $\Sigma_{M,N}$ is of maximal rank. 
In other words:  the exchange matrix $B_{M,N}$ of $\Sigma_{M,N}$ is of maximal rank.

Let us start by proving the statement assuming the claim.
The claim and \cite[Proposition 2.22]{Q2} imply that the algebra $\mathcal{U}_{M,N}$ is contained in $\mathcal{U}_{M',N'}$ whenever $(M,N) \leq (M',N')$.
Moreover, by \cite[Lemma 4.3]{Q2}, we have that all these upper cluster algebras are contained in $\mathcal{U}$, and
$$
    \mathcal{U}= \bigcup_{N \in \Z_{\ge 0}} \mathcal{U}_{-N,N}.
$$
Using that the tensor product with $K$ preserves colimits, we deduce that
\begin{equation}
    \label{eq:upper limit R(G,c)}
   K \otimes  \mathcal{U}= \bigcup_{N \in \Z_{\ge 0}} K \otimes \mathcal{U}_{-N,N}.
\end{equation}
But notice that we also have that 
$$
K \otimes \AA= \bigcup_{N \in \Z_{\ge 0}} K \otimes \mathcal{A}_{M,N}.
$$
Hence, Corollary \ref{cor:A equals U finite} allows to conclude.

We now prove the claim. 
The seed $\Sigma_{0,0}$ is of maximal rank because of Proposition \ref{Prop:seed BFZ} and the fact that the seeds of the cluster structure of the coordinate rings of double Bruhat cells are of maximal rank by \cite[Proposition 2.6]{BFZ}.
Next, assume that $N >M$. Choose a total order $\leq$ on $\VV$ such that the inequality $(i,r) < (j,t)$ holds whenever $s_{(i,r)} > s_{(j,t)}$.
Then, arranging the rows and the columns of $B_{M,N}$ according to this order we get a block  decomposition
$$
B_{M,N}=\pmatrix{
B_{M,N-1} & D_{M,N} \cr
0 & C_{M,N}
},
$$
for some matrices $D_{M,N}$ and $C_{M,N}$ with integer coefficients. 
Notice that the matrix $C_{M,N}$ is the $n \times n$ matrix that accounts for the arrows between the lowest frozen vertices of the quiver $\Gamma_{\wt c, M,N}$ and the mutable vertices immediately above them.
It is immediate to verify that the matrix $C_{M,N}$ is invertible.
Thus, arguing by induction, we deduce that for every $N \in \N$ the claim holds  for the seed $\Sigma_{0,N}$.
The general statement for the seed $\Sigma_{M,N}$ can then be proved in an analogous way by induction on $M$.
\qed
\end{proof}

\section{$U$-invariants}
\label{sec_U_inv}

Recall the right $G$-action on $B(G,c)$ introduced in \S\ref{subsec-action}, and the induced left $G$-action on $R(G,c)$.
In this section we consider the restriction of these actions to the unipotent subgroup $U$, and we
study the subalgebra $R(G,c)^U$ of $U$-invariant functions.

\subsection{Preliminaries}

We start with the following easy lemma.

\begin{Lem}\label{Lem-6.1}
 For every $s\in\Z$, $i\in I$, $w\in W$, the regular function $\De^{(s)}_{w(\varpi_i),\varpi_i}$ is $U$-invariant. 
\end{Lem}

\begin{proof}
By definition of the generalized minors, for every $i\in I$ and $w\in W$ we have
\[
\De_{w(\varpi_i),\varpi_i}(gu) = \De_{w(\varpi_i),\varpi_i}(g),\qquad (g\in G,\ u\in U). 
\]
It then follows from the definition of the $G$-action on $B(G,c)$ that 
$\De^{(s)}_{w(\varpi_i),\varpi_i}$ is $U$-invariant in $R(G,c)$. 
\cqfd
\end{proof}

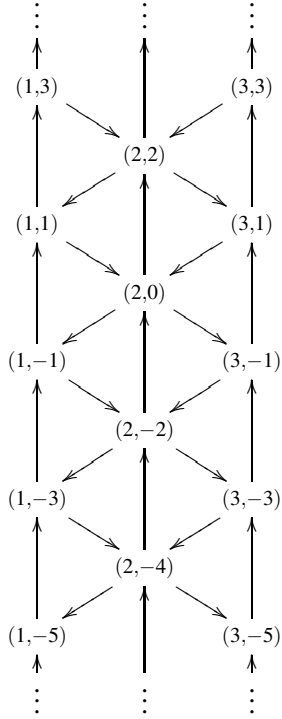
\begin{figure}[t]
\[
\def\objectstyle{\scriptstyle}
\def\lablestyle{\scriptstyle}
\xymatrix@-1.0pc{
&{}\save[]+<0cm,1.5ex>*{\vdots}\restore&{}\save[]+<0cm,1.5ex>*{\vdots}\restore  
&{}\save[]+<0cm,1.5ex>*{\vdots}\restore
\\
&{(1,3)}\ar[rd]\ar[u]&
&\ar[ld] (3,3) \ar[u]
\\
&&\ar[ld] (2,2) \ar[rd]\ar[uu]&&
\\
&\ar[uu]{(1,1)}\ar[rd]&
&\ar[ld] (3,1) \ar[uu]
\\
&&\ar[uu]\ar[ld] (2,0) \ar[rd]&&
\\
&\ar[uu](1,-1) \ar[rd] &&\ar[ld] (3,-1)\ar[uu]
\\
&&\ar[ld] \ar[uu](2,-2) \ar[rd]&&
\\
&\ar[uu](1,-3) \ar[rd] &&\ar[ld] (3,-3)\ar[uu]
\\
&&\ar[ld] \ar[uu](2,-4) \ar[rd]&&
\\
&\ar[uu](1,-5) &&\ar[uu] (3,-5) 
\\
&{}\save[]+<0cm,0ex>*{\vdots}\ar[u]\restore&{}\save[]+<0cm,0ex>*{\vdots}\ar[uu]\restore  
&{}\save[]+<0cm,0ex>*{\vdots}\ar[u]\restore
\\
}
\]
\caption{\label{Fig5} {\it The quiver $\G$ in type $A_3$.}}
\end{figure}

We will now introduce a second infinite rank cluster algebra $\B$. The defining quiver of $\B$
is the quiver $\Gamma$ of \cite[\S2.3]{GHL}. Its vertex set $\VV$ is identical to the vertex set 
of the defining quiver $\Gamma_{\widetilde{c}}$ of the cluster algebra $\AA$. Its arrow set 
is defined as follows. If $(i,r)$ and $(j,s)$ are elements of $\VV$, there is an arrow from 
$(i,r)$ to $(j,s)$ in $\Gamma$ if and only if $c_{ij}\not = 0$, and $s = r + c_{ij}$.
Here, the integers $c_{ij}$ are the entries of the Cartan matrix of $G$, see \S\ref{sec_notation}.
For instance, the quiver $\Gamma$ in type $A_3$ is displayed in Figure~\ref{Fig5}.

Note that $\Gamma$ 
can be seen as the limit of the sequence of quivers obtained from $\Gamma_{\widetilde{c}}$ by translating
lower and lower the middle part consisting of green and red vertices. In the limit this middle part
and the lower black infinite part disappear and we are left with the entirely black quiver $\Gamma$.
This remark leads to the following proposition.

\begin{Prop}\label{Prop-6.2}
Let $x_{(i,r)}$ denote the initial cluster variable of the cluster algebra $\B$ coresponding to $(i,r)\in\VV$. 
The assignment 
\[
 x_{(i,r)} \mapsto \De^{(r_i-1)}_{w_0(\varpi_i),\,\varpi_i}
\]
extends to an injective algebra homomorphism $F$ from $K\otimes \B$ to $R(G,c)^U$.
\end{Prop}

\begin{proof}
Recall the initial seed $\Sigma$ of the cluster algebra structure of $R(G,c)$ given in Theorem~\ref{Thm4.4}.
Let $\Si_{>0}$ be the sub-seed of $\Sigma$ supported on the black vertices of the upper semi-infinite part of~$\Gamma_{\widetilde{c}}$. Here we assume that the lowest cluster variables  $\Delta^{(0)}_{w_0(\varpi_i),\,\varpi_i}$ in each column of $\Si_{>0}$ are frozen.
Consider the sub-cluster algebra $R_{>0}$ of $R(G,c)$ with initial seed $\Si_{>0}$.
By Lemma~\ref{Lem-6.1}, every cluster variable $\Delta^{(s)}_{w_0(\varpi_i),\,\varpi_i}\ (s\ge 0)$ of $\Si_{>0}$ belongs to $R(G,c)^U$, hence $R_{>0}$ is contained in $R(G,c)^U$. 

Now using the translation invariance property of Proposition~\ref{prop translation}, by mutating $\Si$ at every green
vertex we will get a new seed $\tau_{\mathrm{green}}(\Si)$ of $R(G,c)$ with an isomorphic quiver whose upper semi-infinite black part consists of the 
cluster variables $\Delta^{(s)}_{w_0(\varpi_i),\,\varpi_i}\ (s\ge -1)$. Let us call $\Si_{>-1}$ the sub-seed of 
$\tau_{\mathrm{green}}(\Si)$ supported on the sub-cluster 
$\{\Delta^{(s)}_{w_0(\varpi_i),\,\varpi_i}\mid s\ge -1\}$, 
where the $\Delta^{(-1)}_{w_0(\varpi_i),\,\varpi_i}$ are frozen. Let $R_{>-1}$ be the sub-cluster algebra
of $R(G,c)$ with initial seed $\Si_{>-1}$. Then we have 
\[
R_{>0} \subset R_{>-1} \subset R(G,c)^U.
\]
Iterating this procedure, we will get a sequence $R_{>M}\ (M\le 0)$ of cluster subalgebras
\[
R_{>0} \subset R_{>-1} \subset\cdots \subset  R_{>M}\subset \cdots \ R(G,c)^U, 
\]
where $R_{>M}$ has a seed $\Si_{>M}$ whose cluster is $\{\Delta^{(s)}_{w_0(\varpi_i),\,\varpi_i} \mid s\ge M\}$, and whose quiver is isomorphic to the upper semi-infinite black part of $\Gamma_{\widetilde{c}}$.

Therefore the union $\bigcup_{M\le 0} R_{>M}$ will be a cluster subalgebra of $R(G,c)^U$, with initial seed 
\begin{equation}\label{Eq.25p}
\left(\Gamma, \{\Delta^{(s)}_{w_0(\varpi_i),\,\varpi_i}\mid s\in \Z\}\right).
\end{equation}
This cluster subalgebra of $R(G,c)^U$ is nothing but the image of the cluster algebra $\B$ under the unique algebra homomorphism $F$ extending the assignment 
$x_{(i,r)} \mapsto \De^{(r_i-1)}_{w_0(\varpi_i),\,\varpi_i}$.
\cqfd
\end{proof}

Recall the isomorphism $\tau : X(G,c) \to B(G,c)$ of Corollary~\ref{cor: bands product}.
The action of $G$ on $B(G,c)$ corresponds through this isomorphism to the action of $G$ on
$X(G, c)$ induced by right multiplication of $G$ on $X_0$ (the action of $G$ on all factors $X_n$
with $n \not = 0$ is trivial). 
Recall that $X_0=G$, while $X_n = U(c^{-1})\bar{c}$ for $n \not = 0$.

\begin{Lem}
Let $\theta_i \ (i\in I)$ denote the restriction of the generalized minor $\De_{\varpi_i,\varpi_i}$
to $U(c^{-1})\bar{c}$.
The coordinate ring of $U(c^{-1})\bar{c}$ is the polynomial ring in the variables $\theta_i \ (i\in I)$. 
\end{Lem}

\begin{proof}
Let $w\in W$. It is well-known that the coordinate ring of the unipotent subgroup $U(w)$ has the structure
of a cluster algebra with frozen variables the regular functions $u\mapsto \De_{\varpi_i,w^{-1}\varpi_i}(u)$ \cite{GLS1}.
When $w$ is a Coxeter element, this cluster structure is trivial in the sense that all cluster variables
are frozen. In other words, the coordinate ring of the unipotent subgroup $U(c^{-1})$ is the polynomial ring in the regular functions $u \mapsto \De_{\varpi_i,c\varpi_i}(u) = \De_{\varpi_i,\varpi_i}(u\bar{c})$.
The claim follows immediately. \cqfd
\end{proof}

For $i\in I$ and $s\in\Z$, define the function $\theta_i^{(s)}\in R(G,c)$ by 
\[
 \theta_i^{(s)}\left((g(t))_{t\in\Z}\right) = \theta_i(g(s)g(s+1)^{-1}).
\]
Then clearly $\theta_i^{(s)}\in R(G,c)^G \subset R(G,c)^U$. Moreover the previous discussion shows that
\begin{Cor}\label{Cor-6.4}
The invariant ring $R(G,c)^U$ is the polynomial ring in the variables 
\[
\theta_i^{(s)}, \quad (i\in I,\ s\in\Z),
\]
with coefficients in $\pi_0^*(K[G]^U)$. \cqfd
\end{Cor}

The next proposition plays an important role in the sequel.

\begin{Prop}\label{Prop-6.5}
For every $i\in I$ and $s\in\Z$, the function $\theta_i^{(s)}$ is a scalar multiple of a cluster variable of the 
cluster algebra $F(\B)\subset R(G,c)^U$.
\end{Prop}

The proof of Proposition~\ref{Prop-6.5} will require some preparation.

First we note that restricting the action of $G$ on $R(G,c)$ to the maximal torus $T$, we get a $P$-grading
of the algebra $R(G,c)$, namely
\begin{equation}\label{Eq_25}
R(G,c) = \bigoplus_{\mu\in P} R(G,c)_\mu,
\end{equation}
where $R(G,c)_\mu$ denotes the weight space of weight $\mu$.
The regular functions $\Delta^{(s)}_{u(\varpi_i),v(\varpi_i)}$ and $\theta^{(s)}_i$ are homogeneous of degree
\[
\deg\left(\Delta^{(s)}_{u(\varpi_i),v(\varpi_i)}\right) = v(\varpi_i),
\quad
\deg\left(\theta^{(s)}_i\right) = 0.
\]
Next, noting that the right action of $T$ on $G$ induces a grading of $K[G]^U$ by the positive cone $P^+$ of
dominant weights, we have
\[
 K[G]^U = \bigoplus_{\la\in P^+} \left(K[G]^U\right)_\la,
\]
and taking into account Corollary~\ref{Cor-6.4}, we get that the subalgebra $R(G,c)^U$ is $P^+$-graded
\[
 R(G,c)^U = \bigoplus_{\la\in P^+} \left(R(G,c)^U\right)_\la,
\]
with 
\begin{equation}\label{Eq.27}
\left(R(G,c)^U\right)_0 = K[\theta^{(s)}_i\mid i\in I,\ s\in \Z] = R(G,c)^G.
\end{equation}

Recall from \S\ref{subsec-action} that we also have a left action of $T$ on $B(G,c)$ commuting with the right action.
This allows us to refine the $P$-grading of Eq.~(\ref{Eq_25}) into a $P\times P$-grading:
\[
R(G,c)_\mu = \bigoplus_{\nu\in P} {}_\nu\, R(G,c)_\mu, 
\]
The regular functions $\Delta^{(s)}_{u(\varpi_i),v(\varpi_i)}$ and $\theta^{(s)}_i$ are also homogeneous
for this left degree:
\begin{equation}\label{Eq.28}
\ldeg\left(\Delta^{(s)}_{u(\varpi_i),v(\varpi_i)}\right) = c^su(\varpi_i),
\quad
\ldeg\left(\theta^{(s)}_i\right) = c^s(\varpi_i) - c^{s+1}(\varpi_i).  
\end{equation}

We now remark that the subalgebra $F(\B)$ of $R(G,c)^U$ is a $P\times P^+$-graded cluster algebra, 
in the sense of \cite{GL}.
Indeed, the initial seed of Eq.~(\ref{Eq.25p}) consists of homogeneous cluster variables.
Moreover, it is easy to check that for every vertex $v$ of the quiver $\Gamma$, the sum of the right degrees of the cluster variables sitting at the targets of the arrows going out of $v$ is equal to the sum of the right degrees of the cluster variables sitting at the sources of the arrows going into $v$.
The same property also holds for left degrees. This can be checked by using the interpretation of the weights
$c^k(\varpi_i)$ in terms of Auslander-Reiten theory explained in \S\ref{subsubsec:w_{i,k}}, since it then amounts
to additivity of dimension vectors for almost split sequences. (We leave this verification to the reader, but see 
Example~\ref{Ex.6.6} below.) 

As a final preparation, we note that, because of the translation invariance property of Proposition~\ref{prop translation}, it is enough to prove Proposition~\ref{Prop-6.5} in the case $s=0$. To do this we can work in 
a cluster subalgebra of $F(\B)$ of finite cluster rank. Namely let $\Gamma^{(0)}$ be the finite full subquiver
of $\Gamma$ supported on vertices $(i,r)\in \VV$ with $-m_i \le r_i-1 \le 1$. As usual, we regard the extremal vertices 
in each column of $\Gamma^{(0)}$ as frozen. Let $\B^{(0)}$ be the cluster subalgebra of $\B$ supported on $\Gamma^{(0)}$,
and $F(\B^{(0)})$ be its isomorphic image in $R(G,c)^U$. 

\begin{example}\label{Ex.6.6}
{\rm
In type $A_3$ for $c = s_1s_3s_2$ the initial seed of $F(\B^{(0)})$ and its bi-grading are displayed
in Figure~\ref{Fig6}. The homogeneity of left gradings at vertex $(c\varpi_2,\varpi_2)$ for instance amounts to the equality
\[
c^2(\varpi_2) + c(\varpi_1) + c(\varpi_3) = \varpi_2 + c^2(\varpi_1) + c^2(\varpi_3),  
\]
which is equivalent to
\[
(\varpi_2-c(\varpi_2)) + (c(\varpi_2)-c^2(\varpi_2)) = (c(\varpi_1)-c^2(\varpi_1)) + (c(\varpi_3)-c^2(\varpi_3)), 
\]
which in turn amounts to the additivity of dimension vectors for the almost split sequence 
\[
\tau(I_2) \to \tau(I_1) \oplus \tau(I_3) \to I_2 
\]
in the category of representations of the quiver $1 \to 2 \leftarrow 3$ associated with $c$. \cqfd
}
\end{example}

\begin{figure}[t]
\[
\def\objectstyle{\scriptstyle}
\def\lablestyle{\scriptstyle}
\xymatrix@-1.0pc{
&&\ar[ld] {\blue \De^{(1)}_{w_0(\varpi_2),\varpi_2}} \ar[rd]&&
\\
&{\blue \De^{(1)}_{w_0(\varpi_1),\varpi_1}}\ar[rd]&
&\ar[ld] {\blue \De^{(1)}_{w_0(\varpi_3),\varpi_3}}
\\
&&\ar[uu]\ar[ld] \De^{(0)}_{w_0(\varpi_2),\varpi_2} \ar[rd]&&
\\
&\ar[uu]\De^{(0)}_{w_0(\varpi_1),\varpi_1} \ar[rd] &&\ar[ld] \De^{(0)}_{w_0(\varpi_3),\varpi_3}\ar[uu]
\\
&&\ar[ld] \ar[uu]\De^{(-1)}_{w_0(\varpi_2),\varpi_2} \ar[rd]&&
\\
&\ar[uu] \De^{(-1)}_{w_0(\varpi_1),\varpi_1} \ar[rd] &&\ar[ld] \De^{(-1)}_{w_0(\varpi_3),\varpi_3}\ar[uu]
\\
&&\ar[ld] \ar[uu] {\blue \De^{(-2)}_{w_0(\varpi_2),\varpi_2}} \ar[rd]&&
\\
&\ar[uu]{\blue \De^{(-2)}_{w_0(\varpi_1),\varpi_1}} &&\ar[uu] {\blue \De^{(-2)}_{w_0(\varpi_3),\varpi_3}} 
}
\qquad
\xymatrix@-1.0pc{
&&\ar[ld] {\blue (c^3(\varpi_2),\varpi_2)} \ar[rd]&&
\\
&{\blue (c^3(\varpi_1),\varpi_1)}\ar[rd]&
&\ar[ld] {\blue (c^3(\varpi_3),\varpi_3)}
\\
&&\ar[uu]\ar[ld] (c^2(\varpi_2),\varpi_2) \ar[rd]&&
\\
&\ar[uu] (c^2(\varpi_1),\varpi_1) \ar[rd] &&\ar[ld] (c^2(\varpi_3),\varpi_3)\ar[uu]
\\
&&\ar[ld] \ar[uu] (c(\varpi_2),\varpi_2) \ar[rd]&&
\\
&\ar[uu] (c(\varpi_1),\varpi_1) \ar[rd] &&\ar[ld] (c(\varpi_3),\varpi_3)\ar[uu]
\\
&&\ar[ld] \ar[uu] {\blue (\varpi_2,\varpi_2)} \ar[rd]&&
\\
&\ar[uu]{\blue (\varpi_1,\varpi_1)} &&\ar[uu] {\blue (\varpi_3,\varpi_3)} 
}
\]
\caption{\label{Fig6} {\it The initial seed of $F(\B^{(0)})$ in type $A_3$ for $\widetilde{c}=s_2s_1s_3$, and its bi-grading.}}
\end{figure}
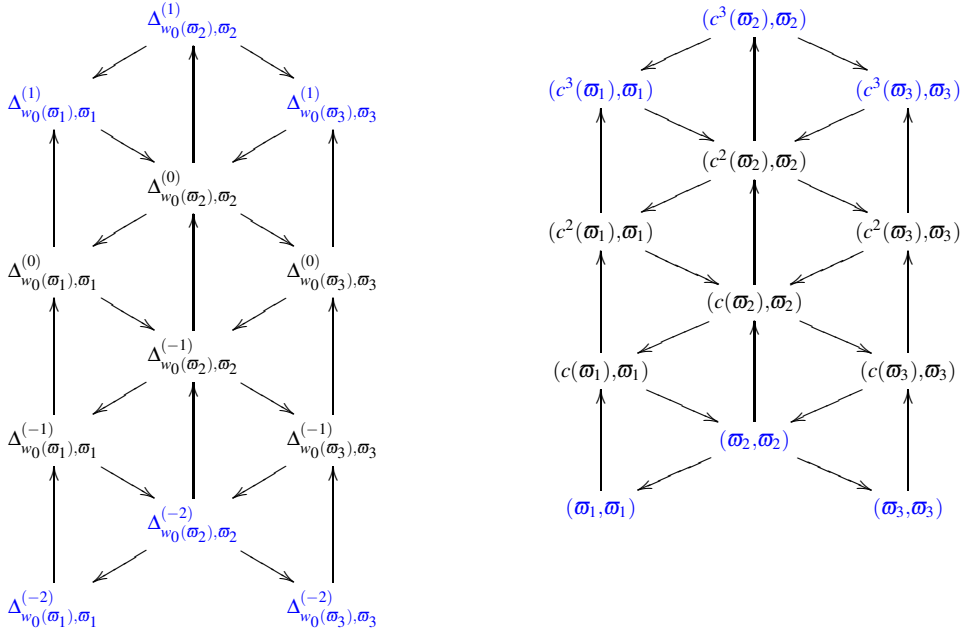

Using Proposition~\ref{lem: glueing formulas}, we can rewrite the initial cluster of $F(\B^{(0)})$
as 
\[
 \left\{\De^{(0)}_{c^k(\varpi_i),\varpi_i} \mid i\in I,\ 0\le k\le m_i \right\}
 \cup
 \left\{\De^{(1)}_{w_0(\varpi_i),\varpi_i} \mid i\in I \right\},
\]
and this shows that $F(\B^{(0)}) \subset R(G,c,0,1)^U$, where $R(G,c,0,1)$ is the ring of regular functions
on the space of finite bands $B(G,c,0,1)$. Hence $F(\B^{(0)}) \subset \pi^*_0( K[G])[\theta^{(0)}_i\mid i\in I]$.
So taking into account Eq.(\ref{Eq.27}), we see that if an element $x$ of $F(\B^{(0)})$ is of right degree 0, then
it belongs to $K[\theta^{(0)}_i\mid i\in I]$. Thus, using also the left degree and Eq.(\ref{Eq.28}), we obtain the following criterium.

\begin{Lem}\label{Lem.6.7}
If $x\in F(\B^{(0)})$ is homogeneous of bi-degree $(\varpi_i-c\varpi_i, 0)$, then $x$ is a scalar multiple
of $\theta^{(0)}_i$. \cqfd
\end{Lem}

The rest of the proof of Proposition~\ref{Prop-6.5} consists in describing an explicit sequence $M$ of
mutations of the initial seed of $F(B^{(0)})$ which produces cluster variables of the desired bi-degrees 
$(\varpi_i-c\varpi_i, 0) \ (i \in I)$. 

Write $\wt{c} = s_{i_1}\cdots s_{i_n}$. The sequence $M$ can be described as follows.
For $i\in I$ and $-m_i+1 \le j \le 0$, let us denote by $\mu_{(i,j)}$ the mutation at vertex $(i,r)$ such that
$r_i - 1 = j$. Furthermore, for $i\in I$ and $k\in \Z$ let us denote by $M_{i,k}$ the sequence of mutations 
$\mu_{(i,0)}$, $\mu_{(i,-1)}$, \ldots, $\mu_{(i,k)}$ if  $k\in [-m_i + 1,0]$, and $M_{i,k} = \emptyset$ otherwise. 
Then, we set:
\begin{eqnarray*}
 M &:=& M_{i_1,-m_{i_1}+1}, M_{i_2,-m_{i_2}+1}, \ldots, M_{i_n,-m_{i_n}+1},
 \\
 &&M_{i_1,-m_{i_1}+2}, M_{i_2,-m_{i_2}+2}, \ldots, M_{i_n,-m_{i_n}+2},
 \\
 &&M_{i_1,-m_{i_1}+3}, M_{i_2,-m_{i_2}+3}, \ldots, M_{i_n,-m_{i_n}+3},
 \\
 &&\ldots
\end{eqnarray*}
We can see that the total number of mutations in column $i$ is $m_i+(m_i-1)+\cdots +1 = m_i(m_i+1)/2$.
Thus
the sequence $M$ contains 
\[
 N = \sum_{i\in I}\frac{m_i(m_i+1)}{2}
\]
mutations in total.

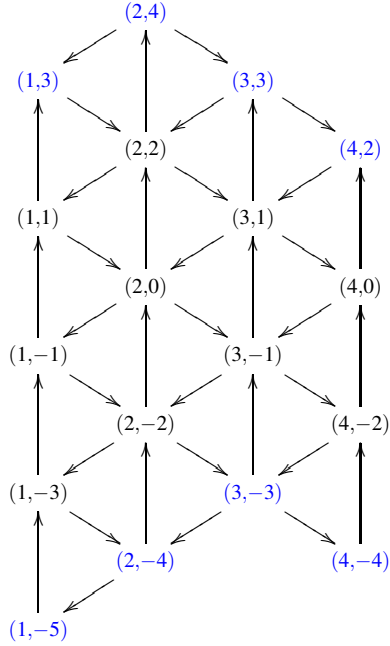
\begin{figure}[t]
\[
\def\objectstyle{\scriptstyle}
\def\lablestyle{\scriptstyle}
\xymatrix@-1.0pc{
&&\ar[ld] {\blue (2,4)} \ar[rd]&&
\\
&{\blue (1,3)}\ar[rd]&
&\ar[ld]{\blue (3,3)} \ar[rd]
\\
&&\ar[uu]\ar[ld] (2,2) \ar[rd]&&\ar[ld]{\blue (4,2)}
\\
&\ar[uu](1,1) \ar[rd] &&\ar[ld] (3,1)\ar[rd]\ar[uu]
\\
&&\ar[ld] \ar[uu](2,0) \ar[rd]&&\ar[ld](4,0)\ar[uu]
\\
&\ar[uu](1,-1) \ar[rd] &&\ar[ld] {(3,-1)}\ar[rd]\ar[uu]
\\
&&\ar[ld] \ar[uu]{ (2,-2)}\ar[rd]&& \ar[ld]{ (4,-2)}\ar[uu]
\\
&\ar[uu]{ (1,-3)}\ar[rd] && \ar[ld]{\blue (3,-3)}\ar[rd]\ar[uu]
\\
&& \ar[ld] \ar[uu]{\blue (2,-4)}&&{\blue (4,-4)}\ar[uu]
\\
&{\blue (1,-5)}\ar[uu]
}
\]
\caption{\label{Fig7} {\it The quiver $\G_0$ in type $A_4$ for $c = s_1s_2s_4s_3$.}}
\end{figure}

\begin{example}
{\rm
Let $G$ be of type $A_4$ and choose $c = s_1s_2s_4s_3$, so that $\widetilde{c} = s_2s_1s_3s_4$.
Then we have $m_1 = m_2 = 3$ and $m_3 = m_4 = 2$. The quiver $\Gamma^{(0)}$ is shown in Figure~\ref{Fig7},
with its frozen vertices painted blue.
The sequence $M$ consists of 18 mutations at the following vertices of $\Gamma^{(0)}$:
\begin{eqnarray*}
&&{(2,2)},\ {(2,0)},\ {(2,-2)},\ {(1,1)},\ {(1,-1)},\ {(1,-3)},\
{(3,1)},\ {(3,-1)},\ {(4,0)},\ {(4,-2)},
\\
&&{(2,2)},\ {(2,0)},\ {(1,1)},\ {(1,-1)},\ {(3,1)},\ {(4,0)},\
\\
&&{(2,2)},\ {(1,1)}.
\end{eqnarray*}
}
\end{example}

Recall that the weight lattice $P$ is endowed with the scalar product $(\cdot,\cdot)$ defined by
\[
 (\varpi_i, \alpha_j) = \delta_{ij},\qquad (i,j\in I).
\]
For $\la,\mu \in P$, let us write $(\la,\mu)_{\ge 0} := \max\{(\la,\mu),0\}$.

\begin{Lem}\label{Lem.6.9}
After performing the sequence of mutations $M$, the mutable cluster variable
\[
\De^{(0)}_{c^k(\varpi_i),\varpi_i} \qquad  (i\in I,\ 1\le k\le m_i)
\]
of the initial cluster of $F(\B^{(0)})$ of bi-degree $(c^k(\varpi_i),\varpi_i)$ is replaced 
by a cluster variable of bi-degree
\[
\left(\varpi_i - c^{m_i-k+1}\varpi_i + \sum_{j\in I} (\widetilde{c}^k(\varpi_i),\alpha_j)_{\ge 0}\,c^{m_j+1}\varpi_j  \ ,
\ \ \sum_{j\in I} (\widetilde{c}^k(\varpi_i),\alpha_j)_{\ge 0}\,\varpi_j\right).
\]
In particular, for $k = m_i$ we have $\widetilde{c}^{\,m_i}(\varpi_i) = -\varpi_{\nu(i)}$ hence 
$(\widetilde{c}^{\,m_i}(\varpi_i),\alpha_j)_{\ge 0} = 0$ for every $j\in I$. 
So after performing the sequence $M$ the cluster variable
$\De^{(0)}_{c^{m_i}(\varpi_i),\varpi_i}$ is replaced  by a cluster
variable of bi-degree $(\varpi_i - c\varpi_i, 0)$.
\end{Lem}

\begin{figure}[t]
\[
\def\objectstyle{\scriptstyle}
\def\lablestyle{\scriptstyle}
\xymatrix@-1.0pc{
&& {\blue (c^3(\varpi_2),\varpi_2)} &&
\\
&{\blue (c^3(\varpi_1),\varpi_1)}&
& {\blue (c^3(\varpi_3),\varpi_3)}
\\
&& (\varpi_2-c(\varpi_2), 0) &&
\\
& (\varpi_1-c(\varpi_1),0)  && (\varpi_3- c(\varpi_3), 0)
\\
&& (\varpi_2-c^2(\varpi_2)+c^3(\varpi_1)+c^3(\varpi_3),\ \varpi_1+\varpi_3) &&
\\
& (\varpi_1-c^2(\varpi_1)+c^3(\varpi_3),\varpi_3)  && (\varpi_3-c^2(\varpi_3)+c^3(\varpi_1),\varpi_1)
\\
&& {\blue (\varpi_2,\varpi_2)} &&
\\
&{\blue (\varpi_1,\varpi_1)} && {\blue (\varpi_3,\varpi_3)} 
}
\]
\caption{\label{Fig8} {\it The bi-grading of the mutated seed of $F(\B^{(0)})$ in type $A_3$ for $\widetilde{c}=s_2s_1s_3$.}}
\end{figure}
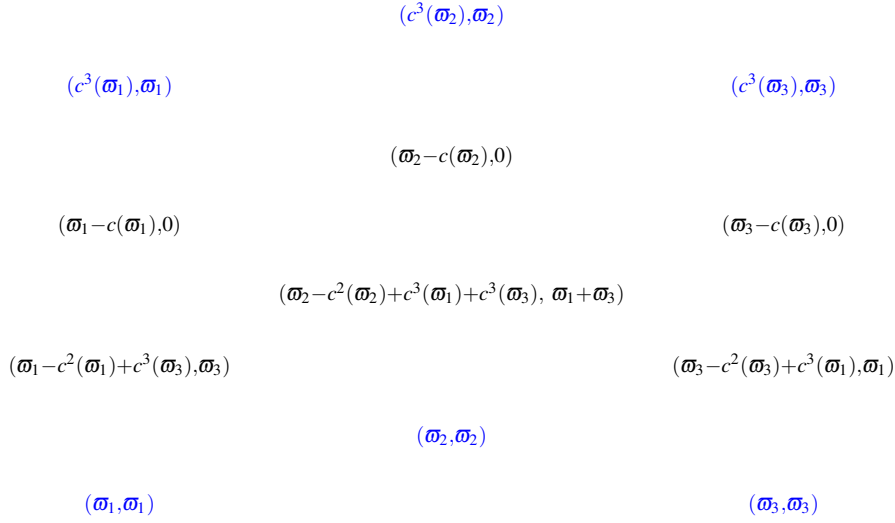

\begin{example}
{\rm
We continue Example~\ref{Ex.6.6}. Here we have $\widetilde{c} = s_2s_1s_3$ and $m_1 = m_2 = m_3 = 2$. We can calculate
\[
\widetilde{c}(\varpi_1) =  \varpi_1 - \alpha_1 -\alpha_2,\quad
\widetilde{c}(\varpi_2) =  \varpi_2 - \alpha_2,\quad
\widetilde{c}(\varpi_3) =  \varpi_3 - \alpha_2 -\alpha_3.
\]
Thus, for instance, 
\[
(\widetilde{c}(\varpi_1), \alpha_1)_{\ge 0} = (\widetilde{c}(\varpi_1), \alpha_2)_{\ge 0} = 0,
\qquad
(\widetilde{c}(\varpi_1), \alpha_3)_{\ge 0} = 1.
\]
Hence after performing the sequence of mutations $M$, the cluster variable 
$\De^{(0)}_{c(\varpi_1),\varpi_1}$ is replaced by a cluster variable of bi-degree
$(\varpi_1 - c^2\varpi_1 + c^3\varpi_3\ ,\ \ \varpi_3)$. 
The bi-degrees of all the cluster variables of the new seed obtained after performing the sequence $M$ are
displayed in Figure~\ref{Fig8}. 
}
\end{example}

\begin{proof}
Lemma~\ref{Lem.6.9} will follow immediately from the following more precise statement.

Recall that the mutation sequence $M$ is partitioned into blocks denoted by $M_{i,j}$.
For $i\in I$ and $1\le k\le m_i$, consider the vertex of $\Gamma^{(0)}$ whose initial cluster variable is
$\Delta^{(0)}_{c^k(\varpi_i),\varpi_i}$.
Let $1\le s \le m_i$, and consider the cluster variable sitting at the same vertex after
the sequence of mutations 
$M^{i,s}:=M_{i_1,-m_{i_1}+1}, M_{i_2,-m_{i_2}+1}, \ldots, M_{i_n,-m_{i_n}+1}, M_{i_1, -m_{i_1}+2}, \ldots, M_{i,-m_i+s}$. 
In other words, $M^{i,s}$ denotes the initial subsequence of $M$ that ends with the block $M_{i,-m_i+s}$.
We claim that the bi-degree of this mutated cluster variable is equal to 
\begin{eqnarray}
\left(c^{k-s}(\varpi_i)- c^{m_i-s+1}(\varpi_i) +
\sum_{j\in I} (\widetilde{c}^{\,s}(\varpi_i),\alpha_j)_{\ge 0}\,c^{m_j+1}\varpi_j\ ,\ 
\sum_{j\in I} (\widetilde{c}^{\,s}(\varpi_i),\alpha_j)_{\ge 0}\,\varpi_j\right)
& \mbox{if}\ s\le k,
\\
\left(\varpi_i- c^{m_i-k+1}(\varpi_i) +
\sum_{j\in I} (\widetilde{c}^{\,k}(\varpi_i),\alpha_j)_{\ge 0}\,c^{m_j+1}\varpi_j\ ,\ 
\sum_{j\in I} (\widetilde{c}^{\,k}(\varpi_i),\alpha_j)_{\ge 0}\,\varpi_j\right)
& \mbox{if}\ s \ge k.
\end{eqnarray}
Notice that
the cluster variable sitting at the chosen vertex after the mutation sequence $M$ is the same as the one obtained by applying $M^{i,m_i}$. Moreover, its bi-degree is given by the second equation.
Notice also that if we set $s=0$ in the 
first equation, we get 
\[
\left(c^{k}(\varpi_i)- c^{m_i+1}(\varpi_i) +
\sum_{j\in I} (\varpi_i,\alpha_j)_{\ge 0}\,c^{m_j+1}\varpi_j\ ,\ 
\sum_{j\in I} (\varpi_i,\alpha_j)_{\ge 0}\,\varpi_j\right) 
=
\left(c^{k}(\varpi_i)\ ,\ \varpi_i\right), 
\]
that is, we recover the bi-degree of the initial cluster variable.
 
The proof of this claim is by induction on the sequence of blocks of mutations $M_{j, -m_j + t}$  composing $M^{i,s}$.
The first block is $M_{i_1, -m_{i_1}+1} = \mu_{(i_1,0)}$, $\mu_{(i_1,-1)}$, \ldots, $\mu_{(i_1,-m_{i_1}+1)}$.
Consider the mutation $\mu_{(i_1,0)}$. The initial cluster variable at this vertex has bi-degree $(c^{m_{i_1}}(\varpi_{i_1}), \varpi_{i_1})$. 
The bi-degree of the mutated cluster variable can be expressed as the sum of  the bi-degrees of all the variables
sitting at sources of arrows going into this vertex, minus $(c^{m_{i_1}}(\varpi_{i_1}), \varpi_{i_1})$. By inspection of the quiver,
and noting that $\widetilde{c}$ starts with $s_{i_1}$, we see that this is equal to
\[
\sum_{j\ :\ c_{i_1j}=-1} (c^{\,m_{j}+1}(\varpi_j), \varpi_j) 
+ ({c}^{\,m_{i_1}-1}(\varpi_{i_1}), \varpi_{i_1}) - ({c}^{\,m_{i_1}}(\varpi_{i_1}), \varpi_{i_1}).
\]
Now, since $\widetilde{c}(\varpi_{i_1}) = s_{i_1}(\varpi_{i_1}) = \varpi_{i_1} - \alpha_{i_1}$, we obtain
that
$(\widetilde{c}(\varpi_{i_1}),\alpha_j)_{\ge 0} = 1$ 
if $c_{i_1j}=-1$, and $(\widetilde{c}(\varpi_{i_1}),\alpha_j)_{\ge 0} = 0$ otherwise. 
Hence this bi-degree can be rewritten as
\[
\left(c^{m_{i_1}-1}(\varpi_i)- c^{m_{i_1}}(\varpi_{i_1}) +
\sum_{j\in I} (\widetilde{c}(\varpi_{i_1}),\alpha_j)_{\ge 0}\,c^{m_j+1}\varpi_j\ ,\ 
\sum_{j\in I} (\widetilde{c}(\varpi_{i_1}),\alpha_j)_{\ge 0}\,\varpi_j\right), 
\]
as claimed. 

To prove the claim for the $m_{i_1}-1$ remaining mutations of $M_{i_1, -m_{i_1} + 1}$, 
one readily checks that
quiver mutation rules imply the following important feature of this particular mutation sequence:

\begin{quote}
(A) \emph{when each of the mutations $\mu_{(i_1,-1)}, \ldots ,\mu_{(i_1,-m_{i_1}+1)}$ is performed, there are only two arrows going \emph{into} the mutation vertex, namely the vertical arrow coming from the vertex immediately above and the vertical arrow coming from the vertex immediately below in the same column.} 
\end{quote}
Thus the desired formula readily follows
by induction from the above formula for the first mutation $\mu_{(i_1,0)}$ of $M_{i_1, -m_{i_1} + 1}$.

The calculation for the second group $M_{i_2, -m_{i_2} + 1}$ of mutations is entirely similar. The only 
difference is that after performing the sequence of mutations $M_{i_1, -m_{i_1} + 1}$, certain arrows of the 
initial quiver $\Gamma^{(0)}$ adjacent to the vertex where
$\mu_{(i_2,0)}$ is performed may have been 
modified. 
In fact, it is enough to know what are the arrows going \emph{into} this vertex at this stage to be able 
to calculate the bi-degree of the mutated variable. 

The second important feature of our particular mutation sequence, which can be deduced by induction from the quiver mutation rules is:

\begin{quote}
(B) \emph{when performing the \emph{first} mutation of the group $M_{j, -m_{j} + t}$, for every $k\in I$ the number of arrows with source the frozen vertex $(c^{m_k+1}(\varpi_k),\varpi_k)$ going into the vertex where the mutation is performed is equal to 
$(\widetilde{c}^{\ t}(\varpi_j),\alpha_k)_{\ge 0}$. The only other arrow going into this vertex is the 
vertical arrow coming from the vertex immediately below in the same column $j$.
} 
\end{quote}
Using this, we readily obtain our claim for the first mutation of $M_{i_2, -m_{i_2} + 1}$. Then for the remaining
$m_{i_2}-1$ mutations of $M_{i_2, -m_{i_2} + 1}$, we observe that statement (A) above is again valid for all these
mutations, and we deduce our claim for the entire sub-sequence $M_{i_2, -m_{i_2} + 1}$ of $M$.

The rest of the calculation proceeds in exactly the same way, applying rule (B) for the mutation
at the first vertex of each block $M_{j, -m_{j} + t}$, and rule (A) at every other vertex.
\cqfd
\end{proof}

We can now finish the proof of Proposition~\ref{Prop-6.5}.
By Lemma~\ref{Lem.6.9}, for every $i\in I$, the cluster algebra $F(\B^{(0)})$ contains a cluster variable of bi-degree
$(\varpi_i-c(\varpi_i), 0)$. By Lemma~\ref{Lem.6.7}, this cluster variable is a scalar multiple of $\theta^{(0)}_i$.
Finally by the translation invariance property, the same is true for $\theta^{(s)}_i$ for every $s\in \Z$.
\cqfd

\subsection{Upper cluster structure}

We are now in a position to state the main result of this section. Recall the cluster algebra $\B$ with initial seed
$(\Gamma, \{x_{(i,r)}\mid (i,r)\in \VV\})$, and the injective algebra homomorphism $F$ from $\B$ to $R(G,c)^U$.
Let $\Xi = (\Gamma, \{F(x_{(i,r)})\mid (i,r)\in \VV\})$ denote the initial seed of the cluster algebra $F(B)$.

\begin{Thm}\label{Thm.6.11}
The $K$-algebra $R(G,c)^U$ is equal to $K\otimes \mathcal{U}(\Xi)$, where $\mathcal{U}(\Xi)$ denotes 
the upper cluster algebra with initial seed $\Xi$. 
\end{Thm}

The proof of Theorem \ref{Thm.6.11} requires some preparation.
Because of Lemma \ref{lem: glueing formulas}, the cluster variables  of the seed $\Xi$ can be written as
\begin{equation}
    \label{eq:seed GU}
   F(x_{(i,r)})= \left\{ \begin{array}{ll}
      \De^{(r_i-1)}_{w_0(\varpi_i), \varpi_i}   &  \mbox{if} \quad 1 \leq r_i,\\[0.5em]
       \De^{(0)}_{c^{m_i+r_i-1}(\varpi_i), \varpi_i} & \mbox{if} \quad -m_i+1 \leq r_i \leq 0,\\[0.5em]
       \De^{(r_i-1+m_i)}_{\varpi_i, \varpi_i}  & \mbox{if} \quad r_i \leq -m_i.
   \end{array} 
\right.
\end{equation}
To any $(i,r) \in \mathcal{V}$, we attach the integer $s_{(i,r)}$ such that in the previous expression we have $F(x_{(i,r)})= \De^{(s_{(i,r)})}_{v(\varpi_i), w(\varpi_i)}$ for appropriate $v,w \in W.$
For $(M,N) \in \P_0$, let $\Xi_{M,N}$ be the seed obtained from $\Xi$ by erasing the vertices $(i,r)$ for which $s_{(i,r)} < N$ or $s_{(i,r)} > M$, along with the associated cluster variables.
Additionally, in the quiver of  $\Xi_{M,N}$, the vertices on the top and on the bottom of any column  are frozen. 
Let $\mathcal{U}(\Xi_{M,N})$ be the upper cluster algebra of the seed $\Xi_{M,N}$ with non-invertible frozen variables.

\begin{Lem}
    \label{lem : ximn max rank}
    Let $(M,N) \in \P_0$. The seed $\Xi_{M,N}$ is of maximal rank.
\end{Lem}

\begin{proof}
Let $B$ denote the exchange matrix of the seed $\Xi_{M,N}$. 
 Recall that the rows  of $B$ are indexed by the  vertices of the quiver of $\Xi_{M,N}$, while its columns are indexed by the mutable vertices.
Let $D$ be the square submatrix of $B$ whose rows are indexed by the vertices $(i,r)$ such that $(i,r-2)$ is a mutable vertex of the quiver of $\Xi_{M,N}$.
The columns of $D$ are indexed by all the mutable vertices, as those of $B$.
Choose a total order $\leq$ on the index set of the rows of $D$ such that, for any two rows $(i,r)$ and $(j,s)$, the following two conditions are satisfied.
\begin{enumerate}
    \item If $r_i > s_j$, then  $(i,r) < (j,s)$. 
    \item If $r_i = s_j$ and $\wt \xi_i > \wt \xi_j$, then  $(i,r) < (j,s)$. 
\end{enumerate}
Moreover, order the columns of $D$ so that the column $(i,r)$ is smaller than the column $(j,s)$ if and only if the row $(i,r+2)$ is smaller than the row $(j,s+2)$. 
Arranging rows and columns of the matrix $D$ according to these orders, we get a lower triangular matrix with diagonal entries equal to $-1$. 
Thus, the matrix $B$ (and hence the seed $\Xi_{M,N}$) is of maximal rank
    \cqfd
\end{proof}

Recall that by \cite{BFZ}, the coordinate ring $K[G^{w_0,e}]$ of the double Bruhat cell $G^{w_0,e}:= Bw_0B \cap B^-$ is an upper cluster algebra with invertible frozen variables 
\begin{equation}
    \De_{\varpi_i, \varpi_i}, \qquad \De_{w_0 \varpi_i, \varpi_i}, \qquad (i \in I).
\end{equation}
Notice that $G^{w_0,e}$ is isomorphic to the open subset of the variety $G/U$ given by the non-vanishing of the previous collection of generalised minors.
The next theorem can be deduced from the results of \cite{Fe} or \cite{Fr1}.

\begin{Thm}
\label{thm: cluster GU}
   The ring $K[G]^U$ has an upper cluster algebra structure given by the same initial seeds as $K[G^{w_0,e}]$, 
but in which the frozen variables are not invertible. 
\end{Thm}

\begin{proof}
   Let $\bi=(i_N,i_{N-1}, \dots ,i_1)$ be a reduced expression of $w_0$, and let $\ba$ be the sequence
    \[
 \ba := (a_j\mid j\in -[1,n]\cup [1,N]) := (-n,-n+1,\ldots,-1,-i_1,\ldots,-i_N).
\]
The sequence $\ba$ corresponds to a seed $t_\ba$ of the Berenstein-Fomin-Zelevinsky cluster structure of $K[G^{w_0,e}]$. 
Recall that the \textit{transpose map} $g \longmapsto g^T \ (g \in G)$ is an involutive anti-automorphism of $G$ defined by \cite[Eq.~2.1]{FZ}, which sends $U$ to $U^-$, and that  by \cite[Proposition 2.7]{FZ} satisfies the property
\begin{equation}
\label{eq:trans}
    \De_{u\varpi_i, v \varpi_i}(g)= \De_{v\varpi_i, u \varpi_i}(g^T), \qquad (i \in I, \  u,v \in W, \ g \in G).
\end{equation}
In \cite[Theorem 8.3.2]{Fr1}, a seed is attached to the reduced expression $\ii$ and it is proved that the ring $K[U^- \setminus G]$ is the upper cluster algebra with non-invertible frozen variables of this seed.
Therefore, by means of the isomorphism $U^- \setminus G \lra G/U$ induced by the transpose map, we can associate a seed $s_\bi$ to the reduced expression $\bi$ whose upper cluster algebra with non-invertible frozen variables is $K[G/U]=K[G]^U.$
Notice that, by  \cite[Eq.~38, Lemma 8.3.4]{Fr1} and Eq.(\ref{eq:trans}), the cluster variables of the seed $s_\ii$ are the elements
$$
\Delta_{s_{i_1}\cdots s_{i_k} (\varpi_{i_k}), \varpi_{i_k}}, \qquad \De_{\varpi_{j}, \varpi_j}, \qquad (1 \leq k \leq N, \ j \in I).
$$
Comparing with \cite[\S 2.3]{BFZ}, we see that these are exactly the cluster variables of the seed $t_\ba$.
Moreover, the quiver of $s_\bi$ is described by \cite[Equation~37]{Fr1} and \cite[Proposition 10.1.13]{Fr2}.
These formulas allow to verify that the quiver of $s_\ii$ is the same as the one $t_\ba$, up to a global change of the orientation of the arrows. 
\cqfd
\end{proof}

\medskip\noindent
\emph{Proof of Theorem \ref{Thm.6.11}.---}
Lemma \ref{lem : ximn max rank} and the same argument used in the proof of Corollary \ref{cor:AU R(G,c)} allow to show that
\begin{equation}
    \label{eq:upper limit}
    K \otimes \mathcal{U}(\Xi)= \bigcup_{N \in \Z_{\ge 0}} K \otimes \mathcal{U}(\Xi_{-N,N}).
\end{equation}
We start by proving that the upper cluster algebra $K \otimes \mathcal{U}(\Xi)$ is contained in $R(G,c)^U$.
Because of Eq.(\ref{eq:upper limit}), it is sufficient to prove that for any $N \in \Z_{\ge 0}$ we have that $K \otimes \mathcal{U}(\Xi_{-N,N})$ is contained in $R(G,c,-N,N)^U.$ 
This inclusion  will follow from the Starfish lemma.
Thus, we now verify that all the hypothesis needed to apply \cite[Corollary 6.4.6]{FWZ} are satisfied (recall also Remark \ref{rem:starfish lemma}).

First, notice that the algebraic group $U$ is connected and its group of characters is trivial. 
Therefore, since the ring $R(G,c,-N,N)$ is a unique factorization domain and it is a  $K$-algebra of finite type by Corollary \ref{cor: geometric prop of finite bands}, then \cite[Theorem 3.17]{VP} implies that $R(G,c,-N,N)^U$ is also a unique factorization domain.
Moreover, the algebra  $R(G,c,-N,N)^U$ is isomorphic to a polynomial ring in finitely many variables with coefficients in $K[G]^U$.
It is well known that the algebra $K[G]^U$ is of finite type.
Thus, $R(G,c,-N,N)^U$ is also of finite type.
Then, it is clear that $\bigl(R(G,c,-N,N)^U\bigr)^\times = R(G,c,-N,N)^\times = K^\times.$
Next, notice that the initial cluster variables and the ones obtained after a one step mutation belong to $R(G,c)^U$ because of Proposition \ref{Prop-6.2}. 
Moreover, it is clear that the initial cluster variables belong to $R(G,c,-N,N)^U$. 
Hence,  the cluster variables obtained after a one step mutation belong to $R(G, c)^U \cap K(G,c,-N,N)$, where we recall that $K(G,c,-N,N)$ denotes the field of fractions of the ring $R(G,c,-N,N)$.
But using Corollary \ref{cor: bands product} one can easily deduce that $R(G,c) \cap K(G,c,-N,N)= R(G,c,-N,N).$
Therefore, the cluster variables obtained after a one step mutation belong to $R(G,c,-N,N)^U.$
Then, let $x$ be either an initial cluster variable or a cluster variable obtained after a one step mutation. 
We want to prove that $x$ is an irreducible element of $R(G,c,-N,N)^U.$
For this, it is sufficient to show that $x$ is irreducible in  $R(G,c,-N,N)$.
Because of the proof of Proposition \ref{Prop-6.2}, we know that $x$ is a cluster variable of the cluster structure of $R(G,c)$.
Hence, there exists some $N' \geq N$ such that $x$ is a cluster variable of the cluster structure of $R(G,c,-N',N')$. 
Thus, \cite[Theorem 1.3]{GLS} implies that $x$ is irreducible in $R(G,c,-N',N')$. 
But since $R(G,c,-N,N)^\times = R(G,c,-N',N')^\times$ and $x$ belongs to $R(G,c,-N,N)$, it follows that $x$ is irreducible in $R(G,c,-N,N)$.

All the hypothesis needed to apply \cite[Corollary 6.4.6]{FWZ} have been verified. 
It follows that the upper cluster algebra $K \otimes \mathcal{U}(\Xi_{-N,N})$ is contained in $R(G,c,-N,N)^U$ and that $K \otimes \mathcal{U}(\Xi)$ is contained in $R(G,c)^U$.

Next, we claim that $\Xi_{0,0}$ 
is the pullback under the morphism $\pi_0$ of a seed of the Berenstein-Fomin-Zelevinsky cluster structure of $K[G^{w_0,e}]$ (up to the orientation of the arrows of the corresponding quivers).
Indeed, let  $\bi=(i_1,i_2, \dots, i_N)$ be a reduced expression of $w_0$ which is adapted to $c$ (in the sense of \S\ref{subsubsec:w_{i,k}}). Consider the sequence
  \[
 \ba = (a_j\mid j\in -[1,n]\cup [1,N]) := (-n,-n+1,\ldots,-1,-i_1,\ldots,-i_N),
\]
and the induced seed $t_\ba$ of the upper cluster algebra structure of $K[G^{w_0,e}]$.
Because of Lemma \ref{Lem2-4}, the cluster variables of this seed are:
$$
\{ \De_{c^k(\varpi_i), \varpi_i} \mid i \in I, \ 0 \leq k \leq m_i\}.
$$
Thus, Eq.(\ref{eq:seed GU}) implies that the pullback of the cluster variables of the seed $t_\ba$ are the cluster variables of the seed~$\Xi_{0,0}$.
Moreover, frozen variables pullback to frozen variables.
Finally, one can readily check that by reversing all the arrows of the iced-quiver of the seed $t_\ba$ we obtain the iced-quiver of the seed~$\Xi_{0,0}$.
Therefore, Theorem \ref{thm: cluster GU} implies that $K \otimes \mathcal{U}(\Xi_{0,0})=\pi_0^*(K[G]^U)$.

The last assertion together with Proposition \ref{Prop-6.5}, Corollary \ref{Cor-6.4} and Eq.~(\ref{eq:upper limit}) imply that $R(G,c)^U $ is contained in $K \otimes \mathcal{U}(\Xi)$.
\cqfd

The proof of Theorem \ref{Thm.6.11}, with only minor modifications, also establishes the following result.

\begin{Cor}
\label{cor: upper finite Uinv}
    Let $(M,N) \in \P_0$. The $K$-algebra $R(G,c,M,N)^U$ is equal to $K\otimes \mathcal{U}(\Xi_{M,N})$. 
\end{Cor}

\begin{remark}
{\rm
If $G$ is of type $A$, it is well known that $K[G]^U$ is equal to the genuine cluster algebra with the same seed.
So in that case Theorem~\ref{Thm.6.11} can be improved by replacing the upper cluster algebra with initial seed $\Xi$
by the ordinary cluster algebra with the same initial seed. In other words, in that case we have proved that $R(G,c)^U = F(\B)$, as stated in Theorem~\ref{Thm.3.4}.
The same remark holds for the statement of Corollary \ref{cor: upper finite Uinv}.
}
\end{remark}

\begin{remark}
{\rm
Proposition~\ref{Prop-6.5} states that $\theta^{(s)}_i$ is a scalar multiple of a cluster variable.
We will show in Proposition~\ref{Prop.8.2} that this scalar factor is in fact equal to 1.
More generally, Proposition~\ref{Prop.8.2} shows that the function $\theta^{(s)}_{i,k}$ defined in the next Section is a cluster variable
for every $k\ge 1$.
}
\end{remark}

\section{$G$-invariants}
\label{sec_G_inv}

In this section we study the subalgebra $R(G,c)^G$ of $G$-invariant regular functions on $B(G,c)$.

\subsection{Functional relations}

By Eq.(\ref{Eq.27}), we have $R(G,c)^G = K[\theta^{(s)}_i\mid i\in I,\ s\in\Z]$, where
the functions $\theta^{(s)}_i$ are defined at the level of points by
\[
 \theta^{(s)}_i((g(t))_{t\in\Z}) = \Delta_{\varpi_i,\varpi_i}(g(s)g(s+1)^{-1}),\qquad (i\in I,\ s\in\Z).
\]
More generally the functions $\theta^{(s)}_{i,\,k}$ defined by
\[
 \theta^{(s)}_{i,\,k}((g(t))_{t\in\Z}) = \Delta_{\varpi_i,\varpi_i}(g(s)g(s+k)^{-1}),\qquad (i\in I,\ s\in\Z,\ k\ge 1)
\]
obviously belong to $R(G,c)^G$.

\begin{Prop}\label{Prop.7.1}
The functions $\theta^{(s)}_{i,\,k}$ satisfy the system of functional relations
\[
\theta^{(s)}_{i,\,k}\theta^{(s+1)}_{i,\,k} =
\theta^{(s)}_{i,\,k+1}\theta^{(s+1)}_{i,k-1} + \prod_{j:\ c_{ij}=-1} \theta^{(s+a_{ij})}_{j,k},
\qquad
(i\in I,\ s\in \Z,\ k\ge 1),
\]
where $\theta^{(s)}_{i,\,0} = 1$, and the integers $a_{ij}$ were defined in Eq.~(\ref{Eq.25}). 
\end{Prop}

\begin{proof}
Let $b=(g(t))_{t\in\Z}\in B(G,c)$.
By definition of $B(G,c)$, for every $s\in\Z$ there exists $u_s\in U(c^{-1})$ such that $g(s) = u_s\overline{c}g(s+1)$. 
By definition of $U(c^{-1})$, we can write $u_s\overline{c} = \overline{c}u'_s$ where
$u'_s\in U^-$. It follows that
\begin{eqnarray*}
 \Delta_{c(\varpi_i),\varpi_i}(g(s)g(s+k)^{-1}) &=& 
 \Delta_{c(\varpi_i),\varpi_i}(\overline{c}u'_sg(s+1)g(s+k)^{-1}) \\
 &=&
 \Delta_{\varpi_i,\varpi_i}(\overline{c}^{-1}\overline{c}u'_sg(s+1)g(s+k)^{-1}) \\
 &=&
 \Delta_{\varpi_i,\varpi_i}(u'_sg(s+1)g(s+k)^{-1}) \\
 &=&
 \Delta_{\varpi_i,\varpi_i}(g(s+1)g(s+k)^{-1}).
\end{eqnarray*}
Here the second equality follows from the definition of the generalized minors, and the last 
one comes from the fact that $u'_s\in U^-$. 

Similarly, we have that
\begin{eqnarray*}
 \Delta_{\varpi_i,\varpi_i}(g(s)g(s+k-1)^{-1}) &=&
 \Delta_{\varpi_i,\varpi_i}(g(s)g(s+k)^{-1}\overline{c}^{-1}u_{s+k-1}^{-1}) \\
 &=&
 \Delta_{\varpi_i,\varpi_i}(g(s)g(s+k)^{-1}\overline{c}^{-1}) 
\end{eqnarray*}
where the second equality follows from the fact that $u_{s+k-1}^{-1}\in U$.
Using these two relations, and noting that
\[
\Delta_{c(\varpi_i),\varpi_i}(g(s)g(s+k)^{-1}) = 
\Delta_{c(\varpi_i),\varpi_i}(g(s)g(s+k)^{-1}\overline{c}^{-1}\overline{c}) =
\Delta_{c(\varpi_i),c(\varpi_i)}(g(s)g(s+k)^{-1}\overline{c}^{-1})
\]
we can rewrite:
\begin{eqnarray*}
\theta^{(s)}_{i,\,k}(b) &=& \Delta_{\varpi_i,\varpi_i}(h_{s,k}),\\
\theta^{(s+1)}_{i,\,k}(b) &=& \Delta_{c(\varpi_i),c(\varpi_i)}(h_{s,k}),\\
\theta^{(s)}_{i,\,k+1}(b) &=& \Delta_{\varpi_i,c(\varpi_i)}(h_{s,k}),\\
\theta^{(s+1)}_{i,\,k-1}(b) &=& \Delta_{c(\varpi_i),\varpi_i}(h_{s,k}),\\
\theta^{(s)}_{j,\,k}(b) &=& \Delta_{\varpi_j,\varpi_j}(h_{s,k}),\\
\theta^{(s+1)}_{j,\,k}(b) &=& \Delta_{c(\varpi_j),c(\varpi_j)}(h_{s,k}),
\end{eqnarray*}
where $h_{s,k}:=g(s)g(s+k+1)^{-1}\overline{c}^{-1}$.
Let us write a reduced decomposition of $c$ in the form $c=xs_iy$, where $x$ and $y$ are products of
Coxeter generators $s_j \not = s_i$. Then $x(\varpi_i) = \varpi_i$ and $xs_i(\varpi_i) = c(\varpi_i)$.
Moreover $x(\varpi_j) = c(\varpi_j)$ if $s_j$ is a factor of $x$, and otherwise $x(\varpi_j) = \varpi_j$.
That is, $x(\varpi_j) = c(\varpi_j)$ if $a_{ij}=1$ and $x(\varpi_j) = \varpi_j$ if $a_{ij}=0$.
Now, by \cite[Theorem 1.17]{FZ} (where we take $u=v=x$), we have the following  identity of elements of $K[G]$:
\[
 \Delta_{\varpi_i,\,\varpi_i}\Delta_{c(\varpi_i),\,c(\varpi_i)} =
 \Delta_{\varpi_i,\,c(\varpi_i)}\Delta_{c(\varpi_i),\,\varpi_i} + 
 \prod_{j:\ c_{ij}=-1,\ a_{ij} = 0 }\Delta_{\varpi_j,\,\varpi_j}
 \prod_{j:\ c_{ij}=-1,\ a_{ij} = 1 }\Delta_{c(\varpi_j),\,c(\varpi_j)}.
\]
Evaluating this identity on the element $h_{s,k}$ of $G$ yields the desired equation
\[
\theta^{(s)}_{i,\,k}(b)\theta^{(s+1)}_{i,\,k}(b) =
\theta^{(s)}_{i,\,k+1}(b)\theta^{(s+1)}_{i,k-1}(b) + \prod_{j:\ c_{ij}=-1} \theta^{(s+a_{ij})}_{j,k}(b),
\qquad (b\in B(G,c)).
\]
\cqfd
\end{proof}

\subsection{Cluster structure}

Proposition~\ref{Prop.7.1} will allow us to endow $R(G,c)^G$ with the structure of a cluster algebra.

To do this, we first introduce the following quiver $\Lambda$.
The vertex set of $\Lambda$ is $\WW = I\times \Z_{>0}$. 
There is an arrow between two vertices $(i,r)$ and $(j,s)$ if and only if one of the two following
conditions is satisfied:
\begin{itemize}
 \item[(i)] $i = j$ and $|r-s| = 1$, or
 \item[(ii)] $r=s$ and $c_{ij} = -1$.
\end{itemize}
The orientation of these arrows is fixed by the following rules: 
\begin{itemize}
 \item[(iii)] the vertical subquivers $\Lambda_i$ with vertex set $\{(i,s) \mid s>0\}$ are in sink-source orientation,
 \item[(iv)] the horizontal subquivers $\Lambda^{(s)}$ with vertex set $\{(i,s)\mid i\in I\}$ are in sink-source orientation,
 \item[(v)] if $c_{ij}=-1$ the square supported on vertices $(i,s)$, $(i,s+1)$, $(j,s)$, $(j,s+1)$, is an oriented $4$-cycle:
 \[
\mbox{either}\qquad 
  \begin{array}{ccc}
   (i,s)&\rightarrow &(j,s)\\
   \uparrow&&\downarrow\\
   (i,s+1)&\leftarrow &(j,s+1)
  \end{array}
\qquad\mbox{or}\qquad
  \begin{array}{ccc}
   (i,s)&\leftarrow &(j,s)\\
   \downarrow&&\uparrow\\
   (i,s+1)&\rightarrow &(j,s+1)
  \end{array}
  \]
\end{itemize}
This quiver $\Lambda$ is sometimes called the product of a sink-source Dynkin quiver of type $A_\infty$ by 
a sink-source Dynkin quiver of the same type as $G$. 

Strictly speaking, there are two quivers $\Lambda^+$
and $\Lambda^-$ satisfying the above rules. They are related to each other by a change of orientation of all arrows.
Let us pick one of the two, say $\Lam:=\Lambda^+$, and let us attach to every vertex $(i,s)$ of $\Lambda$ the function
$\theta^{(n(i,s))}_{i,s}\in R(G,c)^G$, where the integer $n(i,s)$ is uniquely determined by the following rules:
\begin{itemize}
 \item[(vi)] choose $i\in I$ and set $n(i,1)=0$,
 \item[(vii)] if a horizontal arrow $(j,s) \to (k,s)$ has the same orientation as the corresponding arrow between $j$ and $k$ in the quiver $Q$ associated with $c$ as in \S\ref{sec_cluster_bands}, then $n(k,s) = n(j,s)-1$, otherwise $n(k,s) = n(j,s)$,
 \item[(viii)] if a vertical arrow between $(j,s)$ and $(j,s+1)$ points towards $(j,s+1)$, then $n(j,s+1) = n(j,s)$,
 otherwise $n(j,s+1) = n(j,s) -1$.
\end{itemize}
We denote by $\Theta$ the quiver $\Lambda$ with its labelling by the functions $\theta^{(n(i,s))}_{i,s}$. 
Observe that, while the choice of $c$ plays a role in Rule (vii) of the labelling, the definition of $\Gamma$ is independent of the Coxeter element $c$.

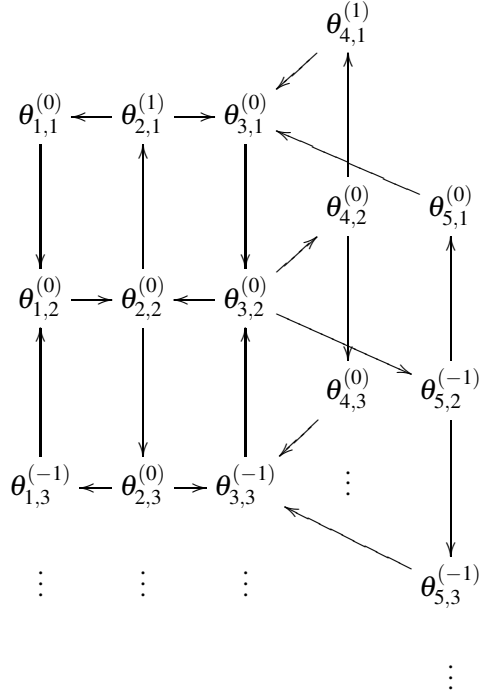
\begin{figure}[t]
\[
\def\objectstyle{\small}
\def\lablestyle{\small}
\xymatrix@-1.0pc{
 &&&\ar[ld]\theta^{(1)}_{4,1}
\\
\theta^{(0)}_{1,1}\ar[dd]& \ar[l]\theta^{(1)}_{2,1}\ar[r] &\theta^{(0)}_{3,1}\ar[dd]
\\
&&&\theta^{(0)}_{4,2}\ar[uu]\ar[dd]&\theta^{(0)}_{5,1}\ar[llu]
\\
\theta^{(0)}_{1,2}\ar[r]& \theta^{(0)}_{2,2}\ar[uu]\ar[dd] &\ar[l]\theta^{(0)}_{3,2}\ar[ur]\ar[rrd]
\\
&&&\ar[ld]\theta^{(0)}_{4,3}&\theta^{(-1)}_{5,2}\ar[uu]\ar[dd]
\\
\theta^{(-1)}_{1,3}\ar[uu]& \ar[l]\theta^{(0)}_{2,3}\ar[r] &\ar[uu]\theta^{(-1)}_{3,3}&\vdots
\\
\vdots&\vdots&\vdots&&\theta^{(-1)}_{5,3}\ar[llu]
\\
&&&&\vdots
}
\]
\caption{\label{Fig9} {\it The first 3 layers of the initial seed $\Theta$ in type $D_5$ for $c=s_2s_4s_1s_3s_5$.}}
\end{figure}
 
\begin{example}
{\rm
Let $G$ be of type $D_5$. Take $ c = s_2s_4s_1s_3s_5$, so that the corresponding quiver $Q$ is
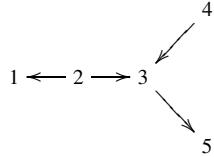
\[
\def\objectstyle{\scriptstyle}
\def\lablestyle{\scriptstyle}
\xymatrix@-0.8pc{
&&&\ar[ld]4
\\
1&\ar[l]2\ar[r]&3\ar[rd]
\\
&&&5
}
\] 
Pick $i=1$. The corresponding labelled quiver $\Theta$ is displayed in Figure~\ref{Fig9}.
}
\end{example}

\begin{Thm}\label{Thm.7.3}
The ring $R(G,c)^G$ of $G$-invariant functions on $B(G,c)$ has a cluster algebra structure with an initial seed given by $\Theta$. 
\end{Thm}

Before proving Theorem~\ref{Thm.7.3}, let us observe two crucial properties of the candidate initial seed~$\Theta$, which can easily be checked
by inspection.
The first property is that :
\begin{quote}
(A) \emph{Every first step mutation of $\Theta$ is an instance of the functional relations of Proposition~\ref{Prop.7.1}.}  
\end{quote}
More precisely, mutating at the vertex of $\theta^{(n(i,s))}_{i,s}$, we get the new cluster variable $\theta^{(n(i,s)+1)}_{i,s}$ 
(\resp $\theta^{(n(i,s)-1)}_{i,s}$) if $\theta^{(n(i,s))}_{i,s}$ is a source (\resp a sink)
of the vertical subquiver $\Lambda_i$.

Let $\WW_0$ (\resp $\WW_1$) denote the subset of $\WW$ consisting
of the vertices which are sources (\resp sinks) of the vertical subquivers $\Lambda_i\ (i\in I)$. Then we 
get a partition $\WW = \WW_0 \sqcup \WW_1$ such that no arrow connects two vertices of $\WW_0$
(\resp of $\WW_1$).
It follows that mutations at vertices of $\WW_0$ (\resp of $\WW_1$) pairwise commute. 
Let us consider the product $\tau_0$ (\resp $\tau_1$) of mutations at all the vertices of $\WW_0$ (\resp $\WW_1$).
These products are infinite, but in practice we will only need to look at their effect on a finite initial section of $\Theta$.
The second main feature is the translation invariance property~:
\begin{quote}
(B) \emph{The product $\tau_1\tau_0$ (\resp $\tau_0\tau_1$) replaces $\Theta$ by $\Theta^+$ (\resp $\Theta^-$)
whose quiver is identical to $\Lambda$, and in which every function $\theta^{(n(i,s))}_{i,s}$ is replaced by
$\theta^{(n(i,s)+1)}_{i,s}$ (\resp by $\theta^{(n(i,s)-1)}_{i,s}$).
}
\end{quote}

The strategy of the proof of Theorem~\ref{Thm.7.3} is similar to that of Theorem~\ref{Thm4.4}. 
We will first consider an increasing sequence of subalgebras of $R(G,c)^G$ whose union is equal to $R(G,c)^G$, 
and then show using (A) and (B) that each of them is a finite rank cluster algebra with initial seed a finite segment of $\Theta$.

More precisely, for $N>0$ we put 
\[
 R(G,c)^G_N := K[\theta_i^{(s)}\mid i\in I,\ n(i,N)\le s \le n(i,N)+N-1].
\]
Since for every $i \in I$ the sequence $(n(i,N))_{N>0}$ is non-increasing and satisfies 
\[
 -\left\lfloor\frac{N}{2}\right\rfloor \le n(i,N)- n(i,1) \le -\left\lfloor\frac{N}{2}\right\rfloor + 1,
\]
we see that 
\[
\lim_{N\to +\infty} n(i,N) = -\infty, \qquad\lim_{N\to +\infty} n(i,N) + N-1 = +\infty.
\]
Moreover 
\[
 R(G,c)^G_N \subset R(G,c)^G_{N+1},\qquad \mbox{and}\qquad \bigcup_{N>0} R(G,c)^G_N = R(G,c)^G.
\]

Let $\Lambda_N$ be the full sub-quiver of $\Lambda$ with vertex set
$\WW_N = I \times \{1,\ldots, N\}$, in which the $|I|$ lowest vertices $(i,N)\ (i\in I)$ are declared frozen.
Let $\Theta_N$ be the corresponding labelled sub-quiver of $\Theta$.
The proof of Theorem~\ref{Thm.7.3} will follow from
\begin{Prop}\label{Prop.7.4}
For every $N>0$, the algebra $R(G,c)^G_N$ is a cluster algebra with initial seed $\Theta_N$. 
\end{Prop}

\begin{proof}
Let us denote by $\tau_{0,N}$ (\resp $\tau_{1,N}$) the product of mutations  
$\tau_0$ (\resp $\tau_1$) restricted to the mutable vertices of $\Theta_N$. 
It follows from (B) that if we apply $\tau_{1,N-1}\tau_{0,N}$ (\resp $\tau_{0,N-1}\tau_{1,N}$) to $\Theta_N$  
we will replace the quiver $\Lambda_N$ by a new quiver $\Lambda^{(1)}_N$ such that the 
full subquiver of $\Lambda^{(1)}_N$ taken on the first $N-2$ rows is identical to $\Lambda_{N-2}$.
Moreover, every $\theta^{(n(i,s))}_{i,s}$ with $s\le N-2$ will be replaced by $\theta^{(n(i,s)+1)}_{i,s}$
(\resp $\theta^{(n(i,s)-1)}_{i,s}$). So the full labelled subquiver supported on the first $N-2$ rows
will be identical to $\Theta_{N-2}$, except that the upper index $n(i,s)$ of each variable 
has been shifted by $1$ (resp by $-1$). Therefore, 
we may then apply $\tau_{1,N-3}\tau_{0,N-2}$ (\resp $\tau_{0,N-3}\tau_{1,N-2}$)
to this labelled quiver, and get a new labelled subquiver whose restriction to the first $N-4$ rows
is identical to $\Theta_{N-4}$, except that the upper index $n(i,s)$ of each variable 
has been shifted by $2$ (resp by $-2$). Iterating this procedure, and collecting all the 
variables that we obtain at vertex $(i,1)$ after performing each of these products of 
mutations, we get exactly the list
\[
\theta_i^{(s)},\qquad (n(i,N)\le s \le n(i,N)+N-1). 
\]
This shows that the functions $\theta^{(n(i,s))}_{i,s}$ labelling the vertices of $\Theta_N$ generate the fraction
field of $R(G,c)^G_N \ $:
\[
K\left(\theta^{(s)}_i \mid i\in I,\ n(i,N)\le s \le n(i,N)+N-1\right).
\]
Moreover, the number of vertices $N|I|$ of $\Theta_N$ equals the number of variables of this 
field of rational functions, so the functions $\theta^{(n(i,s))}_{i,s}$ labelling the vertices of $\Theta_N$
form a free generating set of this field. In other words, $\Theta_N$ is a valid seed in the fraction field
of $R(G,c)^G_N$. 

Let $\AA_N$ be the cluster algebra with initial seed $\Theta_N$ and with coefficients extended to $K$,
regarded as a subalgebra of this fraction field. 
The discussion above shows that all generators of the polynomial ring $R(G,c)^G_N$ are cluster
variables of $\AA_N$. Hence $R(G,c)^G_N$ is contained in $\AA_N$. 

To prove the reverse inclusion, we can apply again the Starfish lemma.
First $R(G,c)^G_N$ is a polynomial ring, hence a unique factorization domain. 
Observation (A) shows that all variables obtained by a first step mutation of $\Theta_N$ belong to $R(G,c)^G_N$. 
Therefore we can conclude the proof in exactly the same way as the proof of Theorem~\ref{Thm4.10}. 
\cqfd
\end{proof}

\medskip\noindent
\emph{Proof of Theorem \ref{Thm.7.3}.---}
As already mentioned, we have 
\[
 R(G,c)^G_N \subset R(G,c)^G_{N+1},\qquad \mbox{and}\qquad \bigcup_{N>0} R(G,c)^G_N = R(G,c)^G.
\]
Moreover the initial seeds $\Theta_N$ form an increasing sequence of subseeds of $\Theta$ with limit $\Theta$, so
by Proposition~\ref{Prop.7.4} the limit algebra $R(G,c)^G$ will inherit the structure of a 
cluster algebra with initial seed $\Theta$. \cqfd

\begin{remark}
\label{rem: AU R(G,c)G}
\rm{
    As for Theorem \ref{Thm4.10}, our proof of Proposition \ref{Prop.7.4} also shows that the upper cluster algebra $\mathcal{U}_N$ with initial seed $\Theta_N$ equals the ordinary cluster algebra $\AA_N$. 
    Indeed, in the notation of the proof of Proposition \ref{Prop.7.4}, the Starfish lemma actually implies that $\mathcal{U}_N \subseteq R(G,c)^G_N$, from which the equality $\AA_N= \mathcal{U}_N$ follows.
    Moreover, the same argument used in the proof of Corollary \ref{cor:AU R(G,c)} can be used to show that the cluster algebra of the seed $\Theta$ is also equal to the upper cluster algebra of the same seed.
    }
\end{remark}

\section{Bands and quantum affine algebras: a dictionary}
\label{sec_band_qaf}

As explained in the introduction, the motivation of this work comes from the representation theory
of quantum affine algebras. In this final section, we will discuss the relations between our construction
and this theory.

\subsection{Quantum affine algebras and their categories of representations}

We refer to the survey lectures \cite{HL5} for an introduction to this subject and its connections with cluster algebras.

Let $\g = \mathrm{Lie}(G)$, and let $U_q(\hg)$ be the quantum enveloping algebra of the affine Lie algebra $\hg$
associated with $\g$. We assume that the quantum deformation parameter $q\in \C^*$ is not a root of unity.
The category $\CC$ of finite-dimensional $U_q(\hg)$-modules has been studied by many authors, see for example \cite{CP,FR,N1, K2}.

In \cite{HL}, a monoidal subcategory $\CC_\Z$ of $\CC$ was introduced, and it was conjectured that the Grothendieck ring 
$K_0(\CC_\Z)$ has the structure of a cluster algebra such that all cluster monomials are classes of simple objects
of $\CC_\Z$. This conjecture, called the monoidal categorification conjecture, is now proved \cite{KKOP2, Q}.

Let $U_q(\widehat{\mathfrak{b}})$ be the Borel subalgebra of $U_q(\hg)$. In \cite{HJ}, Hernandez and Jimbo have introduced 
a category $O$ of representations of $U_q(\widehat{\mathfrak{b}})$ containing all finite-dimensional representations and also some infinite-dimensional ones. Since every finite-dimensional $U_q(\hg)$-module remains irreducible by restriction to $U_q(\widehat{\mathfrak{b}})$, the Grothendieck ring of $\CC$ can be regarded as a subring of the Grothendieck ring of $O$.

In \cite{HL2}, two subcategories $O^+_\Z$ and $O^-_\Z$ of $O$ were introduced, 
both containing all restrictions of simple objects of $\CC_\Z$, 
and it was shown that $K_0(O^+_\Z)$ and $K_0(O^-_\Z)$ are isomorphic and have the same cluster algebra structure. Moreover a monoidal categorification conjecture was formulated, and it was shown that it would follow from a proof of the conjecture of \cite{HL}. Thus, based on \cite{KKOP2, Q}, this conjecture is also proved.

In 2019, Finkelberg and Tsimbaliuk introduced a new class of algebras $U_{q,\mu}(\hg)$ depending on a weight $\mu\in P$, called shifted quantum affine algebras \cite{FT}.
In \cite{H}, a category $\O_\mu$ of $U_{q,\mu}(\hg)$-modules was defined and studied, and it was shown that the direct
sum $\O^{\mathrm{shift}} := \oplus_{\mu \in P} \O_\mu$ is endowed with a fusion product, which yields a ring structure on its Grothendieck group. The category $\O^{\mathrm{shift}}$ contains the subcategory $\CC^{\mathrm{shift}}$ of
finite-dimensional representations. In the same work \cite{H}, Hernandez has introduced a subcategory $\CC^{\mathrm{shift}}_\Z$ of $\CC^{\mathrm{shift}}$ and shown that its Grothendieck ring is isomorphic to the Grothendieck ring of $O_\Z^+$ (as based rings), hence 
possesses the same cluster algebra structure. Because of the above result on $O^+_\Z$, the monoidal categorification conjecture also holds for $\CC_\Z^{\mathrm{shift}}$.  

Finally, in \cite{GHL} the subcategory $\O^{\mathrm{shift}}_\Z$ of $\O^{\mathrm{shift}}$ was introduced, and its Grothendieck group was shown to have a cluster structure. A monoidal categorification conjecture for $\O^{\mathrm{shift}}_\Z$ extending the one for $\CC^{\mathrm{shift}}_\Z$ was formulated, and it is still open (except for $\g = \mathfrak{sl}_2$).

As a result, we have a chain of Grothendieck rings with cluster structures:
\[
 K_0(\CC_\Z) \subset K_0(O^+_\Z) \simeq K_0(O^-_\Z) \simeq K_0(\CC^{\mathrm{shift}}_\Z) \subset K_0(\O^{\mathrm{shift}}_\Z).
\]

\subsection{$R(G,c)$ and the category $\O^{\mathrm{shift}}_\Z$}

As explained in the introduction, the initial motivation of this paper was to construct an algebro-geometric object
whose coordinate ring has the same cluster structure $\AA$ as the one of $K_0(\O^{\mathrm{shift}}_\Z)$ discovered in
\cite{GHL}. By Theorem~\ref{Thm2}, the affine scheme $B(G,c)$ is a solution to this problem for any choice of Coxeter element $c$.

Let us consider in more details the correspondence between the coordinate ring $R(G,c)$ and $K_0(\O^{\mathrm{shift}}_\Z)$.
By comparing the initial seed of $R(G,c)$ given by Theorem~\ref{Thm4.4} and the initial seed of $K_0(\O^{\mathrm{shift}}_\Z)$ of \cite[Theorem 7.4]{GHL},  we see that we have the following correspondence between regular functions in $R(G,c)$ and elements of $K_0(\O^{\mathrm{shift}}_\Z)$ :
\begin{equation}\label{Eq.35}
 \Delta^{(s)}_{c^k(\varpi_i),\ \widetilde{c}^{\,l}(\varpi_i)} 
 \quad\longleftrightarrow\quad 
 \underline{Q}_{\widetilde{c}^{\,l}(\varpi_i),\ q^{2(s+k)-\xi_i}}\,,\qquad
 (s\in\Z,\ i\in I,\ 0\le k,l \le m_i),
\end{equation}
where the integers $\xi_i$ were defined in \S\ref{sec: statement of main theorem}.
Note that here we have used the fact that $\xi_i+\widetilde{\xi}_i-2m_i$ is independent of $i$.
More generally, using cluster mutations on both sides, we can extend the correspondence to
\begin{equation}\label{Eq.35b}
 \Delta^{(s)}_{c^k(\varpi_i),\ w(\varpi_i)} 
 \quad\longleftrightarrow\quad 
 \underline{Q}_{w(\varpi_i),\ q^{2(s+k)-\xi_i}}\,,\qquad
 (s\in\Z,\ i\in I,\ 0\le k \le m_i,\ w\in W).
\end{equation}
Here $\underline{Q}_{w(\varpi_i), q^{2s-\xi_i}}$ is one of the (slightly renormalized) elements of the Grothendieck ring defined by Frenkel and Hernandez by means of their Weyl group action on $q$-characters
\cite{FH1,FH2}. It is proved in \cite{GHL} that $\underline{Q}_{\varpi_i,\ q^{2s-\xi_i}}$ 
(\resp $\underline{Q}_{w_0(\varpi_i),\ q^{2s-\xi_i}}$) is the class of a simple object, namely a 
positive (\resp negative) prefundamental representation. Moreover, it is conjectured that the remaining $Q$-variables are also classes of simple objects (this conjecture is proved in type $A_2$).

Concerning this Weyl group action, it was noted in \cite[\S7.3]{GHL} that one could define a variant of it, that is, an action of the finite covering $\overline{W}$ of $W$ which is compatible with the renormalized $Q$-variables $\underline{Q}$.
From the point of view of $(G,c)$-bands this modified action is very natural. Indeed, since $\overline{W}$ is a subgroup of
$G$, it acts on $B(G,c)$ by right translation. This induces a left action of $\overline{W}$ on $R(G,c)$ such that 
\[
\overline{w} \cdot \Delta^{(s)}_{u(\varpi_i),\,v(\varpi_i)} = \Delta^{(s)}_{u(\varpi_i),\,wv(\varpi_i)}  
\]
if $l(wv) = l(w)+l(v)$. Using the correspondence given by Eq.(\ref{Eq.35}), we see that this matches the formula of
\cite[Remark 7.3]{GHL}.

Frenkel and Hernandez proved that the above elements $\underline{Q}_{w(\varpi_i),\ q^{r}}$ of $K_0(\O^{\,\mathrm{shift}}_\Z)$ satisfy the following system of functional equations called (extended) $QQ$-system \cite[Theorem 5.6]{FH2}. Namely, 
for $(i,r)\in \VV$ and $w\in W$ such that $ws_i > w$, 
we have
\begin{equation}
{\bQ}_{ws_i(\varpi_i),\,q^{r}} \bQ_{w(\varpi_i),\,q^{r-2}} -\ {\bQ}_{ws_i(\varpi_i),\,q^{r-2}} \bQ_{w(\varpi_i),\,q^{r}}
=\prod_{j:\ c_{ij}=-1} \bQ_{w(\varpi_j),\,q^{r-1}}. 
\end{equation}
We refer to \cite{FH2} for the interesting history of the $QQ$-system and its connections with mathematical physics and opers. From the point of view of $(G,c)$-bands, the corresponding system
satisfied by the regular functions $\Delta^{(s)}_{c^k(\varpi_i),\ w(\varpi_i)}$ follows 
directly from the Fomin-Zelevinsky generalized minor identities \cite[Theorem 1.17]{FZ}.
The fact that the $QQ$-system results from relations between generalized minors 
was first discovered by Koroteev and Zeitlin in their paper on $(G,q)$-Wronskians \cite{KZ}.

Finally we should warn the reader that, according to \cite[Theorem 7.4]{GHL}, the image 
of $R(G,c)$ under the above correspondence becomes equal to the entire Grothendieck ring $K_0(\O^{\,\mathrm{shift}}_\Z)$ only after taking closure with respect to the topology of 
$K_0(\O^{\,\mathrm{shift}}_\Z)$. It is expected though, and proved in the case of $G= SL(2)$,
that the image of $R(G,c)$ coincides with $K_0(\O^{\,\mathrm{shift, f}}_\Z)$, 
where $\O^{\,\mathrm{shift, f}}_\Z$ is the full subcategory of finite length modules in $\O^{\,\mathrm{shift}}_\Z$.
Recently Hernandez and Zhang have shown that this category $\O^{\,\mathrm{shift, f}}_\Z$ is closed under fusion products,  and this allowed them to formulate a precise conjecture \cite[\S6.2]{HZ}.

\subsection{$R(G,c)^U$ and the categories $O^+_\Z$ and $\CC^{\,\mathrm{shift}}_\Z$}

Comparing Theorem~\ref{Thm.6.11}, \cite[Theorem 8.8]{H} and \cite[Theorem 4.2]{HL2}, we see that the 
cluster algebras of these three theorems are again defined by the same infinite quivers, and that their initial cluster variables are related by the following correspondence
\begin{equation}\label{Eq.36}
\Delta^{(s)}_{\varpi_i,\ \varpi_i} 
 \quad\longleftrightarrow\quad 
 \underline{Q}_{\varpi_i,\ q^{2s-\xi_i}}
 \quad\longleftrightarrow\quad  
 [-(2s-\xi_i)\varpi_i/2][L(\Psi_{i,\ q^{2s-\xi_i}})]
 \,,
 (s\in\Z,\ i\in I).
\end{equation}
Here $\underline{Q}_{\varpi_i,\ q^{2s-\xi_i}}$ stands for the class of a (renormalized) positive prefundamental representation
in $\O^{\mathrm{\,shift}}_\Z$, which is one-dimensional, so can also be regarded as an object of
$\CC^{\mathrm{\,shift}}_\Z$. On the other hand, $L(\Psi_{i,\ q^{2s-\xi_i}})$ is a positive prefundamental
representation in $O^+_\Z$, which is infinite-dimensional, and $[-(2s-\xi_i)\varpi_i/2]$ is the normalization factor (compare \cite[\S7.2]{GHL} and \cite[Theorem 4.2]{HL2}).

Thus, the invariant subalgebra $R(G,c)^U$ corresponds in a natural way to the Grothendieck rings 
of the categories $O^+_\Z$ and $\CC^{\,\mathrm{shift}}_\Z$.
Note however that, in contrast to Theorem~\ref{Thm.6.11} which involves the upper cluster algebra 
$\mathcal{U}(\Xi)$, 
the Grothendieck rings of \cite[Theorem 4.2]{HL2} and \cite[Theorem 8.8]{H} are genuine cluster algebras. 

One may ask what happens if we replace $U$ by $U^-$. From the point of view of $(G,c)$-bands
it is obvious that $R(G,c)^U$ and $R(G,c)^{U^-}$ are isomorphic. But the natural coordinate system in 
$R(G,c)^{U^-}$ is given by the regular functions $\Delta^{(s)}_{\varpi_i,\  w_0(\varpi_i)}$, which correspond under Eq.(\ref{Eq.35b}) to classes of \emph{negative} prefundamental representations. 
So clearly, $R(G,c)^{U^-}$ corresponds to the Grothendieck ring of the subcategory $O^-_\Z$ of 
the Hernandez-Jimbo category $O$. In terms of representations of shifted quantum affine algebras, what we get is a subcategory of $\O^{\,\mathrm{shift}}_\Z$ having a Grothendieck ring isomorphic to that of $\CC^{\,\mathrm{shift}}_\Z$, but which contains infinite-dimensional objects.

\subsection{$R(G,c)^G$ and the category $\CC_\Z$}
\label{subsec-8.4}

The functional relations of Proposition~\ref{Prop.7.1} are well known and have a long history. 
They appeared in mathematical physics in the 90's under the name of $T$-system \cite{KNS1} (see \cite{KNS2}
for a thorough survey). They are also deeply connected with the representation theory of quantum affine 
algebras.

In order to relate $R(G,c)^G$ with $K_0(\CC_\Z)$, let us recall some basic results. 
It is known that $K_0(\CC_\Z)$ is the polynomial ring in the classes  of the fundamental
$U_q(\hg)$-modules: 
\[
L(Y_{i,q^r}),\qquad ((i,r+1)\in\VV).
\]
The fundamental modules belong to the larger family of Kirillov-Reshetikhin modules:
\[
 W^{(i)}_{k,q^r} := L(Y_{i,q^r}Y_{i,q^{r+2}}\cdots Y_{i,q^{r+2(k-1)}}), \qquad ((i,r+1)\in \VV,\ k>0).
\]
To express the classes $[W^{(i)}_{k,q^r}]\in K_0(\CC_\Z)$ as polynomials in the $[L(Y_{i,q^r})]$, one uses 
the fact that they satisfy the system of equations
\[
 [W^{(i)}_{k,q^r}][W^{(i)}_{k,q^{r+2}}] = [W^{(i)}_{k+1,q^r}][W^{(i)}_{k-1,q^{r+2}}]
 \ + \prod_{j:\ c_{ij}=-1} [W^{(j)}_{k,q^{r+1}}],
\]
first conjectured by Kuniba, Nakanishi and Suzuki \cite{KNS1}, and then proved by Nakajima \cite{N}.

The following result then follows:

\begin{Prop}\label{Prop.8.1}
Let $\iota : R(G,c)^G = K[\theta^{(s)}_{i,1}\mid i\in I,\ s\in \Z] \to 
K\otimes K_0(\CC_\Z) = K[L(Y_{i,q^r})\mid (i,r+1)\in \VV]$ 
be the $K$-algebra isomorphism defined by 
\[
 \iota\left(\theta^{(s)}_{i,1}\right) = \left[L(Y_{i,\,q^{2s+1-\xi_i}})\right],\qquad (i\in I,\ s\in \Z). 
\]
Then for every $i\in I$, $s\in \Z$ and $k>0$ we have
\begin{equation}\label{Eq.34}
 \iota\left(\theta^{(s)}_{i,k}\right) = \left[W^{(i)}_{k,\,q^{2s+1-\xi_i}}\right].
\end{equation}
\end{Prop}

\begin{proof}
The functional relations of Proposition~\ref{Prop.7.1} allow to express $\theta^{(s)}_{i,k+1}$ 
in terms of $\theta^{(t)}_{j,l}$ with $l<k+1$.
Thus by induction they allow to calculate the expression
of every $\theta^{(s)}_{i,k}$ in terms of the generators $\theta^{(t)}_{j,1}$. 
Similarly, the $T$-system relations satisfied by the classes of the Kirillov-Reshetikhin modules allow
to express every $[W^{(i)}_{k,q^{r}}]$ in terms of the $[L(Y_{j,q^{t}})]$. 
Therefore to prove the proposition it is enough to show that the assignment of Eq.~(\ref{Eq.34}) transforms the relations
of Proposition~\ref{Prop.7.1} into the $T$-system relations.

Suppose that $i,j\in I$ are such that $c_{ij}=-1$.
Then $a_{ij} = 1$ (\resp $a_{ij} = 0$) if and only if $s_j$ precedes $s_i$ 
(\resp $s_i$ precedes $s_j$) in a reduced decomposition of $c$.
That is, if and only if $\xi_j = \xi_i + 1$
(\resp $\xi_i = \xi_j + 1$).
It follows by induction on $k$ that the assignment of Eq.~(\ref{Eq.34}) transforms the relations
of Proposition~\ref{Prop.7.1} into
\[
\left[W^{(i)}_{k,q^{2s+1-\xi_i}}\right]
\left[W^{(i)}_{k,q^{2s+3-\xi_i}}\right] 
= 
\left[W^{(i)}_{k+1,q^{2s+1-\xi_i}}\right]
\left[W^{(i)}_{k-1,q^{2s+3-\xi_i}}\right]
 \ + \prod_{j:\ c_{ij}=-1} \left[W^{(j)}_{k,q^{2s+2-\xi_j}}\right] 
\]
as required.
\cqfd
\end{proof}

Recall the initial seed $\Theta = (\Lambda, \{\theta_{i,s}^{n(i,s)}\mid (i,s)\in \WW\})$ of the cluster algebra structure on $R(G,c)^G$.
Let $\iota(\Theta) := (\Lambda, \{\iota(\theta_{i,s}^{n(i,s)})\mid (i,s)\in \WW\})$ denote the seed of $K_0(\CC_\Z)$ obtained by applying the isomorphism $\iota$.
It is known that $K_0(\CC_\Z)$ has a cluster structure with initial seed $\iota(\Theta)$
\cite{KKOP2}. In other words, $\iota$ is a cluster algebra isomorphism. 

We will now use this remark to show that
\begin{Prop}\label{Prop.8.2}
For every $i\in I$, $s\in\Z$, $k\in \Z_{> 0}$, the function $\theta^{(s)}_{i,k}$ is a cluster
variable of $R(G,c)^U$.
\end{Prop}

\begin{proof}
By Proposition~\ref{Prop-6.5}, the function $\theta^{(s)}_{i,1} = \theta^{(s)}_{i}$ is a scalar
multiple of a cluster variable of $R(G,c)^U$. We first show that this scalar multiple is equal to 1.

For this we note that $K_0(\CC_\Z)$ can be regarded as a subalgebra of 
$K_0(\CC^{\,\mathrm{shift}}_\Z) \equiv K_0(O^+_\Z)$, and that in this embedding, the classes $[L(Y_{i,q^{2s+1-\xi_i}})]\ (s\in \Z)$ become cluster variables
of $K_0(O^+_\Z)$. In fact, as shown in the proof of \cite[Proposition 6.1]{HL2}, one can use the algorithm of \cite{HL1} to calculate the cluster expansion
of $[L(Y_{i,q^{2s+1-\xi_i}})]$ with respect to the initial cluster consisting of positive prefundamental representations of $K_0(O^+_\Z)$. This algorithm provides an explicit sequence of mutations from this
initial cluster to the class of a given fundamental module. It turns out that this sequence of 
mutations is the same as the one used in the proof of Lemma~\ref{Lem.6.9}. This implies that the 
cluster variable $x_i$ of bi-degree $(\varpi_i - c\varpi_i, 0)$ obtained in Lemma~\ref{Lem.6.9} is nothing else but the image of $[L(Y_{i,q^{1-\xi_i}})]$ under the unique ring homomorphism 
$\zeta : K_0(O^+_\Z) \to R(G,c)^U$ extending the correspondence of Eq.~(\ref{Eq.36}).

It is well known that the $q$-character of a fundamental module $L(Y_{i,q^r})$ is of the form
\begin{equation}\label{Eq.38}
 \chi_q(L(Y_{i,q^r})) = Y_{i,q^r} + P_{i,r},
\end{equation}
where $P_{i,r}$ is a Laurent polynomial in the variables $Y_{j,q^s}$ with $j\in I$ and $s>r$.
It then follows from the generalized Baxter $TQ$-relations of Frenkel and Hernandez 
\cite[Theorem 4.8]{FH0} that
to obtain the cluster expansion of $[L(Y_{i,q^r})]$ with respect to the cluster of positive prefundamental representations, we only have to  
perform in Eq.~(\ref{Eq.38}) the formal substitution 
\[
Y_{j,q^s} \quad\longrightarrow\quad 
[\varpi_j]\frac{[L(\Psi_{j,\ q^{s-1}})]}{[L(\Psi_{j,\ q^{s+1}})]} 
=
\frac{[-(s-1)\varpi_j/2][L(\Psi_{j,\ q^{s-1}})]}{[-(s+1)\varpi_j/2][L(\Psi_{j,\ q^{s+1}})]}.
\]
It follows that the cluster expansion of $x_i$ is of the form
\[
x_i = \frac{\Delta^{(0)}_{\varpi_i,\,\varpi_i}}{\Delta^{(1)}_{\varpi_i,\,\varpi_i}} 
+ P_i,
\]
where $P_i$ is a Laurent polynomial in the variables 
\begin{equation}\label{eq-45}
\frac{\Delta^{(s)}_{\varpi_j,\,\varpi_j}}{\Delta^{(s+1)}_{\varpi_j,\,\varpi_j}},
\qquad (j\in I,\ s\ge 0). 
\end{equation}
Moreover, if in (\ref{eq-45}) we have $j = i$, then the stronger inequality $s > 0$ holds. 

By Lemma~\ref{Lem.6.9}, there exists a nonzero scalar $\gamma_i$ such that
$\theta^{(0)}_i = \gamma_i\, x_i$.
Let $b = ((g(t))_{t\in\Z})$ be a $(G,c)$-band. Evaluating functions at $b$ gives
\[
\Delta_{\varpi_i,\varpi_i}(g(0)g(1)^{-1}) =
\gamma_i\left(\frac{\Delta_{\varpi_i,\,\varpi_i}(g(0))}{\Delta_{\varpi_i,\,\varpi_i}(g(1))}
+ \widetilde{P}_i
\right)
\]
where $\widetilde{P}_i$ is a Laurent polynomial  in the 
$\Delta_{\varpi_j,\varpi_j}(g(s))\ (j\not = i,\ s\ge 0)$ and $\Delta_{\varpi_i,\varpi_i}(g(s))\ (s> 0)$.
If we choose, as we may, the $(G,c)$-band $b$ so that $g(1) = e$, the unit element of $G$, we get
\[
\Delta_{\varpi_i,\varpi_i}(g(0)) =
\gamma_i\left(\Delta_{\varpi_i,\,\varpi_i}(g(0))
+ \widetilde{P}_i
\right).
\]
Let $b'= ((g'(t))_{t\in\Z})$ be another band such that $g'(t) = g(t)$ for $t\ge 1$,
$\De_{\varpi_j,\varpi_j}(g'(0)) = \De_{\varpi_j,\varpi_j}(g(0))$ for $j\not = i$, but 
$\De_{\varpi_i,\varpi_i}(g'(0)) \not =  \De_{\varpi_i,\varpi_i}(g(0))$.
Such a band always exists because the functions $\De_{\varpi_k,\varpi_k}(g(0))=\theta^{(0)}_k(b)$ are the $|I|$ 
canonical coordinates on the affine space $U(c^{-1})\bar{c}$ to which $g(0) = g(0)g(1)^{-1}$ belongs.
Then comparing evaluations at $b$ and at $b'$ we get 
\[
\Delta_{\varpi_i,\varpi_i}(g(0)) - \Delta_{\varpi_i,\varpi_i}(g'(0))=
\gamma_i(\Delta_{\varpi_i,\,\varpi_i}(g(0))
- \Delta_{\varpi_i,\,\varpi_i}(g'(0))). 
\]
Since $\Delta_{\varpi_i,\varpi_i}(g(0)) \not = \Delta_{\varpi_i,\varpi_i}(g'(0))$, this forces $\gamma_i = 1$, as required.
Thus, we have proved the proposition for $k=1$ and $s=0$. 
By translation invariance, the statement  also holds for $k=1$ and every $s\in\Z$.

To prove it in general, we note that for the same reasons as above, 
the classes of all Kirillov-Reshetikhin modules $[W^{(i)}_{k,q^{2s+1-\xi_i}}]\ (s\in \Z)$ are cluster variables of $K_0(O^+_\Z)$.
Since these classes satisfy the same functional relations as the $\theta^{(s)}_{i,k}$, and 
since the solutions of this system of functional relations are completely determined by their
values for $k=1$, we get that $\zeta([W^{(i)}_{k,q^{2s+1-\xi_i}}]) = \theta^{(s)}_{i,k}$
for every $k\ge 1$ and $s \in \Z$.
Clearly, the above definition of $\zeta$ implies that it maps cluster variables
to cluster variables. This completes the proof.
\cqfd
\end{proof}

\begin{remark}
{\rm
The proof of Proposition~\ref{Prop.8.2} shows that the restriction of $\zeta : K\otimes K_0(O^+_\Z) \to R(G,c)^U$ to $K\otimes K_0(\CC_\Z)$
coincides with the isomorphism $\iota^{-1} : K\otimes K_0(\CC_\Z) \to R(G,c)^G$ of Proposition~\ref{Prop.8.1}.
}
\end{remark}

\begin{remark}\label{Rem.8.4}
{\rm
As explained in the proof of Proposition~\ref{Prop.8.2}, the Baxter $TQ$-relations of Frenkel and Hernandez translate as follows in terms of $(G,c)$-bands. Let $x\in R(G,c)^G$ be an element corresponding to the class of a finite-dimensional module $V$ of $\CC_\Z$, for example a cluster variable or a cluster monomial of $R(G,c)^G$. Then, since $R(G,c)^G\subset R(G,c)^U$, we can write $x$ as a Laurent polynomial in the cluster variables $\Delta^{(s)}_{\varpi_i,\varpi_i}$ of the initial cluster of $R(G,c)^U$. The Baxter $TQ$-relations amount to say that this cluster expansion of $x$ is in fact a Laurent polynomial in the variables  
\[
\frac{\Delta^{(s)}_{\varpi_j,\,\varpi_j}}{\Delta^{(s+1)}_{\varpi_j,\,\varpi_j}},
\qquad (j\in I,\ s \in \Z),
\]
and that replacing in this expansion each of the above variables by $Y_{j,\,q^{2s+1-\xi_j}}$, we get exactly the $q$-character of $V$. For example, calculating the cluster expansion of $\theta^{(0)}_i$ is equivalent to calculating the $q$-character of $L(Y_{i,q^{1-\xi}})$, a non trivial computation in general as we saw in Proposition~\ref{Prop-6.5}.  
}
\end{remark}

\subsection{A dictionary}
We summarize the main items of the correspondence between $(G,c)$-bands and 
representations of (shifted) quantum affine algebras in Table~\ref{table1}.
\begin{table}[t]%
\begin{center}
\begin{tabular}{c|c}
Bands & Quantum affine algebras \\[2mm]
\hline
\\
$R(G,c)$ & $K_0(\O^{\mathrm{shift}}_\Z)$  \\[2mm]
$\Delta^{(s)}_{c^k(\varpi_i),\ w(\varpi_i)}$ & $\underline{Q}_{w(\varpi_i),\ q^{2(s+k)-\xi_i}}$ \\[3 mm]
\hline
\\
$R(G,c)^U$ & $K_0(O^+_\Z) \equiv K_0(\CC^{\mathrm{shift}}_\Z)$ \\[2mm]
$\Delta^{(s)}_{\varpi_i,\ \varpi_i}$ & $[\frac{(\xi_i-2s)\varpi_i)}{2}]\left[L(\Psi_{i,\ q^{2s-\xi_i}})\right]$\\[3mm]
\hline
\\
$R(G,c)^G$ & $K_0(\CC_\Z)$ \\[2mm]
$\theta^{(s)}_{i,k}$ & $\left[W^{(i)}_{k,\,q^{2s+1-\xi_i}}\right]$
\end{tabular}
\end{center}
\bigskip
\caption{Dictionary}
\label{table1}
\end{table}

Note that the weight space decomposition of $R(G,c)$ under the action of the torus $T$
\[
 R(G,c) = \bigoplus_{\mu\in P} R(G,c)_\mu 
\]
matches the natural decomposition
\[
K_0(\O^{\mathrm{shift}}_\Z) = \bigoplus_{\mu\in P}  K_0(\O_{\mu,\Z})
\]
coming from the definition $\O^{\mathrm{shift}}_\Z = \bigoplus_{\mu\in P} \O_{\mu,\Z}$.
Similarly the weight space decomposition
\[
 R(G,c)^U = \bigoplus_{\lambda\in P^+} R(G,c)^U_\lambda
\]
reflects the decomposition
\[
K_0(\CC^{\mathrm{shift}}_\Z) = \bigoplus_{\lambda\in P^+}  K_0(\CC_{\lambda,\Z})
\]
(the category $\O_\mu$ contains finite-dimensional modules if and only if $\mu$ is a dominant weight \cite{H}).

\medskip
We hope that this dictionary will stimulate progress on both sides. 
Here are two possible future directions.

The Grothendieck ring $K_0(\O^{\mathrm{shift}}_\Z)$ has a canonical (topological) basis given by the classes of simple objects. It is conjectured in \cite{GHL} that this basis contains all cluster monomials. Transferring it on the side of $(G,c)$-bands, we get a 
canonical basis of $R(G,c)$. Recall that for every $s\in\Z$, we have a subring $R(G,c,s,s)=\pi_s^*(K[G]) \simeq K[G]$. We expect that this canonical basis contains the pullback of Kashiwara's upper global
basis of $K[G]$ under $\pi_s$. Comparing \cite{K} with \cite[\S9.4]{GHL}, it is easy to see that this is true for $G=SL(2)$.

In the other direction, we have seen that $R(G,c)$ has the natural structure of a $G$-module.
Recalling that $R(G,c)$ is isomorphic to the polynomial ring $\pi_0^*(K[G])[\theta^{(s)}_i \mid i \in I, \ s \in \Z]$, where the 
variables $\theta^{(s)}_i$ are $G$-invariant and where the action on $\pi_0^*(K[G]) \simeq K[G]$ comes from the action of $G$ on itself by right translations, we see that $R(G,c)$ decomposes as a sum of finite-dimensional $G$-modules. Transferring this action of $G$ on the side of quantum affine algebras, we get an action 
of $G$ on $K\otimes K_0(\O^{\mathrm{shift}}_\Z)$, and a similar decomposition into 
finite-dimensional $G$-modules. We believe that this could be a useful tool to understand 
which simple objects of $\O^{\mu}_\Z$ descend to 
a given \emph{truncated} shifted quantum affine algebra (see \cite[\S11]{H}). 
Truncated shifted quantum affine algebras
are a family of quotient algebras $U^Z_{q,\mu}(\hg)$ of $U_{q,\mu}(\hg)$ parametrized by tuples of truncation parameters~$Z$, introduced by Finkelberg and Tsimbaliuk \cite{FT} and conjectured to be isomorphic to certain $K$-theoretic Coulomb branches. Every set of parameters $Z$ determines a unique dominant weight~$\lambda$, and by analogy with the theory of (truncated) shifted Yangians
(see in particular \cite{KT1,KT2}), one expects
that the number of simple representations of $U^Z_{q,\mu}(\hg)$ is equal to the multiplicity of the weight $\mu$ in a certain $G$-module with highest weight $\lambda$.

We were recently informed by Frenkel and Hernandez that they have
constructed \cite{FH4} an action of the Lie algebra $\g$ of $G$ on the Grothendieck
ring of the category $O$ of $U_q(\widehat{\mathfrak{b}})$ by
means of derivations coming from the screening operators. They also
expect applications of this action to the problem of descent to
truncated algebras.
We are very grateful to them for sharing these ideas with us.


\bigskip
\small
\noindent
\begin{tabular}{ll}
Luca {\sc Francone} & Università degli studi di Roma Tor Vergata,\\
&Dipartimento di Matematica,\\
&Via della Ricerca Scientifica 1, 00133 Roma \\
& email : {\tt francone@mat.uniroma2.it}
\\[5mm]
Bernard {\sc Leclerc}  & Universit\'e de Caen Normandie,\\
&CNRS UMR 6139 LMNO,\\ &14032 Caen, France\\
&email : {\tt bernard.leclerc@unicaen.fr}
\end{tabular}

\end{document}